\documentclass[10pt]{amsart}
\usepackage{amsmath, amscd, amsthm} 
\usepackage{amssymb, amsfonts} 
\usepackage[all]{xy} 
\newcommand{\C}{\mathbb{C}} 
\newcommand{\A}{\mathbb{A}}
\renewcommand{\P}{\mathbb{P}}
\newcommand{\Z}{\mathbb{Z}}
\newcommand{\R}{\mathbb{R}}
 
\renewcommand{\S}{\mathbb{S}}

\newcommand{\GG}{\mathcal{G}}
\newcommand{\FF}{\mcal{F}}
\renewcommand{\AA}{\mathcal{A}}
\newcommand{\KK}{\mathcal{K}}
\newcommand{\mf}[1]{\mathfrak{#1}}

\newcommand{\iso}{\cong}
\newcommand{\wkeq}{\simeq}
\newcommand{\onto}{\twoheadrightarrow}
\newcommand{\sst}{\mathrm{sst}}
\newcommand{\semi}{\mathrm{semi}}
\newcommand{\dtop}{\Delta^{\bullet}_{top}}
\newcommand{\dk}{\Delta^{\bullet}_{k}}
\newcommand{\dC}{\Delta^{\bullet}_{\C}}
\newcommand{\bu}{\mathbf{bu}}
\newcommand{\Sch}{\mathrm{Sch}}
\newcommand{\Sm}{\mathrm{Sm}}

\newcommand{\Top}{\mathcal{T}}
\newcommand{\GTop}{G\mathcal{T}}
\newcommand{\Set}{\mathrm{Set}}
\newcommand{\sSet}{\mathrm{sSet}}
\newcommand{\Sym}{\mathrm{Sym}}

\newcommand{\Grass}{\mathrm{Grass}}
\newcommand{\orth}{\mathrm{orth}}

\newcommand{\mcal}[1]{\mathcal{#1}}

\DeclareMathOperator*{\colim}{\mathrm{colim}}
\DeclareMathOperator*{\holim}{\mathrm{holim}}
\DeclareMathOperator*{\hocolim}{\mathrm{hocolim}}

\DeclareMathOperator{\sing}{\mathrm{Sing}}

\DeclareMathOperator{\spec}{\mathrm{Spec}}
\DeclareMathOperator{\im}{\mathrm{im}}
\DeclareMathOperator{\coker}{\mathrm{coker}}
\DeclareMathOperator{\Hom}{Hom}
\DeclareMathOperator{\Map}{Map}
\DeclareMathOperator{\Mor}{Mor}
\DeclareMathOperator{\End}{End}

\DeclareMathOperator{\Rep}{Rep}

\numberwithin{equation}{section} %Fiddles with numbering system of the following.

\theoremstyle{plain}
\newtheorem{theorem}[equation]{Theorem}
\newtheorem*{theorem*}{Theorem}
\newtheorem{proposition}[equation]{Proposition}
\newtheorem{lemma}[equation]{Lemma}
\newtheorem{corollary}[equation]{Corollary}

\theoremstyle{definition}

\newtheorem{definition}[equation]{Definition}

\newtheorem{remark}[equation]{Remark}

\makeatletter
\let\@wraptoccontribs\wraptoccontribs
\makeatother

\begin{document}
\title[Equivariant semi-topological $K$-homology]{Equivariant semi-topological $K$-homology
and a theorem of Thomason}

\author{Jeremiah Heller \and Jens Hornbostel}
\address{Bergische Universit\"at Wuppertal, Gau{\ss}str. 20, D-42119 Wuppertal, Germany}
\email{heller@math.uni-wuppertal.de}
\email{hornbostel@math.uni-wuppertal.de}

\begin{abstract}
We generalize several comparison results between algebraic,
semi-topological and topological $K$-theories to the
equivariant case with respect to a finite group.
\end{abstract}

\maketitle
\tableofcontents
\section{Introduction}

In his landmark article \cite{Thomason:famous},
Thomason establishes an \'etale Atiyah-Hirzebruch type spectral sequence relating \'etale cohomology and Bott-inverted algebraic $K$-theory with finite coefficients. 
When restricted to smooth complex varieties his results say, amongst other things, that there is an
isomorphism
$$
K_{*}^{alg}(X;\Z/n)[\beta^{-1}]\iso KU^{-*}(X^{an};\Z/n)
$$ 
between Bott-inverted algebraic $K$-theory with finite coefficients  and topological $K$-theory with finite coefficients.

In the last decade,
Friedlander and Walker (see e.g., \cite{FW:compK},
\cite{FW:ratisos}) have refined the comparison map $\mcal{K}^{alg}(X)\to \mcal{K}^{top}(X^{an})$
between the algebraic and the topological $K$-theory of complex varieties by introducing
an intermediate theory $\KK^{sst}(X)$, called
{\it semi-topological $K$-theory}. 
One has natural morphisms of spectra
$$\KK^{alg}(X) \to \KK^{sst}(X) \to \mcal{K}^{top}(X^{an})$$
where the left hand one induces a weak equivalence 
$$
\KK^{alg}(X;\Z/n) \xrightarrow{\wkeq} \KK^{sst}(X;\Z/n)
$$
for smooth quasi-projective
$X$.  Upon inverting (the unique lift of) the topological Bott element $\beta$, the right hand map induces a weak equivalence 
$$
\KK^{sst}(X)[\beta^{-1}] \xrightarrow{\wkeq} \mcal{K}^{top}(X^{an})
$$
for smooth quasi-projective $X$. 

\medskip
 
There have been several proofs of the result that Bott-inverted semi-topological $K$-theory agrees with topological $K$-theory; the first ones relying on Thomason's result itself.
In \cite{Walker:Thomason} Walker introduces a {\it bivariant} semi-topological $K$-theory for quasi-projective complex varieties. One of the main results of that article is that the semi-topological $K$-homology of a smooth quasi-projective complex variety is isomorphic to the topological $K$-homology of its underlying complex manifold $X^{an}$. Using this result, Walker gives a new proof, in the case of smooth projective complex varieties, that Bott-inverted semi-topological $K$-theory agrees with complex $K$-theory. His proof does not rely on Thomason's theorem and thus specializes to  give a particularly elegant alternate proof  of Thomason's celebrated theorem comparing algebraic and complex $K$-theory with finite coefficients, in case of smooth projective complex varieties. 

\medskip

In the present article, we generalize Walker's results
mentioned above to the equivariant setting with respect to an arbitrary 
finite group $G$. We begin by introducing a bivariant equivariant semi-topological $K$-theory $\mcal{K}_{G}^{sst}(X,Y)$ for  quasi-projective $G$-varieties $X$ and $Y$. To construct and study this bivariant theory, which is constructed as a $G$-spectrum, we rely on the machinery of equivariant $\Gamma$-spaces, established by Shimakawa. An important case is when $Y=\spec(\C)$, in which case $\mcal{K}_{G}^{sst}(X,\C)$  defines equivariant semi-topological  $K$-theory. 
Similarly the equivariant semi-topological $K$-homology of $Y$ is $\mcal{K}_{G}^{sst}(\C,Y)$.

Our first main result is the following generalization of Walker's comparison theorem, appearing as Theorem \ref{mainthm} below.
\begin{theorem}
 Let $Y$ be a smooth quasi-projective $G$-variety. Then there is a natural  weak
equivalence of $G$-spectra
$$
\bu^{\mf{c}}_{G}(S^{0},Y^{an})\xrightarrow{\wkeq} \mcal{K}^{sst}_{G}(\C,Y).
$$
\end{theorem}
Here $\bu^{\mf{c}}_{G}(S^{0},Y^{an})$ is the equivariant topological  $K$-homology introduced in Section \ref{sec:top} and is  shown to be equivariantly weakly equivalent to $Y^{an}\wedge \bu_{G}$, where $\bu_{G}$ is the connective cover of the $G$-spectrum $KU_{G}$ representing equivariant complex $K$-theory.

The topological Bott element lifts (uniquely) to a ``semi-topological Bott element'', $\beta_{2}\in K_{2}^{G,\,sst}(\C,\C)$ and our second main result is Theorem \ref{mainthmbottinv}, establishing that Bott-inverted equivariant semi-topological $K$-theory and equivariant topological  $K$-theory agree.
\begin{theorem}
Let $X$ be a smooth complex projective
$G$-variety. Then there are natural isomorphisms 
$$
K^{G,sst}_{*}(X,\C)[\beta^{-1}_{2}] \xrightarrow{\iso}
KU_{G}^{-*}(X^{an}).
$$
\end{theorem}
This result is proved following the outline of the argument in \cite{Walker:Thomason}. Namely, the proof relies on the equivariant version of Walker's comparison theorem mentioned above, the pairings constructed in Section \ref{sec:pair}, and a good equivariant
theory of fundamental and Thom classes. Regarding this last item, Walker uses the fact that nonequivariantly connective $K$-theory has Thom classes and satisfies Poincare duality. Here some nontrivial changes need to be made in the equivariant setting since the version of connective equivariant $K$-theory that appears in our work is not {\it complex stable}, see the discussion before Lemma \ref{kuisalmostKU}.

As a consequence of the rigidity property for equivariant algebraic $K$-theory established by Yagunov-{\O}stv{\ae}r \cite{YO:equirigid} and Friedlander-Walker's recognition principle \cite{FW:ratisos}, we establish in Theorem \ref{KalgKsemi} an isomorphism 
$$
K^{G,\,alg}_{*}(X;\Z/n) \xrightarrow{\iso} K_{*}^{G,\,sst}(X;\Z/n)
$$
for smooth $X$. Here, in order to have a comparison map between our equivariant algebraic and semi-topological $K$-theories, it is important to have available an equivariant version of the Grayson-Walker theorem concerning geometric models for $K$-theory spectra. This result is proved by {\O}stv{\ae}r in \cite{Ostvaer}.
As a consequence of the above isomorphism, Theorem \ref{mainthmbottinv} specializes to give an alternate proof (in the case of smooth projective $G$-varieties) of the equivariant version of Thomason's theorem \cite[Theorem 5.9]{Thomason:famousequi}, comparing Bott-inverted equivariant algebraic $K$-theory and equivariant complex $K$-theory (with finite coefficients).

Due to considerations of length, we have not discussed here several other generalizations and related results
which are likely to be true. First, (some version of) Theorem \ref{mainthm}
should be true for Real semi-topological $K$-homology and real varieties. 
Second, the results of section \ref{sec:sstThom}
probably hold for quasi-projective varieties as well, by replacing the homology theories appearing there with a Borel-Moore type homology theory.
Third, Theorem \ref{KalgKsemi} probably holds for bivariant 
algebraic $K$-theory as well, as the base change, normalization
and additivity property necessary to establish rigidity
seem to extend to the corresponding  
categories of $G$-modules. Fourth, the ring structure on the equivariant algebraic $K$-theory introduced in section \ref{sec:pair} presumably coincides with the previously considered ring
structure  and similarly for the topological theory.  (Note that Proposition \ref{prop:conn} implies the product on the cohomology theory is the same in positive degrees, see also
\cite[Remark 6.11]{Walker:Thomason}.) Finally, in light of the equivariant generalizations of  \cite{Walker:Thomason} presented here, it would be interesting to know whether the more general results of \cite{Walker:Thomalg} admit an equivariant generalization as well. 

\medskip

We conclude with an overview of the article.

In section \ref{sec:pre}, we review some material 
on equivariant stable homotopy theory, in particular
about equivariant $\Gamma$-spaces and equivariant group completion. In
Section 3 we introduce and study various models
for equivariant bivariant algebraic, semi-topological and
topological $K$-theory we need to consider. 

In section \ref{sec:comp} we establish the equivariant version of Walker's  comparison theorem between equivariant semi-topological and
equivariant topological $K$-homology. 
Section \ref{sec:pair} is devoted to a detailed study of pairings
and operations (e.g. slant products) for the
various equivariant $K$-theories appearing in this article.
  
In section \ref{sec:sstThom}, we establish that Bott-inverted equivariant semi-topological $K$-theory and equivariant topological $K$-theory agree, for smooth projective complex $G$-varieties. In section \ref{sec:algThom} we show how the semi-topological result implies Thomason's result, thus giving a new proof in the the equivariant setting for smooth projective complex varieties. 

In a companion article \cite{Ostvaer}, {\O}stv{\ae}r shows that the 
Grayson-Walker model of algebraic $K$-theory
\cite{GW:Kmodels} allows for an equivariant 
generalization (which we use  to write
down the comparison map between the equivariant algebraic and semi-topological $K$-theories that is used in  Section \ref{sec:algThom}). We thank him for helpful discussions regarding this result.

We are grateful to the referee for pointing out an error in our previous proof of Theorem \ref{mainthm}.
Additionally, the first author would like to thank M. Voineagu for several useful conversations on closely related topics.

\textbf{Notation:} 
Unless stated otherwise, $G$ will be a finite group. 
We write $\Sch/\C$ for the category of quasi-projective complex varieties and $\Sm/\C$ for the full subcategory of smooth quasi-projective complex varieties. For a complex variety $X$, the set of complex points equipped with the Euclidean topology is denoted by $X^{an}$.

\section{Preliminaries}\label{sec:pre}

\subsection{Stable equivariant homotopy theory, {$\Gamma_{G}$}-spaces, and {$\mcal{W}_{G}$}-spaces}
We write $\GTop$ for the category whose objects are compactly generated 
Hausdorff spaces with $G$-action together with a $G$-invariant base-point and morphisms are based $G$-equivariant maps. Write $\Top_{G}$ for the category with the same objects as  $\GTop$,
but morphisms are all based continuous maps, hence $g(f(x)):=gf(g^{-1}x)$
defines a $G$-action on the morphism sets. Both these categories are enriched over topological spaces and $\Top_{G}$ is enriched over $\GTop$.

In this paper, a $G$-spectrum means an orthogonal $G$-spectrum, unless otherwise specified; we usually omit the adjective orthogonal. If $A$ is an orthogonal $G$-spectrum we also write $A$ for its underlying (pre)-spectrum. We refer to \cite{May:equihomotopy} for background on and a good survey of equivariant stable homotopy theory, \cite{LLM} for further details concerning ``classical'' spectra in the equivariant setting, and \cite{MM:O} for equivariant orthogonal spectra. As is customary we write $[X,Y]_{G}$ for maps in the $G$-equivariant stable homotopy category. For a representation $V$ we write 
$$
\pi_{V}^{G}X = [S^{V},X]_{G}.
$$

In this paper our spectra arise primarily via equivariant $\Gamma$-spaces and $\mcal{W}_{G}$-spaces, and we now recall some details on these. Let $\mcal{W}_{G}$ denote the category of based $G$-spaces that are homeomorphic to finite $G$-$CW$-complexes and maps are all (base-point preserving) maps. A \textit{$\mcal{W}_{G}$}-space is a based, equivariant functor $X:\mcal{W}_{G}\to \Top_{G}$ such that 
$$
X:\Map(A,B) \to \Map(X(A),X(B))
$$ 
is an equivariant continuous map of $G$-spaces.
We have a map of $G$-spaces 
$X(A)\wedge B \to X(A\wedge B)$ obtained as the adjoint of the composition $B\to \Map(A,A\wedge B) \to \Map(X(A), X(A\wedge B))$. In particular, a $\mcal{W}_{G}$-space $X$ functorially determines an orthogonal $G$-spectrum $\mathbb{U}X$ via $(\mathbb{U}X)(V) = X(S^{V})$ and hence it also determines a $G$-prespectrum. Moreover $\mathbb{U}$ is the right adjoint in a Quillen equivalence between $\mcal{W}_{G}$-spaces and orthogonal $G$-spectra (indexed on a complete universe) and the category of $\mcal{W}_{G}$-spaces has a smash-product such that $\mathbb{U}$ is lax symmetric monoidal,  see \cite{AB} for details.

There are two equivalent formulations of equivariant $\Gamma$-spaces. The first is as follows. Let $\Gamma$ denote the category whose objects are pointed sets $\underline{n}_{+}=\{0,1,\ldots,n\}$, pointed at $0$. Maps are base-point preserving set maps. An equivariant $\Gamma$-space is a functor  $X:\Gamma \to \GTop$ such that $X(0)= *$. Write $\Gamma[\GTop]$ for the category whose objects are the equivariant $\Gamma$-spaces and morphisms are natural transformations. 
The second model is as follows. Let $\Gamma_{G}$ denote a skeletal category of finite $G$-sets with morphisms all pointed set maps. The category $\Gamma_{G}[\Top_{G}]$ has as objects equivariant functors $X:\Gamma_{G}\to \Top_{G}$ such that $X(0) = *$ and maps are equivariant natural transformations. A useful observation due to Shimakawa and May \cite{Shimakawa:note} is that there is an adjoint pair of functors
\begin{equation}\label{eqn:gammaeq}
i:\Gamma_{G}[\Top_{G}]\rightleftarrows  \Gamma[\GTop]:P
\end{equation}
which are an equivalence of categories. Here $i$ is induced by the inclusion functor $i:\Gamma\to \Gamma_{G}$.

In this paper we will generally work with the objects of $\Gamma[\GTop]$ which we refer to simply as equivariant $\Gamma$-spaces. We refer to objects of $\Gamma_{G}[\Top_{G}]$ as $\Gamma_{G}$-spaces. The equivalence $P$ is defined as follows. Let $X:\Gamma\to \GTop$ be an equivariant $\Gamma$-space and $S$ a finite $G$-set. Write $S$ also for the contravariant functor  $\Map(-,S)$ which it represents. 
 The value of $PX:\Gamma_{G}\to \Top_{G}$ on a $G$-set is defined via the left Kan extension $PX(S) = S\otimes_{\Gamma}X$. Alternatively, $PX(S)$ can be described as follows. 
A $G$-set $S$ corresponds to a group homomorphism $\rho:G\to \Sigma_{n}$ where $|S|=n$.
Given a homomorphism $\rho:G\to \Sigma_{n}$ one defines a new $G$-action on $X(n)$ via the formula $g\cdot_{\rho}x = X(\rho(g))(gx)$ for $x\in X(n)$. Write $X(n)_{\rho}$ for this $G$-space. Then $PX(S) = X(n)_{\rho}$.

\begin{definition}
\begin{enumerate}
\item A $\Gamma_{G}$-space $X:\Gamma_{G}\to \Top_{G}$ is said to be \textit{special} provided $X(S) \to \Map_{cts*}(S,X(\underline{1}))$ is a $G$-weak equivalence for any $S$. 

\item Say that an equivariant $\Gamma$-space is \textit{special} if for every subgroup $H\subseteq G$ and homomorphism $\rho:H\to \Sigma_{n}$ the map 
$$
X(n)_{\rho} \to (X(1)^{n})_{\rho}
$$
is an $H$-weak equivalence, where $(X(1)^{n})_{\rho}$ is the $G$-space with action given by $g(x_{1},\ldots, x_{n}) = (gx_{\rho(g)(1)},\ldots, gx_{\rho(g)(n)})$. 
\end{enumerate}
\end{definition}
One easily checks that the two notions correspond to each other under the above equivalence. 

Segal introduced $\Gamma$-spaces in order to produce homotopy group completions. 
\begin{definition}
\begin{enumerate}
\item A map $A\to B$ of homotopy associative, homotopy commutative $H$-spaces is said to be a \textit{homotopy group completion} provided that
\begin{enumerate}
 \item $\pi_{0}B$ is an abelian group and the map $\pi_{0}A\to \pi_{0}B$ is a group completion of the abelian monoid $\pi_{0}A$, and
\item $H_{*}(A,R)\to H_{*}(B,R)$ is the localization mapping
$$
H_{*}(A,R) \to \Z[\pi_{0}B]\otimes_{\Z[\pi_{0}A]}H_{*}(A,R),
$$
for any commutative ring $R$.
\end{enumerate}

\item Say that a $G$-space $A$ is an \textit{equivariant homotopy commutative, associative $H$-space} if it is a homotopy commutative, associative $H$-space,  the $H$-space structure map is equivariant, and the homotopies for associativity and commutativity can be taken to be equivariant. Say that an equivariant $H$-space map $A\to B$ is an \textit{equivariant homotopy group completion} provided that $A^{K}\to B^{K}$ is a homotopy group completion for all subgroups $K\subseteq G$. 
\end{enumerate}
\end{definition}
Our basic example occurs when $X(-)$ is a special equivariant $\Gamma$-space (in $G$-$CW$-complexes).

There is a functor from equivariant $\Gamma$-spaces
to $G$-spectra generalizing the classical one for
the trivial group $G$ as follows. Given an equivariant $\Gamma$-space $X$ we obtain a $\mcal{W}_{G}$-functor, which we denote $\widehat{X}$ via 
$$
\widehat{X}(M) = |B(M,\Gamma_{G}, PX)|
$$
where $M$ is viewed as the functor $\Map(-,M):\Gamma_{G}^{op}\to \Top_{G}$ which it represents and $B(-,-,-)$ denotes the two-sided bar construction. Write 
$$
\mathbb{S}X = \{\widehat{X}(S^{V})\}
$$ 
for the spectrum $\mathbb{U}\widehat{X}$ associated to the $\mcal{W}_{G}$-space $\widehat{X}$.

\begin{lemma}\label{lem:Gspc}
Let $X$ be an equivariant $\Gamma$-space.
\begin{enumerate}
\item  View $X(n)\to X(1)^{\times n}$ as a map of $G\times\Sigma_{n}$-spaces, where $(g,\sigma)$ acts on $X(n)$ via $(g,\sigma)\cdot x = X(\sigma)(gx)$ and on the $X(1)^{\times n}$ by $(g,\sigma)(x_{1},\ldots, x_{n}) = (gx_{\sigma(1)},\ldots, gx_{\sigma(n)})$. 

If $X(n)\to X(1)^{\times n}$ is a $G\times\Sigma_{n}$-weak equivariant equivalence for all $n$, then $X$ is special.
\item  If $X$ is a special equivariant
$\Gamma$-space, then $\mathbb{S}X$ is a positive $G-\Omega$-spectra and
the map
$$
X(1)\wkeq \widehat{X}(S^{0}) \to \Omega\widehat{X}(S^{1})=\Omega \mathbb{S}X_{1} 
$$
is an equivariant group completion.
\end{enumerate}
\end{lemma}
\begin{proof}
See \cite{Shimakawa:Gloop}.
\end{proof}

If $X$ is an equivariant $\Gamma$-space or a $\Gamma_{G}$-space, $\underline{n}_{+}\mapsto X(n)^{H}$ defines an ordinary $\Gamma$-space. Given a $\Gamma$-space $\mcal{A}(-)$ write $\mathbf{B}\mcal{A} = (\mcal{A}(\underline{1}_{+}), B\mcal{A}(\underline{1}), B^{2}\mcal{A}(\underline{1}),\ldots )$ for the associated spectrum as in \cite{Segal:gamma}. The condition in the following guarantees that the simplicial space $n\mapsto X(n)$ is good in the sense of \cite{Segal:gamma}

\begin{lemma}\label{lem:spcfix}
 Let $X$ be a special equivariant $\Gamma$-space of the form $X(-) = |X^{\prime}(-)|$ where $X^{\prime}:\Gamma\to G\sSet$ and $H\subseteq G$ a subgroup. Then $\pi_{n}^{H}(\mathbb{S}X) \iso 
\pi_{n}(\mathbf{B}X^{H})$.
\end{lemma}
\begin{proof}
Since $X$ is special, $\mathbb{S}X$ is a positive-$\Omega$-$G$-spectrum and so by \cite[Proposition V.3.2]{MM:O} we have that $\pi_{n}^{H}(\mathbb{S}X) = \pi_{n}(\mathbb{S}X^{H})$. Recall that if $A$ is an orthogonal $G$-spectrum and $V=\R^{n}$ has trivial action then $A^{H}(\R^{n}) = A(\R^{n})^{H}$ (in terms of underlying pre-spectra: $(A^{H})_{n} = (A_{n})^{H}$). 
  By \cite[Proposition 1.2(c)]{Shimakawa:Gloop} the map $|B(S^{n},\Gamma, X^{H})| \to |B(S^{n},\Gamma_{G}, X)^{H}|$ is a homotopy equivalence. By \cite[Proposition 3.2]{Segal:gamma}, $B^{n}X^{H} = S^{n}\otimes_{\Gamma}X^{H}$ and because the overcategory $(\Gamma\downarrow S^{n})$ is filtered, the map $|B(S^{n},\Gamma, X^{H})|\to S^{n}\otimes_{\Gamma}X^{H}$ is a weak equivalence.  Therefore $\pi_{n}(\mathbb{S}X)^{H} = \pi_{n}\mathbf{B}(X^{H})$. 
\end{proof}

\subsection{Pairings}\label{subsectionpairings}
The external product $X\barwedge Y$ of two  $\mcal{W}_{G}$-spaces $X$ and $Y$ is the $\mcal{W}_{G}\times \mcal{W}_{G}$-space given by $(X\barwedge Y)(A,B) = X(A)\wedge Y(B)$. The \textit{smash product} $X\wedge Y$ of $\mcal{W}_{G}$-spaces is defined as the left Kan extension of $X\barwedge Y$ along the functor $\wedge:\mcal{W}_{G}\times \mcal{W}_{G}\to \mcal{W}_{G}$ given by $(A,B)\mapsto A\wedge B$, see \cite{AB}. 
 By the universal property of Kan extension, giving a pairing $X\wedge Y \to Z$ of $\mcal{W}_{G}$-spaces is equivalent to 
giving a map $X\barwedge Y \to Z\circ \wedge$ of $(\mcal{W}_{G}\times\mcal{W}_{G})$-spaces. Since $\mathbb{U}$ is lax symmetric monoidal, the pairing $X\barwedge Y \to Z\circ \wedge$ defines a pairing of associated orthogonal $G$-spectra $\mathbb{U}X\wedge\mathbb{U}Y \to \mathbb{U}Z$.

Let $X$, $Y$ be equivariant $\Gamma$-spaces and $X\barwedge Y$ is the equivariant $\Gamma\times\Gamma$-space $(\underline{p}_{+},\underline{q}_{+})\mapsto X(\underline{p}_{+})\wedge Y(\underline{q}_{+})$.
Let $\wedge:\Gamma\times \Gamma\to \Gamma$ be defined by identifying $\underline{p}_{+}\wedge \underline{q}_{+}$ with $\underline{pq}_{+}$ via the lexicographical ordering.
 A map of equivariant $(\Gamma\times \Gamma)$-spaces 
 $X\barwedge Y \to Z\circ \wedge$ determines a pairing of associated spectra, as we now explain. The functor $P$ which associates a $\Gamma_{G}$-space to an equivariant $\Gamma$-space, described in the previous section, can be extended in the evident way to a functor taking equivariant $(\Gamma\times\Gamma)$-spaces to $(\Gamma_{G}\times \Gamma_{G})$-spaces and we again denote this functor by $P$. A straightforward inspection shows that $P(X\barwedge Y) = PX\barwedge PY$. Together with the natural map $P(Z\circ \wedge) \to PZ\circ \wedge$, this implies that a map of equivariant $(\Gamma\times \Gamma)$-spaces $X\barwedge Y \to Z\circ\wedge$ gives rise to a map $PX\barwedge PY \to PZ\circ\wedge$ of $(\Gamma_{G}\times\Gamma_{G})$-spaces. This map in turn gives rise to a map $\widehat{X}\barwedge\widehat{Y}\to \widehat{Z}\circ\wedge$ of $(\mcal{W}_{G}\times\mcal{W}_{G})$-spaces and therefore we obtain a pairing $\mathbb{S}X\wedge\mathbb{S}Y\to \mathbb{S}Z$.

\subsection{Homotopy colimits of $G$-spaces}
Homotopy colimits can be viewed as the derived functors of the colimit functor. For our purposes it is important to use a functorial model for the homotopy colimit of a diagram of $G$-simplicial sets or spaces and we take the ``standard model''. Explicitly, let $X:D\to \mcal{C}$ be a functor, where $\mcal{C}$ is the category of $G$-simplicial sets or spaces, then
$$
\hocolim_{D}X = |B(*,D,X)|
$$
where $B(-,-,-)$ denotes the two-sided bar construction. Observe that this formula shows that $(\hocolim_{D}X)^{H} = \hocolim_{D}X^{H}$ for any subgroup $H\subseteq G$.

\subsection{{$G$}-modules}
If $X$ is a $G$-scheme then a coherent $G$-module on $X$ is a coherent $\mcal{O}_{X}$-module $\mcal{M}$ together with isomorphisms $\phi_{g}:\mcal{M} \to g_{*}\mcal{M}$ for each $g\in G$ such that $\phi_{e} = id$ and $\phi_{gh} = h_{*}\phi_{g}\phi_{h}$. If $X = \spec(R)$ then $R$ has a $G$-action which we write as a left-action. Specifying a coherent $G$-module $\mcal{M}$ on $X$ is equivalent to specifying an $R$-module $M$ together with a $G$-action on $M$ which is compatible with the action on $R$ in the sense that $(g\cdot r)m = g\cdot (r\cdot (g^{-1}m))$ (i.e. $M$ is a module over the skew-group ring $R^{*}G$).

\section{Bivariant $K$-theories}
In this section we introduce the algebraic, semi-topological and topological bivariant $K$-theories with which we work in the paper. All of these are constructed as $G$-spectra. The construction of the algebraic bivariant $K$-theory spectrum as a $G$-spectrum makes use of the equivariant Grayson-Walker theorem proved by {\O}stv{\ae}r \cite{Ostvaer}. The bivariant semi-topological equivariant $K$-theory is constructed and studied in \ref{sub:sst} and its topological counterparts are introduced and studied in \ref{sec:top} and \ref{sec:tel}. The comparison map between the semi-topological and topological $K$-theories is constructed and studied in the next section. The material in these sections corresponds mostly to  material in Sections 3 and 4 in \cite{Walker:Thomason}. While the overall picture of the results presented here corresponds nicely to that in Walker's paper, there are parts of the picture which differ. Before beginning,
we point out some of the global differences of significance between our presentation of this material and the corresponding material there. First, Walker defines the semi-topological bivariant theory via topological spaces of algebraic maps while we use Friedlander-Walker's simplicial model for this space, see Remark \ref{rem:mor} below. Second, Walker makes use of $\Gamma$-spaces produced by taking nerves of certain topological categories while we prefer to simply describe our $\Gamma$-spaces as being obtained from a homotopy colimit of a certain diagram, see Remark \ref{rem:hocol}.

\subsection{Algebraic $K$-theory}\label{sec:algK}
In this subsection we work with quasi-projective $G$-varieties over an arbitrary field $k$, where $G$ is a finite group whose order is coprime to $char(k)$ (though, only $k=\C$ is used in later sections). Write $\mcal{P}(G;X,Y)$ for the category of coherent $G$-modules on $X\times Y$ which are finite and flat over $X$ (this is the category $\mcal{P}^{0}(G;X,Y)$ in the notation of \cite{Ostvaer}). This is an exact category, and we write $\mcal{K}(G;X,Y)$ for the associated $K$-theory spectrum, as produced  by Waldhausen's construction. We explain how, when specialized to the case of finite groups, the material in \cite{Ostvaer} yields a $G$-spectrum  $\mcal{K}_{G}(X,Y)$ with the property that $\mcal{K}_{G}(X,Y)^{H} \wkeq \mcal{K}(H;X\times\Delta^{\bullet}_{k},Y)$.

Let $V$ be a finite dimensional  $G$-representation over $k$. Then $V$ defines a $G$-bundle  on $\spec(k)$ and we write $\mcal{V}$ for this $G$-bundle. For any $G$-variety $Y$ over $k$, the pullback of $\mcal{V}$ via the structure map $Y\to \spec(k)$ is a $G$-bundle on $Y$ which we denote as $\mcal{V}_{Y}$. 

\begin{definition}\label{defn:G}
Let $X$,$Y$ be a quasi-projective $G$-schemes over $k$ and $V$ a representation of $G$. Write $\pi_{X}:X\times Y\to X$ for the projection map. Define the sheaf of sets $\GG^{V}_{Y}(n):\Sch/k^{op} \to \Set$ to be the sheaf whose value on $X$ is the collection of quotient objects $[\mcal{V}^{n}_{X\times Y}\onto \mcal{M}]$ which satisfy the following conditions.
\begin{enumerate}
\item The support of $\mcal{M}$ is finite over $X$ and $(\pi_{X})_{*}\mcal{M}$ is locally free.
\item The composition $\mcal{V}^{n}_{X}\to (\pi_{X})_{*}\mcal{V}^{n}_{X\times Y}\to (\pi_{X})_{*}\mcal{M}$ is surjective.
\end{enumerate}
\end{definition}

Forgetting the $G$-action, the underlying sheaf $\GG^{V}_{Y}(n)$ is denoted by $G^{mn}_Y$ in \cite{Walker:Thomason}, where $m=\dim V$. When condition (1) is satisfied we say that $\mcal{M}$ is \textit{finite and flat over $X$}. As remarked by Walker,
condition (2) could be omitted and one would still obtain the same bivariant theory, see Remark \ref{rem:WGW} and \cite[Theorem A.10]{Ostvaer}. Its inclusion has several advantages, one of which is that it allows the construction of functorial maps, e.g. Lemma \ref{lem:covariant} below, which otherwise would merely be functorial up to homotopy.

Let $Y$ be a $G$-variety over $k$. 
For each $g\in G$ we have isomorphisms $\phi_{g}:\mcal{V}^{n}_{Y} \to g_{*}\mcal{V}^{n}_{Y}$ such that $\phi_{e} = id$ and $\phi_{gh} = h_{*}\phi_{g}\phi_{h}$. For any $G$-variety $X$, we have a natural $G$-action on $\GG^{V}_{Y}(n)(X)$  given by 
$$
g\cdot[\mcal{V}^{n}_{X\times Y}\onto \mathcal{M}] = [\mcal{V}^{n}_{X\times Y}\xrightarrow{\phi_{g}}g_{*}\mcal{V}^{n}_{X\times Y}\onto g_{*}\mathcal{M}].
$$
This defines a $G$-action on the functor $\mcal{G}^{V}_{Y}(n)$.
By \cite[Lemma 2.2]{Walker:Thomason} (which is written for $k=\C$, but it holds verbatim over any field $k$) the functor $\mcal{G}^{V}_{Y}(n)$ is represented by a quasi-projective variety which we denote by the same symbol. In fact $\mcal{G}^{V}_{Y}(n)$ is an open invariant subscheme of the Quot-scheme $\coprod_{r=0}^{n}Quot^{r}_{\mcal{V}^{n}_{Y}/Y/\spec(k)}$, where $Quot^{r}_{\mcal{V}^{n}_{Y}/Y/\spec(k)}$ is the functor which sends $U$ to the set of  quotient objects $[\mcal{V}^{n}_{U\times Y} \onto \mcal{M}]$ such that $\mcal{M}$ is finite and flat over $U$ and $\pi_{*}\mcal{M}$ is locally free of rank $r$ on $U$. 
The $G$-action on the functor $\GG_{Y}^{V}(n):\Sch/k^{op}\to \Set$ defines an action on the representing scheme. Thus for a $G$-scheme $X$, we have an action on $\Hom_{\Sch/k}(X,\,\GG^{V}_{Y}(n))$ defined by the usual formula, $g\cdot f$ is the function $x\mapsto gf(g^{-1}x)$. This action agrees with the previously described action on $\GG^{V}_{Y}(n)(X)$.

We now describe the equivariant $\Gamma$-space which will define the spectrum $\mcal{K}_{G}(X,Y)$. Write $I$ for the category whose objects are $\underline{n}=\{1,2,\ldots, n\}$ for each $n\geq 0$ ( $\underline{0}$ is the empty set) and morphisms are injective, but not necessarily order-preserving, maps of sets.  An injection $j:\underline{m}\hookrightarrow \underline{n}$ determines a map of coherent $G$-modules $j_{*}:\mcal{V}^{m} \to \mcal{V}^{n}$ where $j_{*}(e_{i}) = e_{j(i)}$. Write $j^{*}:\mcal{V}^{n}\onto \mcal{V}^{m}$ for the transpose of this map. Explicitly we have
$$
j^{*}(e_{i})=  \begin{cases}
e_{k} & \textrm{if} \,\, j(k) = i,\,\,\textrm{and} \\
0 & \textrm{if}\,\, i\notin \im(j) .               
              \end{cases}
$$
Note that $j^{*}j_{*} = id$. We have induced morphisms  $j^{*}:\mcal{V}^{n}_{X\times Y}\onto \mcal{V}^{m}_{X\times Y}$ and precomposition with $j^{*}$ defines a natural transformation $j^{*}:\GG_{Y}^{V}(m)(X)\to G_{Y}^{V}(n)(X)$.
Note that  $j^{*}[p:\mcal{V}_{X\times Y}^{m}\onto \mcal{M}]=[q:\mcal{V}_{X\times Y}^{n}\onto \mcal{N}]\in \GG^{V}_{Y}(n)(X)$   exactly when there exists an isomorphism $\mcal{M}\xrightarrow{\iso} \mcal{N}$ such that the diagram commutes
$$
\xymatrix{
\mcal{V}^{m}_{X\times Y} \ar@{->>}[r]^{p} & \mcal{M} \ar@{-->}[d]^{\iso} \\
\mcal{V}^{n}_{X\times Y} \ar@{->>}[u]^{j^{*}} \ar@{->>}[r]^{q} & \mcal{N} .
}
$$
 
The equivariant maps $j^{*}$ make the assignment
$$
\underline{n}\mapsto \Hom_{\Sch/k}(X,\GG_{Y}^{V}(n))
$$ 
into a functor $I\to GSet$.

\begin{lemma}[{\cite[Lemma 2.3]{Walker:Thomason}}] \label{lem:covariant}
Suppose that $f:Y\to Y'$ is an equivariant map of quasi-projective $G$-varieties over $k$.
There is an equivariant morphism of $G$-varieties
$f_{*}:\mcal{G}^{V}_{Y}(n)\to \mcal{G}^{V}_{Y'}(n)$, 
defined via the equivariant natural transformation of functors, 
$f_{*}:\mcal{G}^{V}_{Y}(n)(-)\to \mcal{G}^{V}_{Y'}(n)(-)$, 
which on $U$ sends the quotient $[q:\mcal{V}^{n}_{U\times Y}\onto \mcal{M}]$ to $[\mcal{V}^{n}_{U\times Y'}\onto (id\times f)_{*}\mcal{M}]$. Moreover $f_{*}$ is natural with respect to maps in $I$ and thus defines  a transformation of $I$-functors.
\end{lemma}
\begin{proof}
The map $f_{*}$ is well-defined as a consequence of condition (2) in Definition \ref{defn:G}. The required naturality and equivariance statements are clear.
\end{proof}

Define the presheaf of $G$-simplicial sets $\AA_{G}(-,\,Y)^{V}$ by
$$
\AA_{G}(X,\,Y)^{V} = \hocolim_{I}\Hom_{\Sch/k}(X\times\Delta^{\bullet}_{k},\,\GG_{Y}^{V}(-)).
$$

We extend this definition to pairs $(Y,y_{0})$ consisting of a quasi-projective $G$-variety together with an invariant $k$-rational basepoint $y_{0}\in Y$. This is useful later, when $k=\spec(\C)$, to discuss comparisons to the topological setting. Usually the pointed variety will be $Y_{+}= Y\coprod \spec(k)$, where we have adjoined a disjoint basepoint.  Note that $\GG_{y_{0}}^{V}(n)\subseteq \GG_{Y}^{V}(n)$ and this inclusion is functorial in the variable $n$. Define 
$$
\GG_{(Y,y_{0})}^{V}(n)(X) = \GG_{Y}^{V}(n)(X)/\GG_{y_{0}}^{V}(n)(X)
$$ 
and
$$
\AA_{G}(X,\,(Y,y_{0}))^{V} = \hocolim_{I}\Hom_{\Sch/k}(X\times\Delta^{\bullet}_{k},\,\GG_{(Y,y_{0})}^{V}(-)).
$$

Note that $\AA_{G}(X,\,(Y_{+},*))^{V} = \AA_{G}(X,\,Y)^{V}$. Usually we omit the base-point from the notation when the context makes it clear what is meant. The $G$-simplicial sets $\AA_{G}(X,\,Y)^{V}$ are clearly contravariantly natural in the first variable and are covariantly natural in the second variable by Lemma \ref{lem:covariant}. Using the covariant naturality in the second variable we obtain equivariant $\Gamma$-spaces
$$
\underline{n}_{+}\mapsto |\AA_{G}(X,\,\underline{n}_{+}\wedge Y_{+})^{V}|.
$$

\begin{remark}
More generally the assignment 
$S\mapsto |\AA_{G}(X,\, S\wedge Y_{+})^{V}|$ 
for a based finite $G$-set $S$, extends the above assignment to a $\Gamma_{G}$-space. By the equivalence of categories (\ref{eqn:gammaeq}) it makes no difference whether we work with the equivariant $\Gamma$-space displayed above or with this $\Gamma_{G}$-space.
\end{remark}

We are most interested in the case of the regular representation $V=k[G]$.
Recall that an equivariant $\Gamma$-space $\mcal{A}$ naturally gives rise to a $G$-spectrum (see Lemma \ref{lem:Gspc}) which we denote $\mathbb{S}\mcal{A}$.
\begin{definition}\label{defKalg}
Let $X$,$Y$ be quasi-projective $G$-schemes. Define the $G$-spectrum $\mcal{K}_{G}(X,Y)$ by
$$
\KK_{G}(X,Y)=\S|\AA_{G}(X,\,-\wedge Y_+)^{k[G]}|,
$$
the spectrum associated 
 to the equivariant $\Gamma$-space $|\AA_{G}(X,\,-\wedge Y_+)^{k[G]}|$.
Write $K^{G,\,\A^{1}}_{*}(X,Y) = \pi^{G}_{*}\mcal{K}_{G}(X,Y)$ 
for the homotopy groups of this spectrum. More generally if $H\subseteq G$ is a subgroup, write 
$K^{H,\,\A^{1}}_{*}(X,Y) = \pi^{H}_{*}\mcal{K}_{G}(X,Y)$.
\end{definition}
\begin{remark}
The spectra $\mcal{K}_{G}(X,Y)$ are contravariantly natural in the first variable and are covariantly natural in the second variable by Lemma \ref{lem:covariant}.
\end{remark}
Later in this section we will see that $K^{G,\,\A^{1}}_{*}(X,Y) = \pi_{*}\mcal{K}(G;X\times\Delta^{\bullet}_{k},Y)$.

We now introduce a slightly different model which is often convenient to work with. Write 
$$
\Hom_{\Sch/k}(X,\,\GG_{Y}^{V}(\infty))=\colim_{n}\Hom_{\Sch/k}(X ,\,\GG_{Y}^{V}(n))
$$
where the colimit is over the standard inclusions $\underline{n}\subseteq \underline{n+1}$ given by $i\mapsto i$ and the transition maps are induced by precomposition with the canonical surjections $\mcal{V}^{n+1} \to \mcal{V}^{n}$. Thus an element is a quotient object $[\mcal{V}^{\infty}_{X \times Y} \xrightarrow{p} \mcal{M}]$ where $\mcal{V}^{\infty}_{X\times Y} \to \mcal{M}$ factors as $\mcal{V}^{\infty}_{X\times Y} \to \mcal{V}^{N}_{X \times Y}\to \mcal{M}$ for some $N$. If $j:\mathbb{N}\to \mathbb{N}$ is an injection we have an induced map given by $j^{*}[\mcal{V}^{\infty}_{X \times Y} \xrightarrow{p} \mcal{M}] = [p\tilde{j}^{*}:\mcal{V}^{\infty}_{X \times Y} \xrightarrow{}\mcal{M}]$, and similarly for pointed varieties $(Y,y_{0})$.

In what follows it will be convenient to write $\infty$ for the set $\mathbb{N}$. Let $\tilde{I}$ be the category whose objects are the finite sets $\underline{n}$ together with the set $\infty$ and whose morphisms are injections. Let $M\subseteq \tilde{I}$ be the full subcategory containing the object $\infty$. The category $M$ consists of one object $\infty$ and $\Hom_{M}(\infty,\infty)$ is the monoid (under composition) of injective set maps $\mathbb{N}\to \mathbb{N}$. We will abuse notation and also write $M$ for this monoid.
\begin{proposition}\label{prop:alginf}
 Let $X$ and $(Y,y_{0})$ be quasi-projective $G$-varieties over $k$. The natural maps
\begin{multline*}
\hocolim_{I}\Hom(X\times\Delta^{\bullet}_{k},\,\GG_{Y}^{V}(-)) \to \hocolim_{\tilde{I}} \Hom(X\times\Delta^{\bullet}_{k},\,\GG_{Y}^{V}(-)) \\
\leftarrow \hocolim_{M}\Hom(X\times\Delta^{\bullet}_{k},\,\GG_{Y}^{V}(\infty))
\end{multline*}
are equivariant weak equivalences.
\end{proposition}
\begin{proof}
 The proposition could be obtained as a particular case of \cite[Proposition 2.2.9]{Shipley:THH} but for the convenience of the reader we sketch the full argument. Write $F(-) = \Hom(X\times\Delta^{\bullet}_{k},\,\GG_{Y}^{V}(-))$.  Because the overcategory $(I\downarrow\infty)$ is filtered, we have the equivariant weak equivalence $L_{hK}F(\infty)\to L_{K}F(\infty)$ where $L_{K}F(-)$ and $L_{hK}(-)$ respectively denote the left Kan extension and the homotopy left Kan extension of $F|_{I}$ along $I\subseteq \tilde{I}$. Moreover, we have that the functors $F$,$L_{K}F:\tilde{I}\to G\sSet$ are equal. Since $\hocolim_{I} F \to \hocolim_{\tilde{I}}L_{hK}F$ is an equivariant weak equivalence we have that
$\hocolim_{I} F \to \hocolim_{\tilde{I}}F$
is an equivariant weak equivalence as well. Since $M\subseteq \tilde{I}$ is right cofinal we have that $\hocolim_{M}F(\infty) \to \hocolim_{\tilde{I}}F$ is an equivariant weak equivalence as well.
\end{proof}

\begin{proposition}\label{prop:Malgwkeq}
The map 
$$
\Hom(X\times\Delta^{\bullet}_{k},\,\GG_{Y}^{V}(\infty)) \to \hocolim_{M}\Hom(X\times\Delta^{\bullet}_{k},\,\GG_{Y}^{V}(\infty))
$$ 
is an equivariant weak equivalence.
\end{proposition}
\begin{proof}
Write $F(\infty)=\Hom(X\times\Delta^{\bullet}_{k},\,\GG_{Y}^{V}(\infty))$.  It suffices to show that $M$ acts on $F(\infty)$ by equivariant weak equivalences. Indeed, in this case it follows from \cite[Lemma p.90]{Quillen:Ktheory} that we have an equivariant homotopy fiber sequence 
$$
F(\infty) \to \hocolim_{M}F(\infty) \xrightarrow{\pi} BM.
$$
The map $\pi$ is surjective and as shown in \cite[proof of Lemma 3.1]{GW:Kmodels} $BM$ is contractible, from which the result follows. 

To see that $M$ acts by equivariant weak equivalences we proceed as follows. Let $\alpha\in M$ be an injection. Then $\alpha$ acts via $\alpha^{*}:F(\infty)\to F(\infty)$ where the quotient object $\alpha^{*}[q]$ is given by
$$
\alpha^{*}q(e_{i}) = \begin{cases}
q(e_{j}) & \textrm{if} \,\, \alpha(j) = i,\,\,\textrm{and} \\
0 & \textrm{if}\,\, i\notin \im(\alpha) .               
              \end{cases}
$$
Write $t:\Delta^{1}_{k}\iso\A^{1}_{k}$ for any map sent to $0$ and $1$ under the face maps. An $n$-simplex of $F(\infty)\times \Delta^{1}$ is a pair 
$([q:\mcal{V}^{\infty}_{X\times\Delta^{n}_{k}\times Y}\to \mcal{M}], \delta:[n]\to [1])$.
Associate to the above pair the quotient $[H]=[H([q],\delta)]$ defined by
$$
H([q],\delta) = \begin{cases}
q(e_{j}) & \textrm{if} \,\, \alpha(j) = i,\,\,\textrm{and} \\
\delta^*(t) \cdot q(e_{i}) & \textrm{if}\,\, i\notin \im(j) .               
              \end{cases}
$$
Here $t:\Delta_{k}^{1}\to \A_{k}^{1}$ is viewed as a global section of $\mcal{O}_{\Delta_{k}^{1}}$ and $\delta^*(t)$ is thus viewed as a global section of $\mcal{O}_{X\times \Delta^{n}_{k} \times Y}$ via pullback. This is easily seen to be a well-defined, equivariant  map of simplicial sets and satisfies $H([q],0) = \alpha^{*}[q]$ and $H([q],1) = [q]$ as desired.
\end{proof}

\begin{remark}\label{rem:WGW}
The bivariant $K$-theories constructed above follow Walker's constructions in \cite{Walker:Thomason} in the nonequivariant setting. In \cite{GW:Kmodels} Grayson-Walker use a slightly different construction and \cite{Ostvaer} is written following this version. These are related as follows. Let $V$ be any representation such that $V^{\infty}$ is a complete universe for $G$ (in particular the regular representation $k[G]$ satisfies this property). The $G$-ind-variety $K_{Y,0}^{(n)}$ used in \cite{Ostvaer} is the variety parameterizing $n$-tuples 
$[p_{1}:\mcal{V}_{X\times Y}^{\infty}\onto \mcal{M}_{1}], \ldots, [p_{n}:\mcal{V}_{X\times Y}^{\infty}\onto \mcal{M}_{n}]$ 
of finite and flat quotient objects which are in ``general position'' in the sense that $(p_{1},\ldots,p_{n})^{t}:\mcal{V}_{X\times Y}^{\infty}\to \oplus\mcal{M}_{i}$ is surjective.

An element of $\Hom(X,\,\GG_{\underline{n}_{+}\wedge Y_{+}}^{V}(\infty))$ is a quotient object $[\mcal{V}^{\infty}_{X\times (\underline{n}\times Y)}\onto \mcal{M}]$ which is finite and flat over $X$ and 
satisfies the additional condition that $\mcal{V}^{\infty}_{X}\to (\pi_{X})_{*}\mcal{M}$ remains surjective. As remarked by Walker \cite[p. 219]{Walker:Thomason}, giving such a quotient object naturally yields an $n$-tuple in general position. In other words we have a natural inclusion of $G$-ind-varieties
$G^{V}_{\underline{n}\times Y}(\infty) \hookrightarrow K_{Y,0}^{(n)}$, identifying $\Hom(X,G^{V}_{\underline{n}\times Y}(\infty))$ with $n$-tuples of quotient objects $([p_{i}:\mcal{V}^{\infty}_{X\times Y} \onto \mcal{M}_{i}])$ which are general position and satisfy the extra condition that $\mcal{V}_{X}^{\infty}\onto (\pi_{X})_{*}\mcal{M}_{i}$ remains surjective.
We thus have natural inclusions of equivariant $\Gamma$-spaces
$$
\Hom(X\times\Delta^{\bullet}_{k}, G^{V}_{-\times Y}(\infty)) \hookrightarrow \Hom(X\times\Delta^{\bullet}_{k},K_{Y,0}^{(-)})
$$
which 
induces an equivalence of the associated $K$-theory spectra by \cite[Theorem A.10]{Ostvaer}.
\end{remark}

\begin{proposition}\label{prop:algspecial}
 For quasi-projective $G$-varieties $X$ and $Y$, the equivariant $\Gamma$-spaces
$$
\underline{n}_{+}\mapsto |\mcal{A}_{G}(X,\,\underline{n}_{+}\wedge Y_{+})^{V}|
$$
and 
$$
\underline{n}_{+}\mapsto |\Hom(X\times\Delta^{\bullet}_{k},\,\GG_{\underline{n}_{+}\wedge Y_{+}}^{ V}(\infty))|
$$
are degreewise equivariantly weakly equivalent, and are both special.
\end{proposition}
\begin{proof}
Propositions \ref{prop:alginf} and \ref{prop:Malgwkeq} imply that the $\Gamma_{G}$-space $|\mcal{A}_{G}(X,\,-\wedge Y_{+})^{V}|$ is equivariantly weakly equivalent to the  $\Gamma_{G}$-space $|\Hom(X\times\dk,\,\GG_{- \wedge Y_{+}}^{ V}(\infty))|$. Therefore it suffices to show that the latter is special. By Lemma \ref{lem:Gspc} it suffices to show that
$$
\Hom(X\times\dk,\,\GG_{\underline{n}_{+} \wedge Y}^{ V}(\infty))\to \Hom(X\times\dk,\,\GG_{Y}^{ V}(\infty))^{\times n}
$$
is a $G\times \Sigma_{n}$-equivalence. This follows by observing that
the previous remark allows the argument given in \cite[Lemma 2.2]{GW:Kmodels} to carry over to our 
setting. That is, the maps defined there are equivariant, preserve the additional surjectivity condition and the explicit homotopy written there is 
$G\times\Sigma_{n}$-equivariant and preserves the additional surjectivity condition.
\end{proof}

 We write $\mcal{P}(G;X,Y)$ for the exact category of coherent $G$-modules on $X\times Y$ which are finite and flat over $X$ and write $\mcal{K}(G;X,Y)$ for the associated the $K$-theory spectrum and $\mcal{K}(G;X\times\Delta^{\bullet}_{k},Y)$ for the realization of the 
 simplicial spectrum $d\mapsto \mcal{K}(G;X\times\Delta^{d}_{k},Y)$. Note that when $X$ is smooth and $Y=\spec(k)$ then $\mcal{K}(G;X)\xrightarrow{\wkeq} \mcal{K}(G;X\times\Delta^{\bullet}_{k},\spec(k))$ where $\mcal{K}(G;X)$ is the equivariant algebraic $K$-theory spectrum introduced by Thomason \cite{Thomason:algKgroup}.

\begin{proposition}\label{prop:fix}
Let $X$, $Y$ be quasi-projective $G$-varieties and $H\subseteq G$ a subgroup. There are natural isomorphisms
$$
K_{n}^{H,\A^{1}}(X,Y)= \pi_{n}^{H}\mcal{K}_{G}(X,Y) \iso \pi_{n}\mcal{K}(H;X\times\Delta^{\bullet}_{k},Y).
$$
\end{proposition}
\begin{proof}
By Proposition \ref{prop:algspecial}, we have that $\mcal{K}_{G}(X,Y)$ is equivariantly weakly equivalent to the $G$-spectrum associated to $|\Hom(X\times\Delta^{\bullet}_{k},\,\GG_{-\wedge Y_{+}}^{k[G]}(\infty))|$. Therefore by Lemma \ref{lem:spcfix}, we have  $\pi_{n}^{H}\mcal{K}_{G}(X,Y)=\pi_{n}\mathbf{B}|\Hom(X\times\Delta^{\bullet}_{k},\,\GG_{-\wedge Y_{+}}^{k[G]}(\infty))^{H}|$. Now \cite[Theorems A.4 and A.10]{Ostvaer} show that this last spectrum is naturally weakly equivalent to the spectrum $\mcal{K}(H;X\times\Delta^{\bullet}_{k},Y)$. 
\end{proof}
\begin{remark}
 In particular we have that 
$$
K_{0}^{G,\A^{1}}(X,Y) = \coker(i_{1}^{*}-i_{0}^{*}:K_{0}^{G,\,alg}(X\times\A^{1},Y)\to K_{0}^{G,\,alg}(X,Y)),
$$
where $i_{k}$ is the inclusion at $k\in\A^{1}$ and $K_{0}^{G,\,alg}(X,Y)=\pi_{0}\mcal{K}(G;X,Y)$.
\end{remark}

\subsection{Semi-topological $K$-theory}\label{sub:sst}
We now explain how to construct an equivariant version of the bivariant semi-topological $K$-theory introduced by Walker in \cite[section 2]{Walker:Thomason}. We begin by recalling a construction of Friedlander-Walker.

If $F$ is a presheaf of sets on $\Sch/\C$ and $T$ is a topological space then $F(T)$ is defined as the left Kan extension of $F$ along the functor $\Sch/\C\to Top$ given by $X\mapsto X^{an}$. Explicitly $F(T)$ is the filtered colimit
$$
F(T) = \colim_{T\to U^{an}}F(U)
$$ 
where the colimit is over continuous maps $T\to U^{an}$ with $U$ a (possibly singular) variety. Applying this in particular to the standard topological simplices $\Delta^{n}_{top}$ yields a presheaf of simplicial sets  $n\mapsto F(\Delta^{n}_{top})$. More generally $F$ could be a presheaf of simplicial sets or spectra and we obtain a presheaf of bisimplicial sets or simplicial spectra and write $F(\Delta^{\bullet}_{top})$ for its realization.

Using this construction we obtain the functor $I\to G\sSet$,
$$
\underline{n}\mapsto \Hom_{\Sch/\C}(X\times\dtop,\,\GG_{Y}^{V}(n)).
$$ 

Define the presheaf of $G$-simplicial sets $\AA^{\sst}_{G}(-,\,Y)^{V}$ by
$$
\AA^{\sst}_{G}(X,\,Y)^{V} = \hocolim_{I}\Hom_{\Sch/\C}(X\times\dtop,\,\GG_{Y}^{V}(-)).
$$

\begin{remark}\label{rem:mor}
 We make use of the simplicial mapping spaces $\Hom(X\times\dtop, \, Y)$ rather than the topological spaces $\Mor(X,Y)$, which are used in \cite{Walker:Thomason}. Shortly after \cite{Walker:Thomason} was written, Friedlander-Walker developed techniques, especially in \cite{FW:ratisos}, which make the spaces $\Hom(X\times\dtop, \, Y)$ easier to work with than the conceptually attractive space $\Mor(X,\,Y)$. When $X$ and $Y$ are projective, \cite[Corollary 4.3]{FW:sstfct} shows that we have a natural isomorphism of simplicial sets $\Hom(X\times\dtop,\,Y) \iso\sing_{\bullet}\Mor(X,\,Y)$. 
\end{remark}

\begin{remark}\label{rem:hocol}
For $G$ trivial, Walker \cite[section 3.1]{Walker:Thomason} introduces $\AA^{\sst}_{G}(X,\,Y)^{V}$ (or rather $\AA^{\semi}_{G}(X,\,Y)^{V}$) as the nerve of a topological category. As noted in Lemma 3.2  of loc. cit., this agrees with the construction as above.
\end{remark}

It will also be convenient to extend this definition to pairs $(Y,y_{0})$ consisting of a quasi-projective $G$-variety together with an invariant basepoint $y_{0}\in Y$. Note that $\GG_{y_{0}}^{V}(n)\subseteq \GG_{Y}^{V}(n)$ and this inclusion is functorial as well in the variable $n$. Define 
$$
\GG_{(Y,y_{0})}^{V}(n)(X) = \GG_{Y}^{V}(n)(X)/\GG_{y_{0}}^{V}(n)(X)
$$ 
and
$$
\AA^{\sst}_{G}(X,\,(Y,y_{0}))^{V} = \hocolim_{I}\Hom_{\Sch/\C}(X\times\dtop,\,\GG_{(Y,y_{0})}^{V}(-)).
$$

Note that $\AA^{sst}_{G}(X,\,(Y_{+},*))^{V} = \AA^{sst}_{G}(X,\,Y)^{V}$. Usually we omit the base-point from the notation when the context makes it clear what is meant. The $G$-simplicial sets $\AA^{sst}_{G}(X,\,Y)^{V}$ are clearly contravariantly natural in the first variable and are covariantly natural in the second variable by Lemma \ref{lem:covariant}. By the covariant naturality in the second variable we obtain equivariant $\Gamma$-spaces
$$
\underline{n}_{+}\mapsto |\AA^{sst}_{G}(X,\,\underline{n}_{+}\wedge Y_{+})^{V}|.
$$

We are most interested in the case where $V=\C[G]$.
Recall that an equivariant $\Gamma$-space $\mcal{A}$ gives rise to a $G$-spectrum (see Lemma \ref{lem:Gspc}) which we write $\mathbb{S}\mcal{A}$.
\begin{definition}\label{defKsst}
Let $X$ and $Y$ be quasi-projective $G$-varieties. The \textit{bivariant semi-topological $K$-theory spectrum} $\KK^{sst}_{G}(X,Y)$ is the $G$-spectrum defined by
$$
\KK^{sst}_{G}(X,Y)=\S |\AA^\sst(X,\,-\wedge Y_+)^{\C[G]}|.
$$
The \textit{bivariant semi-topological $K$-theory groups}  defined to be the homotopy groups
$$
K^{G,sst}_{*}(X,Y) = \pi^{G}_{*}\mcal{K}^{sst}_{G}(X,Y)
$$ 
of this spectrum.
\end{definition}
\begin{remark}
The spectra $\KK^{sst}_{G}(X,Y)$ are evidently contravariantly natural in the first variable and are covariantly natural in the second variable by Lemma \ref{lem:covariant}.
\end{remark}

\begin{lemma}\label{lem:delsst}
For any quasi-projective $G$-varieties $X$,$Y$ we have a natural equivariant weak equivalence
$$
\mcal{K}_{G}^{sst}(X,Y)\xrightarrow{\wkeq} \mcal{K}_{G}(X\times\Delta^{\bullet}_{top},Y) 
$$
\end{lemma}
\begin{proof}
Let $V$ denote the regular representation. Note that if $\mcal{A}= \colim_{i}\mcal{A}_{i}$ is a filtered colimit of equivariant $\Gamma$-spaces then $\colim_{i}\mathbb{S}\mcal{A}_{i} \iso \mathbb{S}\mcal{A}$. Therefore we have
\begin{align*}
\mcal{K}_{G}(X\times \Delta^{n}_{top},Y) & \iso \mathbb{S}|\mcal{A}_{G}(X\times\Delta^{n}_{top}, - \wedge Y_{+})^{V}|\\
& = \mathbb{S}|\hocolim_{I}\Hom(X\times\Delta^{\bullet}_{\C}\times\Delta^{n}_{top},\mcal{G}^{V}_{-\wedge Y_{+}}(-))|.
\end{align*}
The projection $\Delta^{\bullet}_{\C}\to \Delta^{0}_{\C}$ induces a natural map
$\mcal{A}^{sst}_{G}(X,Y)^{V}_{n} \to \mcal{A}_{G}(X\times\Delta^{n}_{top},Y)^{V}$
and allowing $n$ to vary yields a natural map of bisimplicial sets. Taking realizations and associated spectra yields the map
$$
\mathbb{S}|\mcal{A}^{sst}_{G}(X,Y)^{V}| \to \mathbb{S}|\mcal{A}_{G}(X\times\Delta^{\bullet}_{top}, Y)^{V}|
$$ 
and it remains to show that this is an equivariant weak equivalence. Note that
$\mcal{A}_{G}(X\times\Delta^{n}_{top},-\wedge Y_{+})^{V}$ is a special equivariant $\Gamma$-space, being a filtered colimit of such. By Lemma \ref{lem:spcfix}, it suffices to show that 
$$
\mathbf{B}|\hocolim_{I}\Hom(X\times\Delta^{\bullet}_{top},\mcal{G}^{V}_{ Y})^{H}| \to  \mathbf{B}|\hocolim_{I}\Hom(X\times\Delta^{\bullet}_{\C}\times\Delta^{\bullet}_{top},\mcal{G}^{V}_{ Y})^{H}|
$$
is an equivalence of spectra for any subgroup $H\subseteq G$.  The map of associated infinite loop spaces 
$$
\Omega B|\hocolim_{I}\Hom(X\times\Delta^{\bullet}_{top},\mcal{G}^{V}_{ Y})^{H}| \to  
\Omega B|\hocolim_{I}\Hom(X\times\Delta^{\bullet}_{\C}\times\Delta^{\bullet}_{top},\mcal{G}^{V}_{ Y})^{H}|
$$
is a homology equivalence by \cite[Lemma 1.2]{FW:compK} and the result follows.
\end{proof}

The above together with the map induced by the projection $\Delta^{d}_{top}\to \Delta^{0}_{top}$ give us the natural transformations
\begin{equation}\label{eqn:algsst}
\mcal{K}_{G}(X, Y) \to \mcal{K}_{G}(X\times\Delta^{\bullet}_{top},Y)\xleftarrow{\wkeq}\mcal{K}_{G}^{sst}(X,Y).
\end{equation}

Write 
$$
\Hom_{\Sch/\C}(X,\,\GG_{Y}^{V}(\infty))=\colim_{n}\Hom_{\Sch/\C}(X ,\,\GG_{Y}^{V}(n))
$$
where the colimit is over the standard inclusions $\underline{n}\subseteq \underline{n+1}$ given by $i\mapsto i$ and the transition maps are induced by precomposition with the canonical surjections $\mcal{V}^{n+1} \to \mcal{V}^{n}$. Thus an element is a quotient object $[\mcal{V}^{\infty}_{X \times Y} \xrightarrow{p} \mcal{M}]$ where $\mcal{V}^{\infty}_{X\times Y} \to \mcal{M}$ factors as $\mcal{V}^{\infty}_{X\times U\times Y} \to \mcal{V}^{N}_{X\times U \times Y}\to \mcal{M}$ for some $m$. If $j:\mathbb{N}\to \mathbb{N}$ is an injection we have an induced map given by $j^{*}[\mcal{V}^{\infty}_{X \times Y} \xrightarrow{p} \mcal{M}] = [p\tilde{j}^{*}:\mcal{V}^{\infty}_{X \times Y} \xrightarrow{}\mcal{M}]$. Similarly for pointed varieties $(Y,y_{0})$. 

\begin{remark}\label{rem:twomodels}
We will see in Corollary \ref{niceKsst} below that  $\Hom(X\times\dtop,\,\GG_{Y}^{V}(\infty))$ gives rise to a $G$-spectrum that is weakly equivalent to $\mcal{K}^{sst}_{G}(X,Y)$ which was obtained from $\hocolim_{I}\Hom(X\times\dtop,\,\GG_{Y}^{V}(-))$. Because the first model is defined using a filtered colimit it is in many ways easier to work with and indeed we rely on this model to deduce many properties of our bivariant $K$-theory spectra. However the second model is better suited for the 
pairings and operations appearing in Section \ref{sec:pair} which is crucial for this paper and so it is crucial to have both models available. 
In fact, there is a third model, namely the one provided by
Lemma \ref{lem:delsst}, which allows for a convenient comparison map
to equivariant algebraic $K$-theory as used in section \ref{sec:algThom}.  
\end{remark}

 Recall that we write $\tilde{I}$ for the category whose objects are $\underline{n}$ and the set $\mathbb{N}$, which we denote by $\infty$, and whose morphisms are injections. Let $M\subseteq \tilde{I}$ be the full subcategory containing the object $\infty$. We also write $M$ for  $\Hom_{M}(\infty,\infty)$, which is the monoid (under composition) of injective maps $\mathbb{N}\to \mathbb{N}$.
\begin{proposition}\label{prop:sstinf}
 Let $X$ and $(Y,y_{0})$ be quasi-projective complex $G$-varieties. The natural maps
\begin{multline*}
\hocolim_{I}\Hom(X\times\dtop,\,\GG_{Y}^{V}(-)) \to \hocolim_{\tilde{I}} \Hom(X\times\dtop,\,\GG_{Y}^{V}(-)) \\
\leftarrow \hocolim_{M}\Hom(X\times\dtop,\,\GG_{Y}^{V}(\infty))\leftarrow \Hom(X\times\dtop,\,\GG_{Y}^{V}(\infty))
\end{multline*}
are equivariant weak equivalences.
\end{proposition}
\begin{proof}
The proof that the first two arrows are weak equivalences is exactly as for Proposition \ref{prop:alginf}. The proof of the last one is a variant of the argument in the proof of Proposition \ref{prop:Malgwkeq}. 
That is, it suffices to show that $M$ acts by equivariant weak equivalences on $\Hom(X\times\dtop,\,\GG_{Y}^{V}(\infty))$. We write $F(\infty)=\Hom(X\times\dtop,\,\GG_{Y}^{V}(\infty))$. 

Let $\alpha\in M$ be an injection. Then $\alpha$ acts via $\alpha^{*}:F(\infty)\to F(\infty)$ where the quotient object $\alpha^{*}[q]$ is given by
$$
\alpha^{*}q(e_{i}) = \begin{cases}
q(e_{j}) & \textrm{if} \,\, \alpha(j) = i,\,\,\textrm{and} \\
0 & \textrm{if}\,\, i\notin \im(\alpha) .               
              \end{cases}
$$

Let $g:\Delta^{1}_{top}\to (\A^{1})^{an}$ be a map  which sends $0$ to $0$ and $1$ to $1$. We define a simplicial homotopy $H:F(\infty)\times \Delta^{1} \to F(\infty)$ between $\alpha^{*}$ and $id$ as follows. An $n$-simplex of $F(\infty)\times \Delta^{1}$ is represented by a triple 
$$
(f:\Delta^{n}_{top}\to U^{an},[q:\mcal{O}^{\infty}_{X\times U \times Y}\to \mcal{M}], \delta:[n]\to [1]).
$$
 We denote by $\mcal{M}'$ the pullback of $\mcal{M}$ to $X\times U\times \A^{1} \times Y$ and associate to the above triple the element $(f\times \delta^{*}(g):\Delta^{n}_{top}\to (U\times \A^{1})^{an}, H)\in F(\infty)$ where $[H=H(f,[q],\delta):\mcal{O}^{\infty}_{X\times U\times \A^{1}\times Y} \to \mcal{M}']$ is 
$$
H(f,[q],\delta) = \begin{cases}
q(e_{j}) & \textrm{if} \,\, \alpha(j) = i,\,\,\textrm{and} \\
t \cdot q(e_{i}) & \textrm{if}\,\, i\notin \im(j) .               
              \end{cases}
$$
Here $t=id:\A^{1}\to \A^{1}$ is viewed as a global section of $\mcal{O}_{\A^{1}}$ and hence of $\mcal{O}_{X\times U\times \A^{1} \times Y}$ via pullback. This is easily seen to be a well-defined, equivariant  map of simplicial sets and satisfies $H([q],0) = \alpha^{*}[q]$ and $H([q],1) = [q]$ as desired.
\end{proof}

\begin{proposition}\label{prop:sstspecial}
 For quasi-projective $G$-varieties $X$ and $Y$ the equivariant $\Gamma$-spaces
$$
\underline{n}_{+}\mapsto |\mcal{A}^{sst}_{G}(X,\,\underline{n}_{+}\wedge Y_{+})^{V}|
$$
and 
$$
\underline{n}_{+}\mapsto |\Hom(X\times\dtop,\,\GG_{\underline{n}_{+}\wedge Y}^{ V}(\infty))|
$$
are degreewise equivariantly weakly equivalent, and they are both special.
\end{proposition}
\begin{proof}
Using Proposition \ref{prop:sstinf}, the proof is similar to the one of
Proposition \ref{prop:algspecial}.
\end{proof}

\begin{corollary}\label{niceKsst}
 For quasi-projective $G$-varieties $X$ and $Y$
there are equivariant weak equivalences of $G$-spectra 
$$\mcal{K}^{sst}_{G}(X,Y) \simeq  
\mathbb{S} |\Hom(X\times\dtop,\,\GG_{-\wedge Y}^{ V}(\infty))| 
$$
and hence equivariant weak equivalences of associated infinite loop spaces
$$
\Omega^{\infty}\mcal{K}^{sst}_{G}(X,Y) \simeq
\Omega B|\Hom(X\times\dtop,\,\GG_{Y}^{ V}(\infty))|.
$$
\end{corollary}
\begin{proof}
The first weak equivalence follows from Proposition 
\ref{prop:sstspecial} and the second one from
Lemma \ref{lem:Gspc}.
\end{proof}

We finish this section by showing that the group $K_{0}^{G,\,sst}(X,Y)$ has the expected description in terms of certain coherent $G$-modules modulo algebraic equivalence. Let $\mcal{M}_{1}$ and $\mcal{M}_{2}$ be two coherent $G$-modules on $X\times Y$ which are finite and flat over $X$. We say that $\mcal{M}_{1}$ and $\mcal{M}_{2}$ are  \textit{algebraically equivalent} if there is a smooth, connected curve $C$ (without $G$-action), two closed points $c_{1},c_{2}\in C$, a coherent $G$-module $\mcal{N}$ on $X\times C \times Y$ which is finite and flat over $X\times C$ such that $\iota_{k}^{*}\mcal{N} = \mcal{M}_{k}$, where $\iota_{k}$ is inclusion $X\times \{c_{k}\}\times Y \subseteq X\times C \times Y$. Write $\sim_{alg}$ for this equivalence relation. We write $K_{*}^{G,\,alg}(X,Y)$ for the algebraic $K$-theory of the exact category $\mcal{P}_{G}(X,Y)$ of coherent $G$-modules on $X\times Y$ which are finite and flat over $X$. 

\begin{theorem}\label{thm:algeq}
 Let $X$ and $Y$ be quasi-projective $G$-varieties. We have an isomorphism
$$
K_{0}^{G,\,sst}(X,Y) \iso K_{0}^{G,\,alg}(X,Y)/\sim_{alg},
$$
which is contravariantly natural in the first variable and covariantly in the second.
\end{theorem}
\begin{proof}
In this proof we let $V=\C[G]$. It follows from Corollary \ref{niceKsst} that 
$$
K_{0}^{G,\,sst}(X,Y)= [\pi_{0}\Hom(X\times\dtop,\,\GG_{Y}^{V}(\infty))^{G}]^{+}.
$$
Using Lemma \ref{lem:pi0} below together with the same argument as in \cite[Proposition 2.10]{FW:sstfct} shows that  $\pi_{0}\Hom(X\times\dtop,\,\GG_{Y}^{V}(\infty))^{G}$ consists of equivalence classes of coherent $G$-modules $\mcal{M}$ on $X\times Y$ which are finite and flat over $X$ and admit a surjection of the form $\mcal{V}^{N}_{X\times Y}\onto \mcal{M}$ (the equivalence class of $\mcal{M}$ is independent of the surjection). 
Here  $\mcal{M}_{1}$ and $\mcal{M}_{2}$ are equivalent if there is a smooth, connected curve $C$, two closed points $c_{1},c_{2}\in C$, and a coherent $G$-module $\mcal{N}$ on $X\times X\times Y$ such that $\mcal{M}_{i} =\mcal{N}|_{c_{i}}$ where  $\mcal{N}$ is  finite and flat over $X\times C$ and it admits a surjection of the form $\mcal{V}_{X\times C\times  Y}^{K}\onto \mcal{N}$.
The monoid structure on $\pi_{0}\Hom(X\times\dtop,\,\GG_{Y}^{V}(\infty))^{G}$ induced by the $H$-space structure is given by direct sum of modules. We thus have a natural map
$$
K_{0}^{G,\,sst}(X,Y) = [\pi_{0}\Hom(X\times\dtop,\,\GG_{Y}^{V}(\infty))^{G}]^{+} \to K_{0}^{G}(X,Y)/\sim_{alg}.
$$

The argument given in \cite[Proposition 2.12]{FW:sstfct} applies here to show that this map is an isomorphism. The claim regarding the functorialities is easily verified.
\end{proof}

\begin{lemma}\label{lem:pi0}
 Let $\mcal{F}$ be a presheaf of sets on $\Sch/\C$. Then $\pi_{0}\mcal{F}(\Delta^{\bullet}_{top}) = \mcal{F}(\C)/\sim$ where $\sim$ is the equivalence relation generated by $x_{1}\sim x_{2}$ if there is a smooth connected curve $C$, an element $z\in \mcal{F}(C)$, two closed points $c_{0}, c_{1}\in C$ such that $x_{k}= \epsilon_{k}^{*}z$ where $\epsilon_{k}:c_{i}\to C$ is the inclusion.
\end{lemma}
\begin{proof}
 Write $\iota_{0},\iota_{1}:\Delta^{0}\to \Delta^{1}$ for the inclusions at $0$ and $1$. If $x_{1}\sim x_{2}$ then they are in the same path component of $\pi_{0}\mcal{F}(\Delta^{\bullet}_{top})$. Suppose that $x_{1}$ and $x_{2}$ lie in the same path component. This means that there is an $(f:\Delta^{1}_{top}\to T^{an}, y\in \mcal{F}(T))$ such that $\iota_{k}^{*}y= x_{i}$ for some $T$. Since we may assume that $T$ is connected, there is a sequence smooth connected curves $C_{1},\ldots C_{r}$, maps $g_{i}:C_{i}\to T$, and points $c_{i}, d_{i}\in C_{i}$ such that $g_{i}(d_{i}) = g_{i+1}(c_{i+1})$ with $g_{1}(c_{1})=f(0)$ and $g_{r}(d_{r})=f(1)$, from which the lemma follows.
\end{proof}

\subsection{Topological $K$-theory}\label{sec:top}
We now introduce the model for bivariant equivariant $K$-theory with which we work. We restrict attention to $G$-$CW$-complexes. 
For a $G$-$CW$ complex $T$ write $\mcal{C}(T)$ for the nonunital topological ring of all continuous complex valued functions on $T$. When $(T,t_{0})$ is a based $G$-$CW$ complex $T$ write $\mcal{C}_{0}(T)$ for the nonunital topological ring of continuous complex valued functions on $T$ which vanish at the base-point. Note when $T$ is unbased, that $\mcal{C}(T) = \mcal{C}_{0}(T_{+})$.

Complex conjugation defines a natural involution on $\mcal{C}_{0}(T)$. When $T$ is compact this makes $\mcal{C}_{0}(T)$ into a $C^{*}$-algebra. Additionally the $G$-action on $T$ induces a $G$-action on $\mcal{C}_{0}(T)$, where $G$ acts by continuous $\C$-algebra homomorphisms.

If $V$ is a unitary complex $G$-representation then $\End_{\C}(V^{\oplus n})$ is also a $C^{*}$-algebra and $G$ acts on it via $C^{*}$-algebra homomorphisms. For a pointed $G$-$CW$-complex $T$ let $\underline{\Hom}_{*}(\mcal{C}_{0}(T),\,\End_{\C}(V^{\oplus n}))$ be the space of involution-preserving, continuous, $\C$-algebra homomorphisms ($*$-map for short). We write
$$
\mcal{F}_{T}^{V}(n) = \underline{\Hom}_{*}(\mcal{C}_{0}(T),\,\End_{\C}(V^{\oplus n}))
$$
for this space. 
We have that $\Map(W,\End_{\C}(V^{\oplus n})) \iso \mcal{C}_{0}(W)\otimes\End_{\C}(V^{\oplus n})$, which together with adjointness gives
$$
\Map(W,\mcal{F}_{T}^{V}(n)) \iso \underline{\Hom_{*}}(\mcal{C}_{0}(T),\mcal{C}_{0}(W)\otimes\End_{\C}(V^{\oplus n})).
$$
The space $\mcal{F}_{T}^{V}(n)=\underline{\Hom}_{*}(\mcal{C}_{0}(T),\,\End_{\C}(V^{\oplus n}))$ has a $G$-action given by the usual formula. A $*$-map $f:\mcal{C}_{0}(T)\to \End_{\C}(V^{\oplus n})$ factors as a composition of $*$-maps $\mcal{C}_{0}(T)\to \mcal{C}_{0}(\{t_{1},\ldots, t_{k}\})\to \End_{\C}(V^{n}))$ where $\{t_{1},\ldots, t_{k}\}\subseteq T$  is a finite set of points. Thus a point of $\mcal{F}_{T}^{V}(n)$ is identified with a finite (unordered) list of points $t_{1},\ldots, t_{k}$ of $T$ together with a list of pairwise orthogonal subspaces $W_{1},\ldots, W_{k}$ of $V^{\oplus n}$. Then $g\in G$ acts by $g\cdot(t_{1},\ldots, t_{k}; W_{1},\ldots, W_{k}) = (g\cdot t_{1},\ldots g\cdot t_{k};g\cdot W_{1},\ldots, g\cdot W_{k})$. Thus for $H\subseteq G$, an $H$-invariant point of $\mcal{F}_{T}^{V}(n)$ is specified by an $H$-set of points $\{t_{1},\ldots, t_{k}\}$ of $T$ a sub-$H$-space $W\subseteq V^{\oplus n}$ and a vector space decomposition $W = W_{1}\perp \cdots \perp W_{k}$ such that $h\cdot W_{
i} = W_{j}$ for some $j$.

Given an injection $\alpha:\underline{m}\hookrightarrow \underline{n}$ write $\tilde{\alpha}:\End_{\C}(V^{\oplus m}) \to \End_{\C}(V^{\oplus n})$ for the map 
$$
\tilde{\alpha}(f) = \alpha_{*}f\alpha^{*}.
$$
Let $A$, $B$ be based $G$-$CW$ complexes. Using the maps $\tilde{\alpha}$, the assignment
$$
\underline{n}\mapsto \Hom_{cts*}(A\wedge\Delta^{\bullet}_{top\,+},\mcal{F}_{B}^{V}(n)),
$$
defines a functor $I\to G\sSet$.
Now define for based $G$-$CW$-complexes $(A,a_{0})$ and $(B,b_{0})$
$$
\mcal{A}^{top}_{G}(A, B)^{V} = \hocolim_{I}\Hom_{cts*}(A\wedge\Delta^{\bullet}_{top\,+},\mcal{F}_{B}^{V}(-)).
$$
Note that we omit the base point from the notation, leaving it implicit. For each $V$ we now have $\Gamma$-spaces
$$
\underline{n}_{+}\mapsto |\mcal{A}_{G}^{top}(A,\underline{n}_{+}\wedge B)^{V}|.
$$
\begin{remark}
We will work mostly with unpointed spaces when comparing with the algebraic theories. In this context, if $A$ is unpointed it is convenient to still write $|\mcal{A}_{G}^{top}(A, B)^{V}|$, which is to be interpreted as $|\mcal{A}_{G}^{top}(A_{+}, B)^{V}|$, and similarly for $B$ unpointed. This abuse of notation should not lead to confusion.
\end{remark}

Again the main case of interest is $V=\C[G]$. 

\begin{definition}
 Let $A$, $B$ be based $G$-$CW$ complexes. Define 
$$
\bu^{\mf{c}}_{G}(A,\,B)
$$ 
to be the $G$-spectrum associated to the equivariant $\Gamma$-space  $|\mcal{A}^{top}(A,-\wedge B)^{\C[G]}|$. Write 
$$
\bu^{G,\mf{c}}_{*}(A,B) = \pi^{G}_{*}\bu^{\mf{c}}_{G}(A,B)
$$
for the homotopy groups of this spectrum. More generally for a subgroup $H\subseteq G$, write $\bu^{H,\mf{c}}_{*}(A,B) = \pi^{H}_{*}\bu^{\mf{c}}_{G}(A,B)$.
\end{definition}

\begin{remark}
Let $V$ be a real representation. Combining Corollary \ref{cor:cc} and Proposition \ref{prop:conn} we have  that $\bu^{G,\mf{c}}_{V}(A_{+},S^{0}) = KU^{-V}_{G}(A_{+})$.
\end{remark}

As before, it is convenient to introduce a variant of the construction above. We define 
$$
\Hom_{cts*}(A\wedge\Delta^{\bullet}_{top\,+},\,\mcal{F}_{B}^{V}(\infty)) = \colim_{n}\Hom_{cts*}(A\wedge\Delta^{\bullet}_{top\,+},\,\mcal{F}_{B}^{V}(n))
$$ 
where the colimit is over the standard inclusions
 $\underline{n}\subseteq \underline{n+1}$ given by $i\mapsto i$. 
As in the algebraic and semi-topological cases we have the following (see the paragraph preceding Proposition \ref{prop:alginf} for a reminder on the indexing categories used below). 
\begin{proposition}\label{prop:topinf}
For based $G$-$CW$-complexes $A$, $B$, the natural maps
\begin{multline*}
\hocolim_{I}\Hom(A\wedge\Delta^{\bullet}_{top\,+},\,\mcal{F}_{B}^{V}(-)) \to \hocolim_{\tilde{I}} \Hom(A\wedge\Delta^{\bullet}_{top\,+},\,\mcal{F}_{B}^{V}(-)) \\
\leftarrow \holim_{M}\Hom(A\wedge\Delta^{\bullet}_{top\,+},\,\mcal{F}_{B}^{V}(\infty))\leftarrow \Hom(A\wedge\Delta^{\bullet}_{top\,+},\,\mcal{F}_{B}^{V}(\infty)).
\end{multline*}
are equivariant weak equivalences.
\end{proposition}
\begin{proof}
 This is similar to the proofs of Propositions \ref{prop:alginf} and \ref{prop:sstinf}.
\end{proof}

\begin{proposition}\label{prop:topspecial}
 For based $G$-$CW$-complexes, the equivariant $\Gamma$-space
$$
\underline{n}_{+}\mapsto |\mcal{A}^{top}_{G}(A,\underline{n}_{+}\wedge B)^{V}|
$$
and 
$$
\underline{n}_{+}\mapsto |\Hom(A\times\dtop,\,\GG_{\underline{n}_{+}\wedge B}^{ V}(\infty))|
$$
are degreewise equivariantly weakly equivalent, and they are both special.
\end{proposition}
\begin{proof}
Using Proposition \ref{prop:topinf} and proceeding once more
as in the proof of Proposition \ref{prop:algspecial}, it suffices to show the second equivariant $\Gamma$-space is special. The argument for this is again an adaptation of \cite[Lemma 2.2]{GW:Kmodels} to our present context. 
For any integer $M>0$, we will define a map 
$$
\eta:\mcal{F}_{B}^{V}(M)^{\times n}\to \mcal{F}_{\underline{n}_{+}\wedge B}^{V}(nM).
$$
Composing $\eta$ and the inclusions $\epsilon:\mcal{F}_{\underline{n}_{+}\wedge B}^{V}(-) \to \mcal{F}_{B}^{V}(-)^{\times n}$  yields the maps
$$
\mcal{F}_{\underline{n}_{+}\wedge B}^{V}(M)\to \mcal{F}_{\underline{n}_{+}\wedge B}^{V}(nM) \,\,\textrm{and} \,\, \mcal{F}_{B}^{V}(M)^{\times n}\to \mcal{F}_{ B}^{V}(nM)^{\times n},
$$
and the result follows by showing these are $G\times \Sigma_{n}$-equivariantly homotopic to the standard inclusions.

We may identify $\mcal{F}_{\underline{n}_{+}\wedge B}^{V}(M)$ with the subspace of $\mcal{F}_{B}^{V}(M)^{\times n}$ 
consisting of those $n$-tuples $(f_{i})$ of $*$-maps which satisfy $f_{i}(a)f_{j}(b) = 0$ for all $a$,$b\in \mcal{C}_{0}(B)$.
For an injection $\beta:\underline{p}\hookrightarrow \underline{q}$, recall that we write $\tilde{\beta}:\End_{\C}(V^{\oplus p})\to\End_{\C}(V^{\oplus q})$ for the induced map $\psi\mapsto \beta_{*}\psi\beta^{*}$.
Let $\alpha_{i}:\underline{M} \hookrightarrow \underline{nM}$ be the injection $k\mapsto (i-1)M+k$. 
Define $\eta$ by sending the tuple $(f_{i})$ of $*$-maps to the tuple  $(\tilde{\alpha_{i}}f_{i})$ which lies in $\mcal{F}_{\underline{n}_{+}\wedge B}^{V}(nM)$ as desired.

Define the $G\times \Sigma_{n}$-equivariant homotopy  $\mcal{F}_{\underline{n}_{+}\wedge B}^{V}(M)\times I\to \mcal{F}_{\underline{n}_{+}\wedge B}^{V}(nM)$ between $\eta\epsilon$ and the standard inclusion by $((f_{i}), t) \mapsto (t\cdot\tilde{\alpha_{1}}f_{i} + (1-t)\cdot\tilde{\alpha_{i}}f_{i})$. 
Similarly one sees that $\epsilon\eta$ is equivariantly homotopic to the standard inclusion.
\end{proof}

As Walker observes in the nonequivariant case, if $B$ is not connected then for $A$ pointed and connected 
this is not a reasonable spectrum to work with. For example if $B = S^{0}$, $A$ is pointed and connected, 
and $V=\C[G]$ then for any subgroup $K\subseteq G$ the space $|\mcal{A}_{G}^{top}(A,\underline{n}_{+})^{V}|^{K}$ is equivalent to the space of $n$-tuples of $K$-bundles on $A$, each of which 
has rank zero at the base-point. Thus the equivariant $\Gamma$-space  $|\mcal{A}_{G}^{top}(A,\underline{n}_{+}\wedge S^{0})^{V}|$ is contractible. As in \cite{Walker:Thomason} we can remedy this by  replacing $B$ with its suspensions.  To do this we view $|\mcal{A}_{G}^{top}(A,-\wedge B)^{V}|$ as a $\mcal{W}_{G}$-space and consider the associated spectrum.
\begin{definition}
 Let $A$, $B$ be based $G$-$CW$ complexes. Define 
$$
\bu_{G}(A,\,B)
$$ 
to be the $G$-spectrum associated to the $\mcal{W}_{G}$-space  $Y\mapsto |\mcal{A}^{top}(A,Y\wedge B)^{\C[G]}|$. Write 
$$
\bu^{G}_{*}(A,B) = \pi^{G}_{*}\bu_{G}(A,B).
$$
for the homotopy groups of this spectrum. More generally for a subgroup $H\subseteq G$, write $\bu^{H}_{*}(A,B) = \pi^{H}_{*}\bu_{G}(A,B)$
\end{definition}

Recall from Section \ref{sec:pre} that to an equivariant $\Gamma$-space $Y$ we naturally associate a $\mcal{W}_{G}$-space $\widehat{Y}$, which may be described by the formula 
$$
\widehat{Y}(X) = \hocolim_{S\to X}PY(S),
$$
where $PY$ is the $\Gamma_{G}$-space associated to $Y$, and $S\to X$ is an object of the overcategory $(\Gamma_{G}\downarrow X)$. If $X$ is a $\mcal{W}_{G}$-space and $Y$ is the equivariant $\Gamma$-space obtained from $X$ by restriction then we have a natural map $\widehat{Y}\to X$. We thus have a natural map of spectra $\mathbb{S}X \to \mathbb{U}X$. In particular we have a natural map of spectra
\begin{equation}\label{eqn:c2nc}
\bu^{\mf{c}}_{G}(A,\,B) \to \bu_{G}(A,\,B).
\end{equation}

\begin{proposition}[c.f. {\cite[Theorem 3.14]{Walker:Thomason}}]\label{prop:ev}
Let $B$ be a based $G$-$CW$-complex. For a based $G$-$CW$-complex $X$ there is a natural equivariant equivalence 
$$
\hocolim_{S\to X}|\mcal{A}_{G}^{top}(S^{0},S\wedge B)^{V}|\xrightarrow{\wkeq}|\mcal{A}_{G}^{top}(S^{0},X\wedge B)^{V}|.
$$
In particular, $\bu^{\mf{c}}_{G}(S^{0},\,B)$ and the spectrum 
$\{|\mcal{A}_{G}^{top}(S^{0},S^{W}\wedge B)^{\C[G]}|\}$ are equivariantly equivalent.
\end{proposition}
\begin{proof}
Since the indexing category  $(\Gamma_{G}\downarrow X)$ is filtered, we have that
$$
\hocolim_{S\to X}|\mcal{A}_{G}^{top}(S^{0},-\wedge B)^{V}| \xrightarrow{\wkeq} \colim_{S\to X}|\mcal{A}_{G}^{top}(S^{0},S\wedge B)^{V}|
$$ 
is an equivariant weak equivalence. 
It thus suffices to show that  
$$
\colim_{S\to X}|\mcal{A}_{G}^{top}(S^{0},S\wedge B)^{V}| \xrightarrow{\iso} |\mcal{A}_{G}^{top}(S^{0},X\wedge B)^{V}|
$$ 
is an isomorphism. For this, it suffices to show that $\colim_{S\to X}\mcal{F}_{S\wedge B}^{V}(n) \xrightarrow{\iso} \mcal{F}^{V}_{X\wedge B}(n)$ is an isomorphism for any $n$. This follows by an argument similar to the argument given in \cite[Theorem 3.14]{Walker:Thomason}, which we briefly sketch. That this map is bijective follows from the fact that for a based $G$-$CW$-complex $W$, a $*$-map $\mcal{C}_{0}(W)\to \End_{\C}(V^{n})$ factors as $\mcal{C}_{0}(W)\to \mcal{C}_{0}(S) \hookrightarrow \End_{\C}(V^{n})$ where $S\subseteq W$ is a based finite $G$-set. 
The map
$\colim_{S\to X}\mcal{F}_{S\wedge B}^{V}(n) \to \mcal{F}^{V}_{X\wedge B}(n)$ is the colimit of maps $\colim_{S\to W}\mcal{F}^{V}_{S\wedge T}(n)\to \mcal{F}^{V}_{W\wedge T}(n)$ where $T\subseteq B$ and $W\subseteq X$ range over all compact subspaces. 
When $T$, $W$ are compact we have that $\colim_{S\to W}\mcal{F}^{V}_{S\wedge T}(n)\to \mcal{F}^{V}_{W\wedge T}(n)$ is closed since $\mcal{F}^{V}_{C}(n)$ is compact whenever $C$ is compact. It is therefore an equivariant homeomorphism and we are done.
\end{proof}

We have  natural equivariant maps 
\begin{equation}\label{eqn:adjmaps}
|\mcal{A}_{G}^{top}(A,\,B)^{V}|\xrightarrow{}\Map(A,\hocolim_{I}\mcal{F}_{B}^{V}(-))\xleftarrow{}\Map(A,\,|\mcal{A}_{G}^{top}(S^{0},\,B)^{V}|)
\end{equation}
obtained as the composite
\begin{align*} 
|\mcal{A}_{G}^{top}(A,B)^{V}| & = |\hocolim_{I}\sing_{\bullet}\Map(A,\mcal{F}_{B}^{V}(-))| \\
 & \xrightarrow{\iso} \hocolim_{I}|\sing_{\bullet}\Map(A,\mcal{F}_{B}^{V}(-))| 
  \to \hocolim_{I}\Map(A,\mcal{F}_{B}^{V}(-)) \\
& \to \Map(A,\hocolim_{I}\mcal{F}_{B}^{V}(-)) 
 \xleftarrow{\wkeq} \Map(A,\hocolim_{I}|\sing_{\bullet}\mcal{F}_{B}^{V}(-)|) \\
& \xleftarrow{\iso}\Map(A,\,|\mcal{A}_{G}^{top}(S^{0},\,B)^{V}|),
\end{align*}
where the two displayed isomorphisms follow from \cite[Theorem 18.9.10]{Hirschhorn}.

\begin{proposition}\label{prop:adjmaps}
 Let $A$, $B$ be based $G$-$CW$ complexes, with $A$ compact. 
Then the  maps (\ref{eqn:adjmaps}) are equivariant weak equivalences. 
\end{proposition}
\begin{proof}
The right-hand map in (\ref{eqn:adjmaps}) is always a $G$-equivalence.  Proposition \ref{prop:topinf} together with the  equivariant homeomorphism 
$$
\colim_{n}\Map(A,\mcal{F}^{V}_{B}(n)) \xrightarrow{\wkeq} \Map(A,\colim_{n}\mcal{F}^{V}_{B}(n)).
$$
imply that the left-hand map is an equivariant weak equivalence. 
\end{proof}

\begin{corollary}[c.f. {\cite[Theorem 3.17]{Walker:Thomason}}]\label{cor:3.17}
 Let $A$, $B$ be based $G$-$CW$ complexes, with $A$ compact. 
Then the  maps (\ref{eqn:adjmaps}) induce an equivariant equivalence of $G$-spectra
$$
\bu_{G}(A,\, B)\wkeq\Map(A,\bu_{G}^{\mf{c}}(S^{0},B)) \wkeq \Map(A,B\wedge \bu_{G}),
$$ 
where $\bu_{G} = \bu^{\mf{c}}_{G}(S^{0},S^{0})$.
\end{corollary}
\begin{proof}
By the previous proposition the maps in (\ref{eqn:adjmaps}) are equivariant weak equivalences.  
It follows that  $\bu_{G}(A,\,B)$ is  equivariantly weakly equivalent to the spectrum $\{\Map(A, |\mcal{A}_{G}^{top}(S^{0},\,S^{W}\wedge B)^{\C[G]}|)\}$. By Proposition \ref{prop:ev}, we have that $\bu_{G}^{\mf{c}}(S^{0},B)$ is equivariantly weakly equivalent to the spectrum $\{|\mcal{A}_{G}^{top}(S^{0},\,S^{W}\wedge B)^{\C[G]}|\}$ and by \cite[Proposition 3.6]{AB} this spectrum is equivariantly weakly equivalent to $B\wedge \bu_{G}$.
\end{proof}

\begin{remark}\label{onemore}
 It follows from Proposition \ref{prop:conn} that there is a map of $G$-spectra $\bu_{G}\to KU_{G}$, where $KU_{G}$ is the spectrum representing equivariant complex $K$-theory. Furthermore $\bu_{G}$ is its connective cover, in the sense that this map induces isomorphisms $\pi^{H}_{n}\bu_{G} \xrightarrow{\iso} \pi_{n}^{H}KU_{G}$ for all subgroups $H\subseteq G$ and $n\geq 0$ and $\pi^{H}_{n}\bu_{G} = 0$ for $n<0$.
\end{remark}

\subsection{Group completions via mapping telescopes}\label{sec:tel}
We finish this section by giving an alternate description of the equivariant homotopy group completion of the equivariant homotopy monoid $|\mcal{A}_{G}^{top}(S^{0},\,S^{0})^{V}|$.  This is the equivariant analogue of only a part of the results in \cite[Section 4]{Walker:Thomason}. It seems somewhat more complicated to establish the analogous result for the monoids $|\mcal{A}_{G}^{top}(S^{0},\,B)^{V}|$ when $B$ has nontrivial action and we do not need it. Similarly we do not need the semi-topological analogues.
An important consequence of the description of the equivariant homotopy group completion as a mapping telescope appears below in Corollary \ref{cor:cc}. This is used in Section \ref{sec:pair} in order to define a natural map of rings
$$
\bu^{G}_{*}(W,S^{0}) \to KU_{G}^{-*}(W).
$$ 
This natural transformation is crucial for our main results in Sections \ref{sec:sstThom} and \ref{sec:algThom}.

Recall that $\mcal{F}^{V}_{S^{0}}(n)$ is isomorphic to the space of linear subspaces in $V^{n}$. Below, we write $\mcal{F}_{B}^{V}(\infty)=\colim_{n}\mcal{F}_{B}^{V}(n)$ where the colimit is over the standard inclusions  $\underline{n}\subseteq \underline{n+1}$ given by $i\mapsto i$. 

Consider the equivariant maps $\eta:\mcal{F}^{V}_{S^{0}}(n)\to \mcal{F}^{V}_{S^{0}}(n+1)$ given by $(W\subseteq V^{n}) \mapsto (V\oplus W \subseteq V\oplus V^{n})$. Taking colimits defines an equivariant map
 $\eta:\mcal{F}^{V}_{S^{0}}(\infty)\to \mcal{F}^{V}_{S^{0}}(\infty)$. 
Write $X_{i} = \mcal{F}^{V}_{S^{0}}(\infty)$ considered as a based $G$-space
 with basepoint $x_{i}$, where $x_{0} = 0$ and $x_{i}= \eta^{i}x_{0}$. Then $\eta:X_{i}\to X_{i+1}$ is a based map, Write 
$$
Tel(X_{i},\eta) = \hocolim(X_{0}\xrightarrow{\eta} X_{1}\xrightarrow{\eta} X_{2}\xrightarrow{\eta}\cdots )
$$
\begin{theorem}\label{thm:tel}
The map 
$$
\mcal{F}^{V}_{S^{0}}(\infty) \to Tel(X_{i},\eta)
$$
is an equivariant homotopy group completion.
\end{theorem}
\begin{proof}
The proof is modeled on that of \cite[Proposition 3.3]{FW:sstfct}, where $H$-space structures arising from operad actions rather than through $\Gamma$-spaces are used.
We need to show that $\mcal{F}^{V}_{S^{0}}(\infty) \to Tel(X_{i},\eta)$ is a map of equivariant homotopy commutative, associative $H$-spaces such that for any subgroup $K\subseteq G$ the map $\mcal{F}^{V}_{S^{0}}(\infty)^{K} \to Tel(X_{i},\eta)^{K}$ is a homotopy group completion. Recall that this means that for each subgroup $K\subseteq G$,
\begin{enumerate}
\item the map $\pi_{0}(\mcal{F}^{V}_{S^{0}}(\infty)^{K}) \to \pi_{0}(Tel(X_{i},\eta)^{K})$ is a group completion of the monoid $\pi_{0}(\mcal{F}^{V}_{S^{0}}(\infty)^{K})$, and 
\item $H_{*}(\mcal{F}^{V}_{S^{0}}(\infty)^{K};A) \to H_{*}(Tel(X_{i},\eta)^{K};A)$ is localization of the action of $\Z[\pi_{0}(\mcal{F}^{V}_{S^{0}}(\infty)^{K})]$ for any commutative ring $A$. 
\end{enumerate}

First we have to show that $Tel(X_{i},\eta)$ has the structure of an equivariant homotopy commutative and associative $H$-space. We have equivariant homotopy equivalences $Tel(X_{i}\times X_{i}, \eta\times\eta) \wkeq Tel(X_{i},\eta)\times Tel(X_{i},\eta)$ and $Tel(X_{i},\eta)\wkeq Tel(X_{2i}, \eta^{2})$. Thus to define the pairing it suffices to give a map $\mu:X_{n}\times X_{n} \to X_{2n}$ such that $\eta^{2}$ and $\mu\circ(\eta\times\eta)$ are equivariantly homotopic. We take $\mu$ to be the $H$-space product map $\mu:\mcal{F}^{V}_{S^{0}}(\infty)\times \mcal{F}^{V}_{S^{0}}(\infty)\to \mcal{F}^{V}_{S^{0}}(\infty)$. Recall that the $H$-space structure on $\mcal{F}^{V}_{S^{0}}(\infty)$ arises as follows. The equivariant $\Gamma$-space $\underline{n}_{+}\mapsto \mcal{F}^{V}_{\underline{n}_{+}}(\infty)$ is special and $\mu$ arises by choosing a homotopy inverse to $\mcal{F}_{\underline{2}_{+}}^{V}(\infty)\subseteq \mcal{F}^{V}_{S^{0}}(\infty)\times \mcal{F}^{V}_{S^{0}}(\infty)$ together with the multiplication 
map induced by $\underline{2}_{+}\to \underline{1}_{+}$ given by sending both $1$ and $2$ to $1$. Consider the commutative square
$$
\xymatrix{
\mcal{F}^{V}_{S^{0}}(n)\times\mcal{F}^{V}_{S^{0}}(n) \ar[r]^-{\eta\times \eta} & \mcal{F}^{V}_{S^{0}}(n+1)\times\mcal{F}^{V}_{S^{0}}(n+1) \ar[r]^-{\epsilon\gamma\times\gamma} & \mcal{F}^{V}_{S^{0}}(n+2)\times\mcal{F}^{V}_{S^{0}}(n+2) \\
\mcal{F}^{V}_{\underline{2}_{+}}(n) \ar[r]^{\beta}\ar[d]^{m}\ar[r]\ar[u]^{\subseteq}_{\wkeq} &  \mcal{F}^{V}_{\underline{2}_{+}}(n+2) \ar[d]^{m}\ar[ur]^{\subseteq}_{\wkeq} & \\
\mcal{F}^{V}_{S^{0}}(n) \ar[r]^{\eta^{2}} &  \mcal{F}^{V}_{S^{0}}(n+2) , &
}
$$
where $\beta$ sends the pair $(W_{1}, W_{2})$ to $(V\oplus 0 \oplus W_{1}, 0\oplus V \oplus W_{2})$
The map $\gamma$ is induced by $\gamma:\underline{k}_{+}\to \underline{k+1}_{+})$ where $\gamma(i) = i+1$ and $\epsilon$ is induced by $\epsilon:\underline{k}_{+}\to \underline{k}_{+}$ which interchanges $1$ and $2$ and is the identity on the other elements.
As shown in the proof of Proposition \ref{prop:sstinf}, injections $\mathbb{N}\to \mathbb{N}$ induce equivariant homotopy equivalences $\mcal{F}^{V}(\infty) \to \mcal{F}^{V}(\infty)$. In particular $\gamma$,$\epsilon:\mcal{F}^{V}(\infty) \to \mcal{F}^{V}(\infty)$ are both equivariant homotopy equivalences. We thus have that $\eta^{2} \wkeq \mu\circ(\eta\times \eta)$ and thus we obtain the required pairing  giving a multiplication on $Tel(X_{i},\eta)$.

Now we need to show that $\mu$ gives $Tel(X_{i},\eta)$ the structure of a homotopy commutative and homotopy associative $H$-space. First we show that the basepoint $x\in Tel(X_{i},\eta)$ is a right identity up to homotopy. The maps $\eta^{n}:X_{n}\to X_{2n}$ induce the homotopy equivalence 
$Tel(X_{n},\eta)\wkeq Tel(X_{2n},\eta^{2})$ and so to show that $x$ is a right homotopy identity it suffices to show that the maps $\alpha^{n}$, $\mu(-,x_{n}):X_{n}\to X_{2n}$ are homotopic. 
To show homotopy commutativity it suffices to show that the two maps $X_{i}\times X_{i}\to X_{2i}$ given by $\mu$ and $\mu\tau$ are homotopic, where $\tau$ is the map switching the factors. This follows from the fact that the maps $m$,$m\tau:\mcal{F}^{V}_{\underline{2}_{+}}(n)\to \mcal{F}^{V}_{S^{0}}(n)$ are equal, where $\tau$ is the map interchanging $1$ and $2$. Homotopy associativity follows in a similar fashion.

The map $\mcal{F}^{V}_{S^{0}}(\infty) \to Tel(X_{i},\eta)$ is an $H$-space map. 
The monoid $\pi_{0}(\mcal{F}^{V}_{S^{0}}(\infty)^{K})$ is the monoid (under direct sum) of isomorphism classes of $K$-modules which embed in $V^{\infty}$. Observe that $\pi_{0}(\mcal{F}^{V}_{S^{0}}(\infty)^{K})^{+}$ is obtained obtained by inverting the class of $V$ in $\pi_{0}(\mcal{F}^{V}_{S^{0}}(\infty)^{K})$. But this is exactly $\pi_{0}(Tel(X_{i},\eta)^{K})$. The condition on homology follows immediately since $H_{*}(Tel(X_{i},\eta)^{K};A)= \colim_{i}H_{*}(X_{i};A)$.
\end{proof}

The following corollaries are used later in Proposition \ref{prop:conn} to define natural transformations to $KU_{G}^{*}(-)$. 
\begin{corollary}\label{cor:gc}
 Let $W$ be a compact, unbased $G$-$CW$ complex.  Write  $\mcal{A}(-)$ for the equivariant $\Gamma$-space $\underline{n}_{+}\mapsto |\mcal{A}_{G}^{top}(S^{0},\,\underline{n}_{+}\wedge S^{0})^{V}|$. The natural map
$$
\big[W_{+}, |\mcal{A}_{G}^{top}(S^{0},\,S^{0})^{V}|\big]_{G}^{+} \to \big[W_{+}, \Omega \widehat{\mcal{A}}(S^{1})\big]_{G}
$$
is an isomorphism.
\end{corollary}
\begin{proof}
The equivariant $\Gamma$-space $\mcal{A}$ is special by Proposition \ref{prop:topspecial}. By Lemma \ref{lem:Gspc} the map $|\mcal{A}_{G}^{top}(S^{0},\,S^{0})^{V}|\to \Omega \widehat{\mcal{A}}(S^{1})$ is an equivariant homotopy group completion. By Proposition \ref{prop:topinf} we have an equivariant equivalence $|\mcal{A}_{G}^{top}(S^{0},\,S^{0})^{V}|\wkeq \mcal{F}^{V}_{S^{0}}(\infty)$ and so the previous theorem implies that we have an equivariant weak equivalence $\Omega\widehat{\mcal{A}}(S^{1})\wkeq Tel(X_{i},\eta)$. Therefore we have natural isomorphisms $\colim_{i}[W_{+}, X_{i}]_{G}= [W_{+},\Omega\widehat{\mcal{A}}(S^{1})]_{G}$. As unbased spaces we have that $X_{i} = X_{0}$. Thus we have the natural identification $[W_{+}, X_{i}]_{G}\iso [W_{+}, X_{0}]_{G}$ and under this identification the transition maps in $\colim_{i}[W_{+}, X_{i}]_{G} = \colim_{i}[W_{+}, X_{0}]_{G}$ are addition by the same element and so $\colim_{i}[W_{+}, X_{i}]_{G} =[W_{+}, \Omega \mcal{A}(S^{1})]_{G}$ is obtained by inverting 
this class in $[W_{+}, X_{0}]_{G}$. But since $[W_{+}, \Omega \widehat{\mcal{A}}(S^{1})]_{G}$ is a group, this must be the group completion of $[W_{+}, X_{0}]_{G}$.
\end{proof}

\begin{corollary}\label{cor:cc}
 Let $X$ be a compact, unbased $G$-$CW$ complex. The natural map
$$
\bu^{\mf{c}}_{G}(X_{+}, S^{0}) \to \bu_{G}(X_{+},S^{0})
$$
induces an equivalence
$$
\Omega^{\infty}\bu^{\mf{c}}_{G}(X_{+}, S^{0}) \to \Omega^{\infty}\bu_{G}(X_{+},S^{0})
$$
of associated equivariant infinite loop spaces. In particular, 
$$
\pi^{G}_{0}\bu^{\mf{c}}_{G}(X_{+}, S^{0}) = KU_{G}^{0}(X) = \pi^{G}_{0}\bu_{G}(X_{+},S^{0}).
$$
\end{corollary}
\begin{proof}
 By Proposition \ref{prop:adjmaps} we have that the equivariant $\Gamma$-space (resp.~ $\mcal{W}_{G}$-space) $|\mcal{A}^{top}_{G}(X_{+}, -\wedge S^{0})^{\C[G]}|$ is equivariantly weakly equivalence to the equivariant $\Gamma$-space (resp.~ $\mcal{W}_{G}$-space) $\Map(X_{+}, |\mcal{A}^{top}_{G}(S^{0}, -\wedge S^{0})^{\C[G]}|)$.

For convenience write $\mcal{M}(-) = \Map(X_{+}, |\mcal{A}^{top}_{G}(S^{0}, -\wedge S^{0})^{\C[G]}|)$ and $\mcal{A}(-) = |\mcal{A}^{top}_{G}(S^{0}, -)^{\C[G]}|$.  We show that 
$$
\Omega^{W}\widehat{\mcal{M}}(S^{W}) \to 
\Omega^{W}\Map(X_{+}, \mcal{A}(S^{W}))
$$
is an equivariant equivalence for any representation $W$ with $W^{G}\neq 0$. It follows from Proposition \ref{prop:topspecial} that the equivariant $\Gamma$-space $\mcal{M}$ is special and so  by Lemma \ref{lem:Gspc} the associated spectrum $\mathbb{S}\mcal{M}$ is a positive $\Omega$-$G$-spectrum and $\mcal{M}(S^{0})\to \Omega\widehat{\mcal{M}}(S^{1})$ is an equivariant homotopy group completion. It therefore suffices to show that
$$
\Map(X_{+}, \mcal{A}(S^{0}))  \to \Omega^{W}\Map(X_{+}, \mcal{A}(S^{W}) 
 = \Map(X_{+}, \Omega^{W}\mcal{A}(S^{W}))
$$
is an equivariant homotopy group completion whenever $W^{G}\neq 0$. Using Proposition \ref{prop:ev} and that $\mcal{A}(-)$ is special, we have an equivariant weak equivalence $\Omega^{W}\mcal{A}(S^{W})\wkeq \Omega^{1}\mcal{A}(S^{1}) \wkeq \Omega^{1}\widehat{A}(S^{1})$. It follows from Corollary \ref{cor:gc} that 
$$
\pi_{0}\Map(X_{+}, \mcal{A}(S^{0}))^{K}\to \pi_{0}\Map(X_{+}, \Omega^{W}\mcal{A}(S^{W}))^{K}
$$
is a group completion for any subgroup $K\subseteq G$. Furthermore, from Theorem \ref{thm:tel} it follows that we have an equivariant weak homotopy equivalence
$$
\Map(X_{+}, \Omega^{W}\mcal{A}(S^{W})) \wkeq Tel(\Map(X_{+}, \mcal{A}(S^{0})),\eta_{*})
$$ 
from which the condition on homology with coefficients follows.

For the last statement, we have that 
$$
\pi^{G}_{0}\bu_{G}(X_{+},S^{0}) = \big[X_{+},\Omega\mcal{A}(S^{1})\big]_{G} = \big[X_{+}, \mcal{F}^{\C[G]}_{S^{0}}(\infty)\big]_{G}^{+}.
$$
Since $\mcal{F}^{\C[G]}_{S^{0}}$ is the Grassmannian of linear subspaces inside of $\C[G]^{\infty}$ it follows that this group is naturally identified with $KU_{G}^{0}(X)$.
\end{proof}

\section{Walker's comparison theorem}\label{sec:comp}
In this section we define a comparison map between our semi-topological and topological bivariant equivariant $K$-theories. We present a proof of the equivariant version of Walker's Fundamental Comparison Theorem \cite[Theorem 5.1]{Walker:Thomason}. Namely, we have the following result, which is obtained as a special case of Corollary \ref{cor:qtopeq} below.

\begin{theorem}\label{mainthm}
Let $Y$ be a smooth quasi-projective $G$-variety. Then the natural map
(\ref{eqn:comp}) induces a weak equivalence of $G$-spectra,
$$
\bu^{\mf{c}}_{G}(S^{0},Y^{an})\xrightarrow{\wkeq} \mcal{K}^{sst}_{G}(\C,Y).
$$
\end{theorem}
\begin{remark}
 By Corollary \ref{cor:3.17} we have an equivariant weak equivalence of $G$-spectra, $\bu^{\mf{c}}_{G}(S^{0}, Y^{an}) \wkeq Y^{an}_{+}\wedge \bu_{G}$, where $\bu_{G} = \bu^{\mf{c}}_{G}(S^{0},S^{0})$ is the connective cover of $KU_{G}$ (see Remark \ref{onemore}). Thus this theorem says that semi-topological equivariant $K$-homology agrees with $\bu_{G}$-homology.
\end{remark}

The strategy of proof follows the nonequivariant one in \cite{Walker:Thomason}.
The key is to show that we have a natural equivariant homotopy equivalence
$\mcal{F}_{Y^{an}}\to (\mcal{G}_{Y})^{an}$.

A map $f:X\to Y$ of varieties defines also a continuous map $f:X^{an}\to Y^{an}$ of associated analytic spaces. This gives rise to a natural map of $G$-simplicial sets
\begin{equation}\label{eqn:qtop}
\Hom_{\Sch/\C}(X\times\dtop,\,\GG_{Y}^{V}(n)) \to \Hom_{cts}(X^{an}\times\dtop,\,\GG_{Y}^{V}(n)^{an}).
\end{equation}

To aid the comparison between the topological and semi-topological $K$-theories we introduce a bivariant theory based on the spaces $\mcal{G}^{V}_{Y}(n)^{an}$. Let $A$ be a based $G$-$CW$-complex and $Y$ a quasi-projective $G$-variety and define
$$
\mcal{A}_{G}^{qtop}(A,Y)^{V} = \hocolim_{I}\Hom_{cts*}(A\wedge\Delta^{\bullet}_{top\,+},\mcal{G}_{Y}^{V}(-)^{an}).
$$

\begin{definition}\label{defKqtop}
We define $\mcal{K}^{qtop}_{G}(A,Y)$ to be the $G$-spectrum associated to the equivariant $\Gamma$-space
$$
\underline{n}_{+}\mapsto |\mcal{A}_{G}^{qtop}(A, \underline{n}_{+}\wedge Y_{+})^{\C[G]}|.
$$
\end{definition}

The map (\ref{eqn:qtop})  above induces a natural
transformation $\mcal{K}^{sst}_{G} \to \mcal{K}^{qtop}_{G}$. Write 
$$
\Hom_{cts*}(A\wedge\Delta^{\bullet}_{top\,+},\,\mcal{G}_{Y}^{V}(\infty)^{an})=\colim_{n}\Hom_{cts*}(A\wedge\Delta^{\bullet}_{top\,+},\,\mcal{G}_{Y}^{V}(n)^{an})
$$
where the colimit is over the standard inclusions $\underline{n}\subseteq \underline{n+1}$ given by $i\mapsto i$.

As before we have the following (see the paragraph preceding Proposition \ref{prop:alginf} for a reminder on the indexing categories used below).
\begin{proposition}\label{prop:qtopinf}
For based $G$-$CW$-complexes $A$ and a quasi-projective $G$-variety $Y$, the natural maps
\begin{multline*}
\hocolim_{I}\Hom(A\wedge\Delta^{\bullet}_{top\,+},\,\mcal{G}_{Y}^{V}(-)) \to \hocolim_{\tilde{I}} \Hom(A\wedge\Delta^{\bullet}_{top\,+},\,\mcal{G}_{Y}^{V}(-)) \\
\leftarrow \hocolim_{M}\Hom(A\wedge\Delta^{\bullet}_{top\,+},\,\mcal{G}_{Y}^{V}(\infty))\leftarrow \Hom(A\wedge\Delta^{\bullet}_{top\,+},\,\mcal{G}_{Y}^{V}(\infty)).
\end{multline*}
are equivariant weak equivalences.
\end{proposition}

Let $V$ be a unitary complex representation. By \cite[Section 5]{Walker:Thomason} there is a natural map of spaces
\begin{equation}\label{eqn:comp}
\mcal{F}_{Y^{an}}^{V}(n) \to  \mcal{G}^{V}_{Y}(n)^{an},
\end{equation}
which will be shown in Theorem \ref{thm:ew} to be an equivariant weak equivalence.
This map can be described as follows. A point of $\mcal{F}_{Y^{an}}^{V}(n)$ is specified by a list of distinct points $y_{1},\ldots, y_{k}$ of $Y^{an}$ and a list of pairwise orthogonal subspaces $W_{1},\ldots, W_{k}$ of $V^{n}$. To this point we associate the surjection
$$
\mcal{V}^{n}_{Y} \xrightarrow{(eval_{y_{1},\ldots,y_{k}})} \bigoplus_{i=1}^{k}V^{n} \longrightarrow \bigoplus_{i=1}^{k}W_{i}^{*}.
$$
By \cite[Lemma 5.3]{Walker:Thomason} this defines a map $\mcal{F}^{V}_{Y^{an}}(n)\subseteq \mcal{G}^{V}_{Y}(n)^{an}$ which is a subspace inclusion and the 
image of $\mcal{F}_{Y^{an}}^{V}(n)$ can be characterized as follows.
When $Y=\spec(R)$ is affine, a point of $\mcal{G}^{V}_{Y}(n)$ is represented by a pair $(V^{n}\to P, \rho:R\to \End_{\C}(P))$ where $P$ is a quotient vector space and $\rho$ is a map of $\C$-algebras.  Two pairs $(V^{n}\to P, \rho:R\to \End_{\C}(P))$ and $(V^{n}\to Q, \rho:R\to \End_{\C}(Q))$ represent the same point when there is a vector-space isomorphism $P\iso Q$ which make the evident triangles commute. 
A pair  $( V^{n}\to P, \rho:R\to \End_{\C}(P))$ is in the image of $\mcal{F}^{V}_{Y^{an}}(n)$ exactly when  $\rho$ is a normal map of $\C^{*}$-algebras (i.e. $\rho(r)$ and $\rho(r)^{*}$ commute with each other for all $r\in R$).

The subspace $\mcal{F}^{V}_{Y^{an}}(n)\subseteq \mcal{G}^{V}_{Y}(n)^{an}$ is $G$-invariant and is compatible with the maps in $I$. We therefore have maps of equivariant $\Gamma$-spaces
$$
\mcal{A}^{sst}(X,-\wedge Y_{+})^{V} \to \mcal{A}^{qtop}(X^{an},-\wedge Y_{+})^{V}\leftarrow\mcal{A}^{top}(X^{an}, -\wedge Y^{an}_{+})^{V}
$$
 and thus natural maps of $G$-spectra
\begin{equation}\label{eqn:compbiv}
\mcal{K}^{sst}_{G}(X, Y)\to \mcal{K}^{qtop}_{G}(X^{an}, Y) \xleftarrow{\wkeq} \bu^{\mf{c}}_{G}(X^{an},Y^{an}).
\end{equation}
We will see in Corollary \ref{cor:qtopeq} that the right-hand map is a weak equivalence (note that the left map is an equivalence for $X=\spec(\C)$).

As shown in \cite[Section 5]{Walker:Thomason}, sending a quotient $\mcal{V}^{n}_Y\to \mcal{M}$ to its support defines a map of varieties
$\theta:\mcal{G}^{V}_{Y}(n)\to \coprod_{k=0}^{n|V|} \Sym^{k}Y$ and this is equivariant. Write $\theta_{top}$ for the restriction of $\theta$ to $\mcal{F}_{Y}^{V}(n)$.

The following lemma shows that in order to establish that (\ref{eqn:comp}) is an equivariant weak equivalence, it suffices to show that for each subgroup $H\subseteq G$, the map
$\mcal{F}_{Y}^{V}(n)^{H} \to (\GG^{V}_{Y}(n)^{an})^{H}$
is a weak equivalence  locally on $(\coprod_{k=0}^{n|V|} \Sym^{k}Y^{an})^{H}$, in a suitable sense.  This lemma and its proof are a slight modification of \cite[Lemma 5.4]{Walker:Thomason} in order to conclude a weak equivalence rather than a homology equivalence.
\begin{lemma}\label{lem:triangle}
Let 
$$
\xymatrix{
W\ar[dr]_-{p}\ar[rr]^{f} & & X \ar[dl]^-{q}\\
 & Z &
}
$$
be a commutative triangle of topological spaces
with the property that for every $z\in U\subseteq Z$ with $U$ open, there is an open $V$ with $z\in V\subseteq U$ such that the induced map $f|_{p^{-1}(V)}:p^{-1}(V) \to q^{-1}(V)$ is a weak equivalence. Then $f:W\to X$ is a weak equivalence.
\end{lemma}
\begin{proof}
The hypothesis implies that we can construct a hypercover $\mcal{V}_\bullet\to Z$ with the property that each $\mcal{V}_{n}$ is a disjoint union of open subsets
 $V\subseteq Z $ with the property that $p^{-1}(V) \to q^{-1}(V)$ is a weak equivalence. 
Write $(\mcal{V}_{n})_{X}$ and $(\mcal{V}_{n})_{W}$ for the pullback of $\mcal{V}_{n}$ to $X$ and $W$ respectively. We thus have weak equivalences $(\mcal{V}_{n})_{W}\to (\mcal{V}_{n})_{X}$ for each $n$, which induces a weak equivalence upon taking homotopy colimits. The result follows from the commutative diagram
$$
\xymatrix{
\hocolim_{\Delta}(\mcal{V}_{n})_{W} \ar[r]^{\simeq}\ar[d]_{\simeq} & \hocolim_{\Delta}(\mcal{V}_n)_{X} \ar[d]^{\simeq} \\
W \ar[r] & X.
}
$$  
where the vertical maps are weak equivalences by \cite[Theorem 1.3]{DI:tophyp}. 
\end{proof}

\begin{lemma}\label{lem:locet}
 Let $f:X\to Y$ be an equivariant \'etale map between quasi-projective complex $G$-varieties and $[q:\mcal{V}^{n}_{X}\to \mcal{M}]\in \mcal{G}_{X}^{V}(n)(\C)$ a point such that 
$\theta([q])=m_{1}x_{1}+\cdots + m_{v}x_{v}$ (with $x_{i}\neq x_{j}$ for $i\neq j$) has the property that $f(x_{i})\neq f(x_{j})$ for $i\neq j$. Then the equivariant map 
$f_{*}:\mcal{G}_{X}^{V}(n)\to \mcal{G}_{Y}^{V}(n)$ is \'etale at each point in $G\smash\cdot[q]$. 
\end{lemma}
\begin{proof}
This follows from \cite[Lemma 5.5]{Walker:Thomason}, as each $g\smash\cdot[q]$ satisfies the  hypothesis of the lemma.
\end{proof}

For a subspace $D\subseteq Y^{an}$ write $(\mcal{G}^{V}_{Y}(n))^{an}_{D}$ (resp.~ $\mcal{F}^{V}_{Y}(n)_{D}$) for the inverse image of $\coprod \Sym^{k}D \subseteq \coprod \Sym^{k}Y^{an}$ under $\theta$ (resp.~ under $\theta_{top}$). Observe that this is an open subspace (resp.~ closed) when $D$ is open (resp.~ closed) and it is invariant when $D$ is invariant. 

\begin{lemma}\label{lem:lociso}
 Suppose that $f:X\to Y$ is an equivariant map of quasi-projective $G$-varieties. Let $D\subseteq X^{an}$ be an invariant analytic open subset such that $f|_{D}$ is one-to-one and $f$ is \'etale at every point of $D$. Then the equivariant map $f_{*}:\mcal{G}^{V}_{X}(n)\to\mcal{G}^{V}_{Y}(n)$ induces an equivariant homeomorphism
$$
f_{*}:(\mcal{G}_{X}^{V}(n))^{an}_{D} \xrightarrow{\iso} (\mcal{G}_{Y}^{V}(n))^{an}_{f(D)}.
$$ 
Moreover this restricts to an equivariant homeomorphism 
$$
f_{*}:(\mcal{F}_{X}^{V}(n))^{an}_{D} \xrightarrow{\iso} (\mcal{F}_{Y}^{V}(n))^{an}_{f(D)}.
$$
\end{lemma}
\begin{proof}
 These maps are shown to be homeomorphisms in \cite[Lemma 5.7]{Walker:Thomason}, in our situation they are additionally equivariant. 
\end{proof}

\begin{theorem}\label{thm:ew}
 Let $Y$ be a smooth quasi-projective complex $G$-variety. Then $\mcal{F}^{V}_{Y}(n) \to \mcal{G}^{V}_{Y}(n)^{an}$ is an equivariant weak equivalence.  
\end{theorem}
\begin{proof}
 We show that for each subgroup $H\subset G$ the triangle
\begin{equation}\label{eqn:tritemp}
\xymatrix{
(\mcal{F}_{Y}^{V}(n))^{H} \ar[rr]\ar[rd]_-{\theta_{top}} & & (\mcal{G}_{Y}^{V}(n)^{an})^{H} \ar[ld]^-{\theta} \\
&(\coprod _{k=0}^{n|V|}\Sym^{k}Y^{an})^{H} & 
}
\end{equation}
satisfies the hypothesis for Lemma \ref{lem:triangle}. For a point $y\in Y^{an}$ with stabilizer $H_{y}$ write $[H/H_{y}\smash\cdot y]$ for the 
sum over the points of the $H$-orbit of $y$.
Any point in $(\coprod \Sym^{k}Y^{an})^{H}$ can be written in the form $m_{1}[H/H_{1}\smash\cdot y_{1}] + \cdots + m_{r}[H/H_{r}\smash\cdot y_{r}]$ with 
  $H/H_{i}\smash\cdot y_{i}\neq H/H_{j}\cdot y_{j}$ for $i\neq j$.
 Choose a point $P=m_{1}[H/H_{1}\smash\cdot y_{1}] + \cdots + m_{r}[H/H_{r}\smash\cdot y_{r}]$
in $(\coprod \Sym^{k}Y^{an})^{H}$ and an analytic open neighborhood $U$ of this point. As $Y$ is quasi-projective, we can find an invariant affine open subscheme of $Y$ which contains every point of the orbits  $H/H_{i}\smash\cdot y_{i}$ and so we may assume that $Y$ is affine. 
 
We have equivariant \'etale maps
$$
H\times^{H_{i}} T_{y_{i}}Y \xleftarrow{q_{i}} H\times^{H_{i}} Y \xrightarrow{\pi_{i}} Y.
$$ 
(the left-hand map is defined by using an $H_{i}$-equivariant splitting to the projection $d:m\to m/m^{2}$, where $m$ is the maximal ideal of $y_{i}$ in $\mcal{O}_{Y,Y_{i}}$). These maps
 satisfy
\begin{itemize}
 \item $\pi_{i}$, $q_{i}$ are $H$-equivariant and $q_{i}$ maps $(h, y_{i})$ to $(h,0)$,
\item $\pi_{i}$, $q_{i}$ are \'etale at every point of $H/H_{i}\times\{y_{i}\}$,
\item the images of $\pi_{i}(H/H_{i}\times y_{i})$ and $\pi_{j}(H/H_{j}\times y_{j})$ are disjoint if $i\neq j$, and
\item the images of $q_{i}(H/H_{i}\times y_{i})$ and $q_{j}(H/H_{j}\times y_{j})$ are disjoint if $i\neq j$.
\end{itemize}
Let $D_{i}\subseteq Y$ be $H_{i}$-invariant open analytic neighborhoods of $y_{i}\in Y$ such that 
 $hD_{i}\cap h'D_{i} = \emptyset$ if $h$, $h'$ belong to different cosets of $H_{i}$, $H\smash\cdot D_{i}\cap H\smash\cdot D_{j}= \emptyset$ if $i\neq j$, and $q_{i}:D_{i}\to q_iD_{i}$ is a homeomorphism with $q_{i}D_{i}\subseteq T_{y_i}Y$ an open polydisk. We may take the $D_{i}$ small enough so that
$\Sym^d (\cup H\smash\cdot D_{i})$ is contained in $U$ (where $d$ is such that $P\in\Sym^{d} Y$).
We may further assume that $\pi_i$ and $q_i$ are \'etale on 
every $H\times^{H_i}D_{i}$.

We show that $(\mcal{F}_{Y}^{V}(n))^{H}_{\cup H\smash\cdot D_{i}} \to 
(\mcal{G}_{Y}^{V}(n)^{an})^{H}_{\cup H\smash\cdot D_{i}}$ is a weak equivalence. This in turn implies that $\theta_{top}^{-1}(\Sym^{d}(\cup H\smash\cdot D_{i}))\to \theta^{-1}(\Sym^{d}(\cup H\smash\cdot D_{i}))$ is a weak equivalence, which establishes that (\ref{eqn:tritemp}) satisfies the hypothesis of Lemma \ref{lem:triangle} as required.

Applying Lemma \ref{lem:lociso} we have the following commutative diagram of homeomorphisms (where for notational convenience we surpress both $n$ and the superscript $an$ and write 
$\tilde{D}_{i}= H\times^{H_{y_i}}D_{i}$)
$$
\xymatrix{
\big(\mcal{F}^{V}_{\coprod (H\times^{H_{i}}T_{y_{i}}Y)}\big)_{\coprod q_{i}(\tilde{D}_{i})}^{H} \ar[d] & \big(\mcal{F}^{V}_{(\coprod H\times^{H_{i}}Y)}\big)_{\coprod \tilde{D}_{i}}^{H}\ar[l]_-{\iso} \ar[d] \ar[r]^-{\iso} & \big(\mcal{F}^{V}_{Y}\big)_{\cup \pi_{i}(\tilde{D}_{i})}^{H} \ar[d] \\
\big(\mcal{G}^{V}_{\coprod (H\times^{H_{i}}T_{y_{i}}Y)}\big)_{\coprod q_{i}(\tilde{D}_{i})}^{H} & \big(\mcal{G}^{V}_{\coprod H\times^{H_{i}}Y}\big)_{\coprod \tilde{D}_{i}}^{H}\ar[l]_-{\iso}  \ar[r]^-{\iso} & \big(\mcal{G}^{V}_{Y}\big)_{\cup \pi_{i}(\tilde{D}_{i})}^{H} .
}
$$
It therefore suffices to show that the left-hand vertical map is a weak equivalence. For each $i$ there is an $H_{i}$-equivariant linear map $T_{y_{i}}Y \to T_{y_{i}}Y$ mapping $q_i(D_{i})$ homeomorphically to the open unit polydisk $C_{i}\subseteq T_{y_{i}}Y$. Therefore it is enough to show that 
$$
\big(\mcal{F}^{V}_{\coprod H\times^{H_{i}} T_{y_{i}}Y}\big)^{H}_{\coprod H\times^{H_{i}}C_{i}} \to
\big(\mcal{G}^{V}_{\coprod H\times^{H_{i}} T_{y_{i}}Y}\big)^{H}_{\coprod H\times^{H_{i}}C_{i}}
$$ 
is a weak equivalence.

By Lemma \ref{lem:ehpt} below, the inclusion $H/H_{i} \subseteq H\times^{H_{i}}T_{y_{i}}Y$, at the point $0\in T_{y_i}Y$, induce weak homotopy equivalences 
$$
\big(\mcal{G}^{V}_{\coprod H/H_{i}}\big)^{H} \xrightarrow{\wkeq} \big(\mcal{G}^{V}_{\coprod H\times^{H_{i}} T_{y_{i}}Y}\big)^{H}_{\coprod C_{i}^{h}} \;\;\textrm{and} \;\; \big(\mcal{F}^{V}_{\coprod H/H_{i}}\big)^{H} \xrightarrow{\wkeq} \big(\mcal{F}^{V}_{\coprod H\times^{H_{i}} T_{y_{i}}Y}\big)^{H}_{\coprod C_{i}^{h}}.
$$ 
It thus suffices to show that $\big(\mcal{F}^{V}_{\coprod H/H_{i}}(n)\big)^{H} \to \big(\mcal{G}^{V}_{\coprod H/H_{i}}(n)^{an}\big)^{H}$ is a weak equivalence.
(These are equivariant analogues of the spaces denoted ${}^{0}(\mcal{G}^{n}_{T})^{an}_{C}$ and ${}^{0}(\mcal{F}^{n}_{T})_{C}$ in \cite{Walker:Thomason}).

A point of the space $\big(\mcal{G}^{V}_{\coprod H/H_{i}}(n)^{an}\big)^{H}$ consists of an $H$-module quotient $[V^{n}\onto W]$ and an equivariant map  $\C^{H/H_{1}}\times \cdots \times\C^{H/H_{r}}\to \End_{\C}(W)$ of $\C$-algebras (here $\C^{H/H_i}\iso\oplus_{H/H_{i}}\C$ denotes the coordinate ring of $H/H_i$). This amounts to giving an $H$-module quotient $[V^{n}\onto W]$, a direct sum decomposition $W= W_{1}\oplus \cdots \oplus W_{r}$ of $H$-modules, and $H_{i}$-module decomposition $W_{i}= \oplus_{H/H_{i}}W_{i}'$ inducing an $H$-module isomorphism  
$\mathrm{Ind}_{H_i}(W_i')\iso W_i$.
Similarly a point of $\big(\mcal{F}^{V}_{\coprod H/H_{i}}(n)\big)^{H}$  is given by an $H$-module quotient $[V^{n}\onto W]$ and an equivariant normal map of $C^{*}$-algebras $\C^{H/H_{1}}\times \cdots \times\C^{H/H_{r}}\to \End_{\C}(W)$. This amounts to giving an $H$-module quotient $[V^{n}\onto W]$, an $H$-module orthogonal sum decomposition $W= W_{1}\perp \cdots \perp W_{r}$, and an orthogonal $H_{i}$-module decomposition $W_{i}= \perp_{H/H_{i}}W_{i}'$ inducing an $H$-module isomorphism   
$\mathrm{Ind}_{H_i}(W_i')\iso W_i$.

Let $W$ be a finite dimensional Hermitian inner product $H$-module. For $H$-modules $W_1,\ldots, W_r$ such that $W\iso W_1\oplus \cdots \oplus W_r$,
write $X(W_1,\ldots, W_r)$ for the 
 space of decompositions $W=W_1\oplus \cdots \oplus W_r$ together with an $H_{i}$-module decomposition of each $W_{i}$,  $W_{i}= \oplus_{H/H_{i}}W_{i}'$ which induces an $H$-module isomorphism $W_i\iso \mathrm{Ind}_{H_i}(W')$.
Similarly, write $X^{\orth}(W_1,\ldots, W_r)$ for the space of orthogonal decompositions 
$W=W_1\perp \cdots \perp W_r$ with an orthogonal $H_{i}$-module decomposition of each $W_{i}$,  $W_{i}= \perp_{H/H_{i}}W_{i}'$ which induces an $H$-module isomorphism $W_i\iso \mathrm{Ind}_{H_i}(W')$.

First we consider the space of all decompositions of $W$ into an $r$-fold direct sum (resp.~ orthogonal sum) of $H$-modules. These spaces break up into a disjoint union of connected components $\Grass_{W}(W_1,\ldots, W_{r})$ (resp.~ $\Grass^{\orth}_{W}(W_1,\ldots, W_{r})$) of decompositions where the $i$th summand is $H$-isomorphic to $W_i$. As in \cite{Walker:Thomason}, we proceed by induction on $r$ to show that $\Grass_{W}^{\orth}(W_1,\ldots, W_{r})\subseteq \Grass_{W}(W_1,\ldots, W_{r})$ is a weak equivalence for any $W$ and a tuple  $(W_1,\ldots, W_r)$
of $G$-modules, such that $W = \oplus_iW_i$. The case $r=1$ is clear.

The map 
$\Grass_W(W_1,\ldots, W_{r}) \to \Grass_W(W_1,\ldots, W_{r-1}\oplus W_{r})$ is a fibration with fiber $\Grass_{W'}(W_{r-1},W_r)$ where $W'= W_{r-1}\oplus W_{r}$. Similarly in the orthogonal case we have a fibration $\Grass^{\orth}_W(W_1,\ldots, W_{r}) \to \Grass^{\orth}_W(W_1,\ldots, W_{r-1}\oplus W_{r})$ 
with fiber $\Grass^{\orth}_{W'}(W_{r-1},W_r)$. Now the forgetful map from $\Grass^{\orth}_{W'}(W_{r-1},W_r)$  to the space $\Grass^{H}(W_{r-1}\subseteq W')$ of sub-$H$-planes isomorphic to $W_{r-1}$ is a homeomorphism. On the other hand
$\Grass_{W'}(W_{r-1},W_r)\to \Grass^{H}(W_{r-1}\subseteq W')$ is a fibration with contractible fibers. We conclude that $\Grass_{W}^{\orth}(W_1,\ldots, W_{r})\subseteq \Grass_{W}(W_1,\ldots, W_{r})$ is a weak equivalence.

The forgetful map $X(W_1,\ldots,W_r)\to \Grass_W(W_1,\ldots, W_{r})$  is also a fibration with fiber $F$ which is the space of $H_{i}$-module decompositions $W_{i} = \oplus_{H/H_{i}}W_i'$ inducing an $H$-module isomorphism $\mathrm{Ind}_{H_i}(W_{i}')\iso W_{i}$, $i=1,\ldots, r$. Such a decomposition is determined by an $H_{i}$-equivariant embedding $W_{i}'\subseteq W_{i}$, such that the induced map $\mathrm{Ind}_{H_i}(W_{i}')\iso W_{i}$ is an isomorphism, together with a choice of $H_{i}$-equivariant splitting of the projection $W_i\to W_{i}/W_i'$.  
This implies that the fibers of $F\to Y_{H_1}(W_1'\subseteq W_1)\times\cdots\times Y_{H_r}(W_r'\subseteq W_r)$, where $Y_{H_{i}}(W_{i}'\subseteq W_{i})$ is the space of $H_i$-equivariant embeddings $W_{i}'\subseteq W_{i}$ which induce an isomorphism $\mathrm{Ind}_{H_i}W_i'\iso W_{i}$, are contractible. 
Similarly $X^{\orth}(W_1,\ldots,W_r)\to \Grass_W^{\orth}(W_1,\ldots, W_{r})$ is a fibration with fiber $F^{\orth}$, which is the space of decompositions of $H_{i}$-module decompositions $W_{i} = \perp_{H/H_{i}}W_i'$ inducing an $H$-module isomorphism $\mathrm{Ind}_{H_i}(W_{i}')\iso W_{i}$. This space is homeomorphic to $\Grass^{H_1}(W_1'\subseteq W_1)\times\cdots\times \Grass^{H_r}(W_r'\subseteq W_r)$.
\end{proof}

\begin{lemma}\label{lem:ehpt}
Let $G$ be a finite group and $V$ a representation. Let $\{H_{j}\}$ be finite set of subgroups of $G$, $W_{j}$ representations of $H_{j}$, $C_{j}$ the unit open polydisk in $W_j$, and $\tilde{C}_{j}= G\times^{H_j}C_j$.
The inclusions $\iota:\coprod G/H_{j}\to \coprod G\times^{H_{j}}W_{j}$ induce equivariant weak equivalences of topological spaces
$\iota_{*}:\mcal{G}^{V}_{\coprod G/H_{i}}(n)^{an}\to (\mcal{G}^{V}_{\coprod G\times^{H_{j}}W_{j}}(n)^{an})_{\coprod\tilde{C}_{j}}$
and
$\iota_{*}:\mcal{F}^{V}_{\coprod G/H_{i}}(n)\to (\mcal{F}^{V}_{\coprod G\times^{H_{j}}W_{j}}(n))_{\coprod\tilde{C}_{j}}$.
\end{lemma}
\begin{proof}
The arguments in both cases are similiar, so we treat only the first case.
Let $U=\spec(A)$ be an affine $G$-variety. We have an  equivariant natural transformation of functors 
$$
\mcal{G}^{V}_{U}(n)\times \A^{1}\to \mcal{G}^{V}_{U\times \A^1}(n)
$$
defined on affine varieties as follows. An $R$-valued point of the left-hand side is a triple $(V^n\otimes_{\C} R\onto Q, \rho:A\to \End_{R}(Q), f:\C[t]\to R)$, where $Q$ is a projective $R$-module, $\rho$ and $f$ are $\C$-algebra maps. Write $\phi$ for the $\C$-algebra map defined to be the composition $\eta f:\C[t]\to R\to \End_{R}(Q)$, where $\eta(r)$ is multiplication by $r$. Now define $F$ by sending this triple to $(V^{n}\otimes R\onto Q, \rho\otimes\phi:A\otimes \C[t]\to \End_{R}(Q))$. Moreover, the resulting square
$$
\xymatrix{
\mcal{G}^{V}_{U}(n)\times \A^{1}\ar[r]\ar[d]^{\theta_U\times\mathrm{id}} & \mcal{G}^{V}_{U\times \A^1}(n) \ar[d]^{\theta_{U\times\A^1}}\\
(\coprod \Sym^k(U))\times \A^{1}\ar[r] & \coprod \Sym^{k}(U\times\A^{1})
}
$$
is commutative.

Write $\pi: \coprod G\times^{H_{j}}W_{j} \to \coprod G/H_{j}$ for the projection. Then $\pi_{*}\iota_{*} = id_{*}$.  
We obtain an equivariant homotopy between $\iota_{*}\pi_{*}$ and $id_{*}$ by using the restriction of the natural transformation $F$ to $I=[0,1]$ and the maps $(\coprod G\times^{H_{j}}W_{j})\times\A^{1}\to \coprod G\times^{H_{j}}W_{j}$ given by $(x,t)\mapsto t\smash\cdot x$. Note that the commutativity of the above square implies that the resulting homotopy restricts to an equivariant homotopy $(\mcal{G}^{V}_{\coprod G\times^{H_{j}}W_{j}}(n)^{an})_{\coprod\tilde{C}_{j}}\times I \to (\mcal{G}^{V}_{\coprod G\times^{H_{j}}W_{j}}(n)^{an})_{\coprod\tilde{C}_{j}}$.
\end{proof}

Since we have that $\mcal{K}^{sst}_{G}(\C,\, Y) = \mcal{K}^{qtop}_{G}(S^0,\, Y)$, 
Theorem \ref{mainthm} is a special case of the following equivariant generalization of
\cite[Corollary 5.9]{Walker:Thomason}.

 \begin{corollary}\label{cor:qtopeq}
  Let $Y$ be a smooth quasi-projective $G$-variety and $A$ a based $G$-$CW$-complex. Then 
 $$
 \bu^{\mf{c}}_{G}(A,\,Y^{an})\to \mcal{K}^{qtop}_{G}(A,\, Y) 
 $$
 is an equivariant weak equivalence. 
 \end{corollary}
 \begin{proof}
 By the previous theorem the map 
 $$
 |\Hom_{cts*}(A\wedge\Delta^{\bullet}_{top\,+},\,\mcal{F}_{Y^{an}}^{V}(n))| \to 
 |\Hom_{cts*}(A\wedge\Delta^{\bullet}_{top\,+},\,\mcal{G}_{Y}^{V}(n))^{an}|
 $$
 is an equivariant weak equivalence for all $n$ and any $V$. Taking a homotopy colimit over $I$  gives an equivariant weak equivalence of equivariant $\Gamma$-spaces
 $$
 |\mcal{A}^{top}_{G}(A,(-)\wedge Y_{+}^{an})^{V}|\to |\mcal{A}^{qtop}_{G}(A,(-)\wedge Y_{+})^{V}|
 $$ 
 and thus of associated $G$-spectra. The result is the particular case $V=\C[G]$.
 \end{proof}

\section{Pairings and operations}\label{sec:pair}
In this section we establish two basic pairings on our bivariant theories generalizing those of \cite[section 6]{Walker:Thomason} to the equivariant setting. These pairings are compatible with the natural comparison maps of $G$-spectra
\begin{align}\label{eqn:manymaps}
\mcal{K}_{G}(X,Y) & \to \mcal{K}_{G}(X\times\Delta^{\bullet}_{top},Y) \xleftarrow{\wkeq} \mcal{K}^{sst}_{G}(X,\, Y) \\
\nonumber & \to \mcal{K}^{qtop}_{G}(X^{an},\, Y) \stackrel{\simeq}{\leftarrow} \bu^{\mf{c}}_{G}(X^{an},\,Y^{an}) 
\to \bu_{G}(X^{an},\,Y^{an}),
\end{align}
obtained from (\ref{eqn:algsst}), (\ref{eqn:compbiv}), and (\ref{eqn:c2nc}). The existence and compatibility of these pairings plays a crucial role in the applications in Sections \ref{sec:sstThom} and \ref{sec:algThom}. An important special case occurs when $Y=\spec(\C)$. Then the pairings specialize to give these spectra the structure of commutative ring spectra and these maps are maps of ring spectra. In particular, combined with Proposition \ref{prop:conn} below, we have that
$$
K_{*}^{G}(X) \to K_{*}^{G,\,sst}(X)\to KU^{-*}_{G}(X^{an})
$$
are graded ring homomorphisms.

 The construction of these bivariant theories all begin with the consideration of a functor $F:I\to G\sSet$, where $I$ is the category whose objects are the sets $\underline{n}=\{1,2,\ldots, n\}$ for each $n\geq 0$ (so $\underline{0}$ is the empty set) together with injective set maps. Let $\diamond:I\times I \to I$ denote the functor which sends the pair $(\underline{m},\underline{n})$ to $\underline{mn}$. Given injections $\alpha:\underline{m}\to \underline{n}$ and $\beta:\underline{p}\to \underline{q}$ in $I$ we define $\alpha\diamond\beta:\underline{mp}\to \underline{nq}$ by 
$$
\alpha\diamond\beta((i-1)p+j) =  (\alpha(i)-1)q+\beta(j)
$$ 
for $i\in\underline{m}$, $j\in \underline{p}$, which is the map obtained by identifying $\underline{m}\times\underline{p}$ and $\underline{n}\times\underline{q}$ with $\underline{mp}$ and $\underline{nq}$ via the lexicographical ordering.
\begin{definition}
Let $F$, $G$, $H:I\to G\sSet$ be $I$-diagrams of $G$-simplicial sets. A \textit{pairing  of $I$-diagrams} $F\times G\to H$, is a natural transformation $F\times G \to H\diamond$ of functors $I\times I\to G\sSet$. 
\end{definition}
Such a pairing induces a pairing of $G$-simplicial sets
$$
\hocolim_{I}F\times\hocolim_{I}G\xrightarrow{\iso} \hocolim_{I\times I}F\times G \to \hocolim_{I\times I}H\diamond \to \hocolim_{I} H .
$$

The \textit{external product} is defined as follows. First we define a pairing 
$$
\boxtimes:\Hom(X,\,\GG_{Y}^{V}(m))\times \Hom(W,\,\GG_{Z}^{V}(n)) \to \Hom(X\times W,\,\GG_{Y\times Z}^{V}(mn))
$$
of $I$-diagrams as follows.
Define
$$
[p:\mcal{V}^{m}_{X\times Y}\onto \mcal{M}]\boxtimes [q:\mcal{V}^{n}_{W\times Z}\onto \mcal{N}]
$$
to be the composition
$$
\mcal{V}^{mn}_{X\times W\times Y\times Z} \iso \pi_{X\times Y}^{*}\mcal{V}^{m}_{X\times Y}\otimes \pi^{*}_{W\times Z}\mcal{V}^{n}_{X\times Y} \to \pi^{*}_{X\times Y}\mcal{M}\otimes\pi^{*}_{W\times Z}\mcal{N},
$$
where the first isomorphism is given using the lexicographical ordering, that is $e_{i}\otimes e_{j}$ is sent to $e_{(i-1)n+j}$. The maps $\pi_{X\times Y}$ and $\pi_{W\times Z}$ are the evident projections. It is clear that $\alpha^{*}\boxtimes \beta^{*} = (\alpha\diamond\beta)^{*}$ and thus we have a natural pairing of $I$-diagrams of $G$-sets. We thus obtain the natural pairing of $I$-diagrams of $G$-simplicial sets
\begin{align*}
\boxtimes:\Hom(X\times\Delta^{\bullet}_{\C},&\,\GG_{Y}^{V}(m))\times \Hom(W\times\Delta^{\bullet}_{\C},\,\GG_{Z}^{V}(n)) \\ 
& \to \Hom(X\times W\times\Delta^{\bullet}_{\C}\times\Delta^{\bullet}_{\C},\,\GG_{Y\times Z}^{V}(mn)) \\ 
& \to \Hom(X\times W\times\Delta^{\bullet}_{\C},\,\GG_{Y\times Z}^{V}(mn)) 
\end{align*}
where the second map is induced by the diagonal $\Delta^{d}_{\C}\to \Delta^{d}_{\C}\times\Delta^{d}_{\C}$.
Similarly making use of the diagonal $\Delta^{d}_{top}\to \Delta^{d}_{top}\times\Delta^{d}_{top}$ we obtain
\begin{align*}
\boxtimes:\Hom(X\times\Delta^{\bullet}_{top},&\,\GG_{Y}^{V}(m))\times \Hom(W\times\Delta^{\bullet}_{top},\,\GG_{Z}^{V}(n)) \\ 
& \to \Hom(X\times W\times\Delta^{\bullet}_{top},\,\GG_{Y\times Z}^{V}(mn)).
\end{align*}

Taking homotopy colimits we obtain the external pairing of $G$-simplicial sets
\begin{align*}
\boxtimes&:\mcal{A}_{G}(X,Y)^{V} \times \mcal{A}_{G}(W,Z)^{V}\to \mcal{A}_{G}(X\times W,Y\times Z)^{V},\\
\boxtimes&:\mcal{A}_{G}(X\times\Delta^{\bullet}_{top},Y)^{V} \times \mcal{A}_{G}(W\times\Delta^{\bullet}_{top},Z)^{V}\to \mcal{A}_{G}(X\times W\times\Delta^{\bullet}_{top},Y\times Z)^{V},\,\,\textrm{and} \\
\boxtimes&:\mcal{A}^{sst}_{G}(X,Y)^{V} \times \mcal{A}^{sst}_{G}(W,Z)^{V}\to \mcal{A}^{sst}_{G}(X\times W,Y\times Z)^{V}.
\end{align*}
A straightforward verification shows that this pairing is associative.
To see that the three pairings are compatible with the first two natural
transformations of \ref{eqn:manymaps}, use the naturality of $\Hom$
in the first variable applied to the diagonal $\Delta^{\bullet} \to
\Delta^{\bullet} \times \Delta^{\bullet}$ and the projection
$\Delta^{\bullet} \to \Delta^{0}$, both for
$\dC$ and $\Delta^{\bullet}_{top}$. 

The pairing $\boxtimes$ extends to give an associative pairing 
$$
\boxtimes:\mcal{A}^{qtop}_{G}(X^{an}, Y)^{V} \times \mcal{A}^{qtop}_{G}(W^{an},Z)^{V} \to \mcal{A}^{qtop}_{G}((X\times W)^{an},Y\times Z)^{V}.
$$

compatible with the one on $\mcal{A}^{sst}_{G}$.

These pairings give rise to an external pairing of equivariant $\Gamma$-spaces 
\begin{multline*}
 (\underline{m}_{+}\mapsto \mcal{A}_{G}(X,m_{+}\wedge Y_{+})^{V})\times (\underline{n}_{+}\mapsto \mcal{A}_{G}(W,n_{+}\wedge Z_{+})^{V}) \\
\xrightarrow{\boxtimes} ((\underline{m}_{+},\underline{n}_{+})\mapsto \mcal{A}_{G}(X\times W, m_{+}\wedge Y_{+} \wedge n_{+}\wedge Z_{+})
\end{multline*}
and similarly for $\mcal{A}^{sst}_{G}$ and $\mcal{A}^{qtop}_{G}$. We therefore obtain by the discussion in Section \ref{subsectionpairings} pairings of natural and associative pairings of spectra.
\begin{align*}
\boxtimes&:\mcal{K}_{G}(X,Y)\wedge\mcal{K}_{G}(W, Z) \to \mcal{K}_{G}(X\times W, Y\times Z), \\
\boxtimes&:\mcal{K}_{G}(X\times\Delta^{\bullet}_{top},Y)\wedge\mcal{K}_{G}(W\times\Delta^{\bullet}_{top}, Z) \to \mcal{K}_{G}(X\times W\times\Delta^{\bullet}_{top}, Y\times Z), \\
\boxtimes&:\mcal{K}^{sst}_{G}(X,Y)\wedge\mcal{K}^{sst}_{G}(W, Z) \to \mcal{K}^{sst}_{G}(X\times W, Y\times Z)\textrm{, and} 
\\
\boxtimes&:\mcal{K}^{qtop}_{G}(X^{an},Y)\wedge\mcal{K}^{qtop}_{G}(W^{an}, Z) \to \mcal{K}^{qtop}_{G}(X^{an}\times W^{an}, Y\times Z).
\end{align*}

Now we define a pairing
\begin{multline*}
 \boxtimes:\mcal{F}^{V}_{S}(m)\times \mcal{F}^{V}_{T}(n) = \underline{\Hom}_{*}(\mcal{C}_{0}(S), \End_{\C}(V^{m}))\times \underline{\Hom}_{*}(\mcal{C}_{0}(T),\End_{\C}(V^{n})) \\
\to \underline{\Hom}_{*}(\mcal{C}(S\wedge T), \End_{\C}(V^{mn})) =  \mcal{F}^{V}_{S\wedge T}(mn).
\end{multline*}
Given $f:\mcal{C}_{0}(S)\to \End_{\C}(V^{m})$ and $g:\mcal{C}_{0}(T)\to \End_{\C}(V^{n})$ we define
$f\boxtimes g$ to be the composition
$$
f\boxtimes g:\mcal{C}_{0}(S\wedge T)\iso \mcal{C}_{0}(S)\otimes\mcal{C}_{0}(T)\xrightarrow{f\otimes g} \End_{\C}(V^{m})\otimes \End_{C}(V^{n})\to \End_{\C}(V^{mn})
$$
where the last map uses the lexicographical indexing $\ell:V^{m}\otimes V^{n}\iso V^{mn}$ via $e_{i}\otimes e_{j}\mapsto e_{n(i-1)+j}$. Given injections $\alpha:\underline{m}\to \underline{p}$ and $\beta:\underline{n}\to \underline{q}$ then under the above isomorphism we have that $\tilde{\alpha}f\otimes \tilde{\beta}g$ agrees with $\widetilde{\alpha\diamond\beta}\ell(f\otimes g)$. We obtain a natural pairing
\begin{align*}
\boxtimes:\Hom_{cts*}(A,\mcal{F}^{V}_{S}(m)) & \times \Hom_{cts*}(B,\mcal{F}^{V}_{S}(n)) \\
& \to \Hom_{cts*}(A\wedge B,\mcal{F}^{V}_{S}(m)\times\mcal{F}^{V}_{T}(n))\\
& \to \Hom_{cts*}(A\wedge B,\mcal{F}^{V}_{S\wedge T}(mn))
\end{align*}
of $I$-diagrams. This pairing is associative in the evident sense and induces the external pairing of equivariant $\Gamma$-spaces
$$
\boxtimes:\mcal{A}^{top}_{G}(A,-\wedge S)^{V}\wedge \mcal{A}^{top}_{G}(B,-\wedge T)^{V}\to \mcal{A}^{top}_{G}(A\wedge B, -\wedge S\wedge - \wedge T)^{V}
$$
and thus a pairing of $G$-spectra
$$
\boxtimes:\bu^{\mf{c}}_{G}(A,S)\wedge \bu^{\mf{c}}_{G}(B,T) \to \bu^{\mf{c}}_{G}(A\wedge B, S\wedge T).
$$
This pairing extends to an external pairing of $\mcal{W}_{G}$-spaces and thus we obtain the pairing of $G$-spectra
$$
\boxtimes:\bu_{G}(A,S)\wedge \bu_{G}(B,T) \to \bu_{G}(A\wedge B, S\wedge T).
$$ 

It is clear that the last of the natural transformations of
\ref{eqn:manymaps} is compatible with the pairings. To see this
for the second last one, one uses the naturality of the maps
discussed in section \ref{sec:comp}.
Write $\mcal{K}^{?}_{G}(X,Y)$ for any one of the six bivariant theories appearing in (\ref{eqn:manymaps}). The following proposition summarizes the preceding discussion.
\begin{proposition}
Let $X$, $Y$, $W$, $Z$, $S$, and $T$, be quasi-projective $G$-varieties. We have pairings 
$$
\boxtimes:\mcal{K}_{G}^{?}(X,Y) \wedge \mcal{K}^{?}_{G}(W,Z)\to \mcal{K}^{?}_{G}(X\times W, Y\times Z).
$$
This pairing is associative in the sense that the two evident maps
 $$
\mcal{K}^{?}_{G}(X,Y) \times \mcal{K}^{?}_{G}(W,Z)\times \mcal{K}^{?}_{G}(S,T)\to \mcal{K}^{?}_{G}(X\times W\times S,Y\times Z\times T)
$$
agree. Moreover, these pairings are compatible with the each of the natural transformations (\ref{eqn:manymaps}).
\end{proposition}

Taking $X=W=\spec(\C)$, here and below this is to be interpreted as $S^{0}$ in the topological case, in the external product defines the \textit{external product for homology}
$$
\underline{\wedge}:\mcal{K}^{?}_{G}(\C,Y)\wedge\mcal{K}^{?}_{G}(\C,Z)\to \mcal{K}^{?}_{G}(\C,Y\times Z).
$$
Specializing to $Y=Z=\spec(\C)$  in the external pairing defines the \textit{external product for cohomology},
$$
\barwedge:\mcal{K}^{?}_{G}(X,\C)\wedge \mcal{K}^{?}_{G}(W,\C)\to \mcal{K}^{?}_{G}(X\times W,\C).
$$
We define the cup product by specializing further to $X=W$ and composing with the pullback along the diagonal $\Delta:X\to X\times X$,
\begin{equation}\label{eqn:cup}
\cup:\mcal{K}^{?}_{G}(X,\C)\wedge \mcal{K}^{?}_{G}(X,\C) \to \mcal{K}^{?}_{G}(X,\C).
\end{equation}
The cup-product turns $\pi_{*}\mcal{K}^{?}_{G}(X,\C)$ into a graded ring,
and even into a graded $\pi_{*}\mcal{K}^{?}_{G}(\C,\C)$-algebra. Immediate from the definitions we have the following.
\begin{proposition}\label{prop:cup}
 Let $X$ be a quasi-projective $G$-variety. The natural maps 
\begin{align*}
K^{G,\,\A^{1}}_{*}(X,\C) & \to K^{G}_{*}(X\times\Delta^{\bullet}_{top},\C)\xleftarrow{\iso} K^{G,\, sst}_{*}(X,\C) \\ & \to K^{G,\,qtop}_{*}(X^{an},\C) \stackrel{\iso}
{\leftarrow} \bu^{G,\mf{c}}_{*}(X^{an},S^{0}) 
\to \bu^{G}_{*}(X^{an},S^{0})
\end{align*}
induced by (\ref{eqn:manymaps}) are graded ring homomorphisms.
\end{proposition}

Our second basic pairing is the composition pairing. There is a composition pairing
$$
\theta_{X,Y,Z}:\Hom_{\Sch/\C}(X,\,\GG_{Y}^{V}(m)) \times \Hom_{\Sch/\C}(Y,\,\GG_{Z}^{V}(n)) \to \Hom_{\Sch/\C}(X,\,\GG_{Z}^{V}(mn))
$$
defined as follows. Given a pair
$$
([p:\mcal{V}^{m}_{X\times Y}\onto \mcal{M}],[q:\mcal{V}^{n}_{Y\times Z}\onto\mcal{N}])
$$
we have the quotient object
$$
[p\otimes q:\mcal{V}^{mn}_{X\times Y\times Z}\iso \mcal{V}^{m}_{X\times Y\times Z}\otimes \mcal{V}^{n}_{X\times Y\times Z} \to \pi^{*}_{X\times Y}\mcal{M}\otimes\pi_{Y\times Z}^{*}\mcal{N}]
$$
where the isomorphism is via the lexicographical ordering, as in the definition of the external product pairing above, and $\pi_{X\times Y}$ and $\pi_{Y\times Z}$ are the evident projections. Now pushforward along the projection $\pi_{X\times Z}$ define $\theta(p,q)$
$$
[\theta(p,q):\mcal{V}^{mn}_{X\times Z} \to (\pi_{X\times Z})_{*}(\pi^{*}_{X\times Y}\mcal{M}\otimes\pi_{Y\times Z}^{*}\mcal{N})].
$$
It is easily verified that $\theta_{X,Y,Z}(p,q)\in \Hom_{\Sch/\C}(X,\,\GG_{Z}^{V}(mn))$ as needed. Moreover if $\alpha:\underline{m}\to\underline{m'}$ and $\beta:\underline{n}\to \underline{n'}$ are injections then $\theta(\alpha^{*}p,\beta^{*}q) = (\alpha\diamond\beta)^{*}\theta(p,q)$ and thus $\theta$ defines a pairing of $I$-diagrams. Abusing notation slightly, we also write 
$$
\theta_{X,Y,Z}:\Hom(X\times U,\,\GG_{Y}^{V}(m)) \times \Hom(Y\times U,\,\GG_{Z}^{V}(n)) \to \Hom(X\times U,\,\GG_{Z}^{V}(mn))
$$
for the pairing obtained by composing with the  map $\mcal{G}_{Y}^{V}(n)\to \mcal{G}_{Y\times U}^{V}(n)$ defined by $[q:\mcal{V}^{n}_{X\times Y}\onto\mcal{M}]\mapsto [\mcal{V}^{n}_{X\times Y\times U}\onto \pi^{*}_{X\times Y}\mcal{M}]$. We thus obtain pairings
$$
\Hom(X\times \Delta^{\bullet}_{\C},\,\GG_{Y}^{V}(m)) \times \Hom(Y\times \Delta^{\bullet}_{\C},\,\GG_{Z}^{V}(n)) \to \Hom(X\times \Delta^{\bullet}_{\C},\,\GG_{Z}^{V}(mn))
$$
and
$$
\Hom(X\times \Delta^{\bullet}_{top},\,\GG_{Y}^{V}(n)) \times \Hom(Y\times \Delta^{\bullet}_{top},\,\GG_{Z}^{V}(m)) \to \Hom(X\times \Delta^{\bullet}_{top},\,\GG_{Z}^{V}(mn)).
$$
After taking homotopy colimits we obtain the pairing of equivariant $\Gamma$-spaces
$$
\mcal{A}_{G}(X,\underline{m}_{+}\wedge Y_{+})\times \mcal{A}_{G}(Y, \underline{n}_{+}\wedge Z_{+})\to \mcal{A}_{G}(X,\underline{mn}_{+}\wedge Z_{+})
$$
where as usual $\underline{m}_{+}\wedge\underline{n}_{+}$ is identified with $\underline{mn}_{+}$ via $(i,j)\mapsto (i-1)n+j$, and similarly for $\mcal{A}^{sst}_{G}$. For a space $T$, $\Hom_{cts}(T,G_{Y}^{an}) = \colim_{T\to U^{an}}\Hom_{\Sch/\C}(U,G_{Y}^{an})$. With this observation, the definition of $\theta$ readily extends to a pairing for $\mcal{A}^{qtop}_{top}$.
We thus obtain natural pairings of $G$-spectra
\begin{align*}
\theta&:\mcal{K}_{G}(X,Y)\wedge \mcal{K}_{G}(Y,Z) \to \mcal{K}_{G}(X,Z), \\
\theta&:\mcal{K}_{G}(X\times\Delta^{\bullet}_{top},Y)\wedge \mcal{K}_{G}(Y\times\Delta^{\bullet}_{top},Z) \to \mcal{K}_{G}(X\times\Delta^{\bullet}_{top},Z), \\
\theta&:\mcal{K}^{sst}_{G}(X,Y)\wedge \mcal{K}^{sst}_{G}(Y,Z) \to \mcal{K}^{sst}_{G}(X,Z), \,\,\textrm{and}\\
\theta&:\mcal{K}^{qtop}_{G}(X,Y)\wedge \mcal{K}^{qtop}_{G}(Y,Z) \to \mcal{K}^{qtop}_{G}(X,Z)
\end{align*}

Now for based $G$-$CW$-complexes $S$,$T$, and $U$ we define the pairing
$$
\theta_{S,T,U}:\Hom_{cts*}(S,\,\mcal{F}_{T}^{V}(m)) \times \Hom_{cts*}(T,\,\mcal{F}_{U}^{V}(n)) \to \Hom_{cts*}(S,\,\mcal{F}_{U}^{V}(mn))
$$
defined by sending a pair of $*$-maps 
$$
(p:\mcal{C}_{0}(T)\to \mcal{C}_{0}(S)\otimes\End_{\C}(V^{m}),\,\, q:\mcal{C}_{0}(U)\to \mcal{C}_{0}(T)\otimes\End_{\C}(V^{n}))
$$ 
to the composite
\begin{align*}
\mcal{C}_{0}(U)\xrightarrow{q} \mcal{C}_{0}(T)&\otimes\End_{\C}(V^{n})\xrightarrow{p\otimes 1} \mcal{C}_{0}(S)\otimes\End_{\C}(V^{m})\otimes\End_{\C}(V^{n}) \\
&\to \mcal{C}_{0}(S)\otimes \End_{\C}(V^{m}\otimes V^{n})\xrightarrow{\iso} \mcal{C}_{0}(S)\otimes \End_{\C}(V^{mn}),
\end{align*}
where in the last map we have identified $V^{m}\otimes V^{n}$ with $V^{mn}$ via the lexicographical ordering as above. It is straightforward to check that $\theta(\tilde{\alpha}p,\tilde{\beta}q) = \widetilde{\alpha\diamond\beta}\theta(p,q)$. We thus obtain a pairing of $I$-diagrams which gives a pairing of  $\Gamma$-spaces and therefore a pairing of $G$-spectra
$$
\theta:\bu^{\mf{c}}_{G}(S,T)\wedge \bu_{G}^{\mf{c}}(T,U) \to \bu^{\mf{c}}_{G}(S,U).
$$
Similarly we have a pairing of $\mcal{W}_{G}$-spaces leading to a pairing of $G$-spectra
$$
\theta:\bu_{G}(S,T)\wedge \bu_{G}(T,U)\to \bu_{G}(S,U).
$$

The pairing $\theta$ enjoys the same 
properties as in the non-equivariant case, namely naturality,
associativity and compatibility both with the pairing $\boxtimes$ and with 
the natural transformations (\ref{eqn:manymaps}). That is, the equivariant analogues of
\cite[Propositions 6.4, 6.5 and 6.6.]{Walker:Thomason} all
hold.

We define slant products and the cap product in the
usual fashion. In the topological case, $\spec(\C)$ is interpreted as $S^{0}$.

\begin{definition}
As above, we write $\mcal{K}^{?}_{G}$ for any one of the bivariant theories 
appearing in (\ref{eqn:manymaps}).
Let $X$ and $Y$ be quasiprojective $G$-varieties resp.~ $G-CW$-complexes.
\begin{enumerate}
\item
We define the slant product pairing
$$ /: \mcal{K}^{?}_{G}(X \times Y,\C) \wedge \mcal{K}^{?}_{G}(\C,Y) \to 
\mcal{K}^{?}_{G}(X,\C)$$
by $a/b:=\theta_{X,X \times Y, \C}(\tau(a \wedge (1_X \boxtimes b)))$,
where $\tau$ is the obvious involution.
\item
We define the slant product pairing 
$$ \backslash: \mcal{K}^{?}_{G}(X,\C) \wedge \mcal{K}^{?}_{G}(\C,X \times Y) 
\to \mcal{K}^{?}_{G}(\C,X)$$
by $a \backslash b:= \theta_{\C,X \times Y,Y}(\tau(a\boxtimes 1_Y)\wedge b))$.
\item
Finally, we define the cap product 
$$\cap: \mcal{K}^{?}_{G}(X,\C) \wedge \mcal{K}^{?}_{G}(\C,X) \to \mcal{K}^{?}_{G}(\C,X)$$ 
by $a \cap b := a \backslash \Delta_*(b)$, where
$\Delta: X \to X \times X$ is the diagonal embedding. 
\end{enumerate}
\end{definition}

Again by definition, these products are compatible with the natural
 transformations (\ref{eqn:manymaps}).

As in \cite[Proposition 6.10]{Walker:Thomason} we observe that the operations given here coincide with the
 ``classical'' ones. (See \cite[section XIII.5]{May:equihomotopy} for a discussion of the ``classical operations'' in the equivariant setting.) 
The results \cite[Lemma 6.12, Proposition 6.13]{Walker:Thomason} also hold equivariantly and are needed later.

We can also use these pairings to define
transfer maps for finite (but not necessarily dominant) equivariant morphisms $f:X\to Y$ between smooth projective complex $G$-varieties, for the  bivariant 
equivariant $K$-theories we consider.
Recall that we write $K_{*}^{G,\,alg}(X,Y)$ for the algebraic $K$-theory of the exact category $\mcal{P}(G;X,Y)$ of coherent $G$-modules on $X\times Y$ which are finite and flat over $X$. Write $K_{*}^{\prime}(G;X,Y)$ for the $K$-theory groups of the abelian category $\mcal{M}(G;X,Y)$ of coherent $G$-modules on $X\times Y$ which are finite over $X$.
To define these transfer maps, we make use of the following equivariant analog of \cite[Lemma 6.14]{Walker:Thomason}.
\begin{lemma}
 Let $X$ and $Y$ be smooth quasi-projective $G$-varieties, with $Y$ projective. The natural mapping
$$
\mcal{K}_{n}^{G,\,alg}(X,Y) \to \mcal{K}_{n}^{\prime}(G;X,Y)
$$
induced by the inclusion $\mcal{P}(G;X,Y)\subseteq \mcal{M}(G;X,Y)$ is an isomorphism for all $n\geq 0$.
\end{lemma}
\begin{proof}
The proof of \cite[Lemma 2.2]{Walker:Adamsop} concerning
the nonequivariant bivariant $K$-theory with $Y=(\P ^1)^{\times s}$ generalizes 
to equivariant $K$-theory and $X$ an arbitrary smooth quasi-projective 
$G$-variety if one replaces the subscheme $D$ in the proof 
by the union $GD:=\cup_{g \in G}gD$. The map $GD \to X$ 
is quasi-finite as $D \to X$ is and because $G$ is finite.
To see that it is proper (and hence finite), one uses again 
that $G$ is finite. 
The same argument as in \cite[Lemma 2.3]{Walker:Adamsop}
then yields the equivalence for $Y=\P ^n$. For an arbitrary $Y$
 with non-trivial $G$-action we have a finite, surjective equivariant map $Y\to Y/G$. We may assume that $Y/G$ is connected and applying Noether normalization to $Y/G$ we obtain a finite, surjective equivariant map $Y\to \P^{n}$. Arguing as in
\cite[Lemma 6.14]{Walker:Thomason} we see that the result for $Y$ follows from the result for $\P^{n}$.
\end{proof}

When $X$, $Y$ are smooth with $Y$ projective, the previous lemma, Proposition \ref{prop:fix}, and the sequence of natural transformations (\ref{eqn:manymaps}) imply that a class $[\mcal{M}]\in K^{\prime}_{0}(G;X,Y)$ naturally defines a class in each of the bivariant theories in (\ref{eqn:manymaps}). We will write $[\mcal{M}]$ as well for the class induced in any one of these bivariant theories. We continue to use the notation $\mcal{K}^{?}_{G}$ for any one of the bivariant theories appearing in (\ref{eqn:manymaps}) and write $K^{G,\,?}_{*}$ for its homotopy groups.

\begin{definition}\label{transfer} 
Let $f: Y \to X$ be a finite morphism of smooth
projective complex $G$-varieties, and let
$[\Gamma_f^{t}] \in K^{G,\,alg}_0(X,Y)$ be the element
represented by the transpose of the graph of $f$.
We define transfer maps
$$
f_*:K^{G,\,?}_*(Y,Z) \to K^{G,\,?}_*(X,Z)
$$
and 
$$
f^*:K^{G,\,?}_*(Z,X) \to K^{G,\,?}_*(Z,Y)
$$
by $f_*(a):=\theta([\Gamma_f^{t}],a)$ and 
$f^*(b):=\theta(b,[\Gamma_f^{t}])$.
\end{definition}
Note that the notation $f^*$ is used both for the transfer map in the second variable as well as the 
usual contravariance in the first variable (and a similar overlap for the meaning  of $f_{*}$).  We adopt this notation to conform to \cite{Walker:Thomason}.

The compatibility of the pairings 
$\theta$ with the natural transformations (\ref{eqn:manymaps})
between 
the various equivariant $K$-theories implies that these natural transformations are also compatible
with the transfer maps.

The cup product pairing (\ref{eqn:cup}) gives $\bu^{G}_{*}(W,S^{0})$ the structure of a graded commutative ring. We now relate this ring to $KU^{-*}_{G}(W)$, the periodic equivariant complex $K$-theory. For details on equivariant $K$-theory we refer the reader to \cite[chapter XIV]{May:equihomotopy}
or \cite{Segal}. Recall that $KU_{G}^{0}(S^{0}) = R(G)$, the complex representation ring. For real representations $\alpha$, $\beta\in \Rep_{\R}(G)$, the tensor product of bundles defines a product
\begin{multline*}
 KU_{G}^{-\alpha}(W)\otimes KU_{G}^{-\gamma}(W) = KU_{G}^{0}(S^{\alpha}\wedge W)\otimes KU_{G}^{0}(S^{\gamma}\wedge W) \\
\to KU_{G}^{0}(S^{\alpha}\wedge W \wedge S^{\gamma}\wedge W) \to KU_{G}^{0}(S^{\alpha}\wedge S^{\gamma}\wedge W) = KU_{G}^{-\alpha-\gamma}(W),
\end{multline*}
 making $\oplus_{\alpha\in \Rep_{\R}(G)}KU_{G}^{-\alpha}(W)$ into a ring.

\begin{proposition}\label{prop:conn}
 For any based, compact $G$-$CW$-complex $W$, there is a natural isomorphism of graded rings
$$
\oplus_{\alpha\in \Rep_{\R}(G)} \bu^{G}_{\alpha}(W,S^{0}) \xrightarrow{\iso} \oplus_{\alpha\in \Rep_{\R}(G)}KU_{G}^{-\alpha}(W).
$$
\end{proposition}
\begin{proof}
 The argument is similar to that of \cite[Proposition 6.18]{Walker:Thomason}. Briefly, we have that $\bu^{G}_{0}(W,S^{0}) = KU^{0}_{G}(W)$ since by Corollary \ref{cor:cc} both groups are the quotient of $KU^{0}(W_{+})$ by the subgroup $KU^{0}(S^{0})$. From Corollary \ref{cor:3.17} it follows that $\bu^{G}_{\alpha}(W,S^{0}) = \bu_{0}^{G}(S^{\alpha}\wedge W, S^{0})$ therefore it suffices to show that the diagram
$$
\xymatrix{
\bu_{\alpha}^{G}(W,S^{0}) \otimes \bu_{\gamma}^{G}(W,S^{0}) \ar[r]\ar[d] & \bu_{\alpha+\gamma}(W,S^{0}) \ar[d] \\
\bu_{0}^{G}(S^{\alpha}, S^{0}) \otimes \bu_{0}^{G}(S^{\gamma},S^{0}) \ar[r] & \bu_{0}^{G}(S^{\alpha}\wedge S^{\gamma}\wedge W,S^{0})
}
$$
commutes. This is easily seen to hold by definition of the cup product.
\end{proof}

For any complex representation $V$ there is a \textit{Bott element} $\beta_{V}\in KU_{G}^{0}(S^{V})$ such that for any $X$, multiplication by $\beta_{V}$ is an isomorphism
$$
-\cup\beta_{V}:KU_{G}^{0}(X) \xrightarrow{\iso} KU_{G}^{0}(X\wedge S^{V}).
$$
In particular there is a Bott element $\beta_{2}\in KU^{-2}_{G}(S^{0}) = \bu^{G}_{2}(S^{0},S^{0})$ corresponding to the trivial one-dimensional complex representation. 
We refer the reader to \cite[Section XIV.3]{May:equihomotopy}
and of course to \cite{Segal} for details on equivariant Bott periodicity.

\begin{corollary}\label{cor:bottinv}
 Let $W$ be a based, compact $G$-$CW$-complex. There is a natural map of graded rings
$$
\oplus_{\alpha\in RO(G)}\bu^{G}_{\alpha}(W,S^{0}) \to \oplus_{\alpha\in RO(G)}KU_{G}^{-\alpha}(W).
$$
Let $\Lambda$ be a complete set of irreducible complex representations. Inverting the Bott elements corresponding to $\Lambda$ yields an isomorphism of $RO(G)$-graded rings
$$
\oplus_{\alpha\in \Rep_{\R}(G)}\bu^{G}_{\alpha}(W,S^{0})[\beta_{V}^{-1},\,V\in\Lambda] \xrightarrow{\iso} \oplus_{\alpha\in RO(G)}KU_{G}^{-\alpha}(W).
$$
\end{corollary}
\begin{proof}
For any representation $\alpha\in \Rep_{\R}(G)$ and any complex representation $V$ we have isomorphisms 
$$
-\cup \beta_{V} :\bu_{\alpha}^{G}(W, S^{0}) \xrightarrow{\iso} \bu_{\alpha+V}(W, S^{0}).
$$
If $\gamma$ is a real representation then $\gamma\oplus \gamma$ can be given the structure of a complex representation. For real representations $\alpha$, $\gamma \in \Rep(G)$ the composition 
$$
\bu^{G}_{\alpha-\gamma}(W,S^{0})  \xrightarrow{\beta_{2\gamma}} \bu^{G}_{\alpha+\gamma}(W,S^{0}) 
\xrightarrow{}KU_{G}^{-\alpha-\gamma}(W)
\xrightarrow{\beta_{2\gamma}^{-1}} KU_{G}^{-\alpha + \gamma}(W)
$$
defines the desired map $\oplus_{\alpha\in RO(G)}\bu^{G}_{\alpha}(W,S^{0}) \to \oplus_{\alpha\in RO(G)}KU_{G}^{-\alpha}(W)$. The second statement is immediate.
\end{proof}

\section{Comparing semi-topological and topological equivariant
$K$-theory} \label{sec:sstThom}

The main result of this section is Theorem \ref{mainthmbottinv} below, where we show that Bott-inverted equivariant semi-topological $K$-theory and equivariant topological $K$-theory agree for projective $G$-varieties. In the next section, we will see that this yields a new proof of the equivariant version of Thomason's theorem. Similar to \cite[Theorem 7.11]{Walker:Thomason}, Theorem \ref{mainthmbottinv}
 follows by combining three ingredients: Theorem \ref{mainthm},
the compatibility of operations established 
in the previous section,  
and Theorem \ref{walker7.10equi} below comparing
the action of certain operations with multiplication by the Bott element.
The most significant difference is that unlike in the nonequivariant case $ku^{-*}_{G}(-)$ need not satisfy Poincare duality. Consequently, we have to modify several arguments.

We will write $ku^{*}_{G}(-) = \bu^{G}_{-*}(-,S^{0})$ and $ku_{*}^{G}(-)=\bu^{G}_{*}(S^{0},-)$ in this section. By Corollary \ref{cor:3.17} these agree with the cohomology and homology theories associated to the spectrum $\bu_{G}(S^{0},S^{0})$. From Proposition \ref{prop:cup} and Corollary \ref{cor:bottinv} we have natural maps of graded rings
\begin{equation}\label{eqn:maps}
K^{G,\,sst}_{*}(X,\C)\to ku_{G}^{-*}(X^{an}) \to KU_{G}^{-*}(X^{an}).
\end{equation}
In this and the next section $\ast$ will always denote ${\bf Z}$-grading.

\begin{definition}
Write $\beta_{2}\in K^{G,\,sst}_{2}(\C,\C)$ for the element corresponding to the Bott element $\beta_{2}\in ku^{-2}_{G}(S^{0})$ (see the discussion preceding Corollary \ref{cor:bottinv}) under the isomorphism  $K^{G,\,sst}_{2}(\C,\C)\iso ku^{-2}_{G}(S^{0})$ obtained from
Theorem \ref{mainthm}. The element $\beta_{2}\in K^{G,\,sst}_{2}(\C,\C)$ is referred to as the \textit{semi-topological Bott element}.
\end{definition}

By Corollary \ref{cor:bottinv} the right map of (\ref{eqn:maps}) induces
an isomorphism of graded rings $ku_{G}^{-*}(X^{an})[\beta^{-1}_{2}] \stackrel{\simeq}{\to} KU_{G}^{-*}(X^{an})$. In Theorem \ref{mainthmbottinv} below we show that 
that when $X$ is a smooth and projective complex $G$-variety, the maps (\ref{eqn:maps}) induce isomorphisms
$$
K^{G,sst}_{*}(X,\C)[\beta^{-1}_{2}] \stackrel{\cong}{\to} 
ku_{G}^{-*}(X^{an})[\beta^{-1}_{2}]\xrightarrow{\iso} KU_{G}^{-*}(X^{an}).
$$

\begin{definition}
We define $\delta_X \in K_0^G(X \times X)$ to be the class of the coherent $G$-module $\mcal{O}_{\Delta}$. 
 We also write  $\delta_X \in K_0^{G,sst}(X \times X)$
for its image in semi-topological $K$-theory, and
$\delta_{X^{an}}$ for its image in either  $ku^{0}_G(X^{an} \times X^{an})$ or 
$KU^{0}_G(X^{an} \times X^{an})$.
\end{definition}

Note that  $\delta_{X}=\Delta_{*}(1)$, where $\Delta_{*}$ is the transfer map, defined in Definition \ref{transfer}. The remaining ingredient for the proof of Theorem \ref{mainthmbottinv} is the following.

\begin{theorem}[c.f. {\cite[Theorem 7.10]{Walker:Thomason}}]\label{walker7.10equi}
Let $X$ be a smooth complex projective $G$-variety of dimension $d$. There are classes $[X]\in KU^{G,\, sst}_{2d}(\C,X)$ and $\delta_{X}\in K_{0}^{G,\,sst}(X\times X,\C)$ such that the composition
$$
K^{G,sst}_{*}(X,\C) \xrightarrow{-\cap[X]} K^{G,sst}_{*+2d}(\C,X) \xrightarrow{\delta_{X}/-} K^{G,sst}_{*+2d}(X,\C) 
$$
coincides with multiplication by $\beta_2^{d}\cup u$ for some unit $u\in K_{0}^{G,\,sst}(X,\C)$. Similarly there are classes $[X^{an}]\in ku^{G}_{2d}(X^{an})$ and $\delta_{X^{an}}\in ku^{0}_{G}(X^{an})$ such that 
$$
ku_{G}^{*}(X^{an}) \xrightarrow{-\cap[X^{an}]} ku^{G}_{2d-*}(X^{an}) \xrightarrow{\delta_{X^{an}}/-}  ku^{G}_{*-2d}(X^{an}) 
$$ 
coincides with multiplication by $\beta^{d}\cup v$ for some unit $v\in ku_{G}^{0}(X^{an})$.
\end{theorem}

The proof of this theorem will occupy the remainder of this section but first, we prove the main result of this section.

\begin{theorem}\label{mainthmbottinv}
Let $G$ be a finite group, and let $X$ be a smooth complex projective
$G$-variety of dimension $d$. Then the map of (\ref{eqn:maps})
induces an isomorphism 
$$
K^{G,sst}_{*}(X,\C)[\beta^{-1}_{2}] \stackrel{\cong}{\to} 
ku_{G}^{-*}(X^{an})[\beta^{-1}_{2}] \stackrel{\cong}{\to}
KU_{G}^{-*}(X^{an}).
$$
\end{theorem}
\begin{proof}
That the right map is isomorphism is the second part of
Corollary \ref{cor:bottinv}. The argument that the left map is an isomorphism is the same as  \cite[Theorem 7.11]{Walker:Thomason}, in the nonequivariant case. Namely, we consider the diagram
$$
\xymatrix{
K^{G,sst}_{*}(X,\C) \ar[r]^{-\cap[X]}\ar[d] & K^{G,sst}_{*+2d}(\C,X) \ar[r]^{\delta_{X}/-}\ar[d]^{\iso} & K^{G,sst}_{*+2d}(X,\C) \ar[d] \\
ku_{G}^{-*}(X^{an}) \ar[r]^{-\cap[X^{an}]} & ku^{G}_{*+2d}(X^{an}) \ar[r]^{\delta_{X^{an}}/-} & ku^{G}_{-*-2d}(X^{an}) 
}
$$
which commutes by the compatibility of operations established in the previous section. 
Using properties of the operations, one can see that the horizontal maps are multiplication by $\delta_{X}/[X]$. Moreover, by Theorem \ref{walker7.10equi}
they are both multiplication by $\beta_2^{d}$, up to a unit in $K^{G,sst}_{0}(X,\C)$ (resp.~ in $ku_{G}^{0}(X^{an})$). The result follows easily by a simple diagram 
chase.
\end{proof}

The remainder of this section is devoted to the proof of Theorem \ref{walker7.10equi}. After some important modifications, its proof is similar to the nonequivariant case and we focus our attention on the necessary modifications.
First, we recall some facts about equivariant complex orientation
and Poincar\'e duality. These are significantly more complicated in the equivariant setting, but we can simplify things by restricting our attention to those theories which are complex stable.
 See \cite[Chapter XVI.9]{May:equihomotopy} and \cite[Chapter III.6]{LLM} for a general and comprehensive treatment of these topics. A useful summary of Poincar\'e duality for complex stable theories may be found in \cite{GreenleesWilliams}.
Recall that an equivariant cohomology theory $E^{*}_{G}(-)$ is said to be \textit{complex stable} if for each complex representation $V$, there is a class $\sigma_{V}\in \widetilde{E}_{G}^{|V|}(S^{V})$ which gives isomorphisms 
$$
\widetilde{E}^{*}_{G}(S^{|V|}\wedge X) \xrightarrow{\iso} \widetilde{E}^{*}_{G}(S^{V}\wedge X),
$$
for any $G$-space $X$. As equivariant complex topological $K$-theory  satisfies Bott periodicity, it is complex stable.

Let $E$ be a commutative ring $G$-spectrum representing a complex stable cohomology theory and $M$ a smooth $G$-manifold. For any $x\in M$ the slice theorem 
implies that there are isomorphisms $E_{*}^{G}(M,M-G\{x\}) \iso \widetilde{E}_{*}^{G}(G_{+}\wedge_{G_{x}}S^{V_{x}})$ and $E^{*}_{G}(M,M-G\{x\}) \iso \widetilde{E}^{*}_{G}(G_{+}\wedge_{G_{x}}S^{V_{x}})$  where $G_{x}\subseteq G$ is the isotropy subgroup of $x$ and $V_{x}$ is the tangent space to $M$ at $x$ (see e.g. \cite[Lemma 3.1]{GreenleesWilliams}). 

\begin{definition}
Let $E$ be as above and $M$ a smooth $G$-manifold of dimension $n$. For an $x\in M$ let $\phi_{G\{x\}}$ denote the composition
$$
E_{*}^{G}(M)\xrightarrow{} E_{*}^{G}(M,M-G\{x\}) \iso \widetilde{E}_{*}^{G}(G_{+}\wedge_{G_{x}} S^{V_{x}}) \iso \widetilde{E}^{G_{x}}_{*}(S^{V_{x}}),
$$ 
where the first map is induced by the inclusion of pairs $(M,\emptyset)\subseteq (M,M-G\{x\})$, the second is the isomorphism from the paragraph above and the third is the change of groups isomorphism. 
An element $[M]\in \widetilde{E}^{G}_{n}(M)$ is called a \textit{fundamental class} for $M$ if $\phi_{G\{x\}}([M])$ is an  $\widetilde{E}^{G_{x}}_{*}$-module generator of $\widetilde{E}^{G_{x}}_{*}(S^{V_{x}})$ for all $x\in M$.
\end{definition}

\begin{definition}\label{Thomclass}
Let $E$ be as above and $M$ a smooth compact complex $G$-manifold of complex dimension $d$. For any orbit $i_{G\{x\}}:G\{x\}\to M$ let $\psi_{G\{x\}}$ denote the composition 
\begin{multline*}
E^{*}_G(M \times M, M \times M - \Delta) \to 
E^{*}_G(M \times G\{x\},M \times G\{x\} - \Delta(G\{x\}) ) \\
\cong E^{*}_{G}(G_{+}\wedge_{G_{x}}S^{V_{x}})\iso E^{*}_{G_{x}}(S^{V_{x}})
\end{multline*}
where the first map is obtained from the map of pairs induced by $id\times i_{G\{x\}}$.
An element $t_M \in E^{2d}_G(M \times M, M \times M - \Delta)$ is called
a {\it Thom class} for $M$ if $\psi_{G\{x\}}(t_{M})$ is an $E^*_{G_{x}}$-module generator of $E^*_{G_{x}}(S^{V_{x}})$ for all $x\in M$.
\end{definition}

\begin{lemma}
Let $M$ be a smooth complex compact $G$-manifold of complex dimension $d$ and $E$ a commutative ring $G$-spectrum representing a complex stable cohomology theory. There is a bijection between $E$-Thom classes for $M$ and $E$-fundamental classes for $M$.
\end{lemma}
\begin{proof}
 This is \cite[Proposition III.6.7]{LLM}. One needs to observe that the definitions used there agree with the ones used here, as one can see using
Remark \ref{allthomagree} below.
\end{proof}

As a result of  Bott periodicity and the Thom isomorphism for $KU_{G}^{0}$,  any smooth complex compact $G$-manifold $M$ of complex dimension $d$ has a Thom class in $KU^{2d}_{G}(M)$ in the sense above and therefore it has a fundamental class.
The map 
$$
\cap[M]:KU^*_G(M) \stackrel{\cong}{\to}
KU^G_{2d-*}(M)
$$ 
is an isomorphism by \cite[III.6.4]{LLM} (or see \cite[Theorem 3.6]{GreenleesWilliams}). This is the \textit{equivariant Poincar\'e duality} isomorphism.

We know that $ku_{G}^{*}(-)$ is in general not complex stable (see e.g. \cite[Section 4]{Greenlees:equivforms}). Therefore,  unlike in \cite{Walker:Thomason},  we work with periodic rather than with connective equivariant $K$-theory from now on. This is not a problem by the following lemma.

\begin{lemma}\label{kuisalmostKU}
 Let $W$ be a finite dimensional $G$-$CW$ complex. Then
$$
ku^{G}_{i}(W) \xrightarrow{\iso} KU^{G}_{i}(W)
$$
is an isomorphism for $i\geq \dim(W)$. 
\end{lemma}
\begin{proof}
Let $W^{(n)}$ denote the $n$-skeleton of $W$. For each $n$, $W^{(n)}/W^{(n-1)}$ is wedge of  spheres of the form $S^{n}\wedge G/H_{+}$. We show that $ku^{G}_{i}(W^{(n)}) \xrightarrow{\iso} KU^{G}_{i}(W^{(n)})$ is an isomorphism for $i\geq n$ and an injection of $i = n-1$. The map of $G$-spectra $ku^{G} \to KU^{G}$ induces a comparison of long-exact sequences
$$
\xymatrix{
\cdots \ar[r] & ku^{G}_{i}(X^{(n)}) \ar[r]\ar[d] & ku^{G}_{i}(X^{(n+1)}) \ar[r]\ar[d] & ku^{G}_{i}(X^{(n+1)}/X^{(n)}) \ar[d]\ar[r] & \cdots \\
\cdots \ar[r] & KU^{G}_{i}(X^{(n)}) \ar[r] & KU^{G}_{i}(X^{(n+1)}) \ar[r] & KU^{G}_{i}(X^{(n+1)}/X^{(n)}) \ar[r] & \cdots .
}
$$

The right-hand map is a sum of maps of the form 
$$
ku^{G}_{i}(S^{n+1}\wedge G/H_{+}) \to  KU^{G}_{i}(S^{n+1}\wedge G/H_{+}).
$$
 Via the change of groups isomorphism it is identified with 
$$
ku^{H}_{i-n-1}(S^{0}) \to KU^{H}_{i-n-1}(S^{0})
$$ 
which is an isomorphism for $i\geq n+1$ and $ku^{H}_{i-n-1}(S^{0})=0$ otherwise. 
\end{proof}

The lemma allows us to lift fundamental classes to the equivariant homology theory $ku^{G}_{*}$ and thus to the semi-topological equivariant $K$-homology as well.

\begin{definition}
Let $X$ be a smooth projective complex $G$-variety of complex dimension $d$ and $[X^{an}]\in KU^G_{2d}(X^{an})$ a fundamental class. Define classes  $[X^{an}]\in ku_{2d}^{G}(X^{an})$ 
and  $[X]\in K^{G,\,sst}_{2d}(\C,X)$ to be lifts of $[X^{an}]\in KU_{2d}(X^{an})$ under the isomorphisms $K^{G,\,sst}_{2d}(\C,X)\iso ku_{2d}^{G}(X^{an})\iso KU^{G}_{2d}(X)$ provided by Theorem \ref{mainthm} and Lemma \ref{kuisalmostKU}.
\end{definition}

\begin{remark}\label{allthomagree}
There are several equivalent descriptions of the Thom space
and thus of the cohomology groups in which 
Thom classes live. Let $p:X \to B$ be a complex vector bundle
of rank $n$ with zero section $s$ and let $E^*$ a complex orientable cohomology
theory. The Thom space $Th(p)$ is homotopy equivalent to the homotopy cofiber of $X-s(B) \to X$.
If $B$ is compact, then the one point compactification $X^+$
is homeomorphic to $Th(p)$ and 
the Thom isomorphism for topological $K$-theory  is often stated using $\widetilde{KU}{}^{*}(X^+)$. In \cite{Walker:Thomason}, Thom classes for $M$ live in the cohomology group
 $E^*_M(M \times M)$,  defined using the homotopy cofiber of $M\times M-\Delta \to M\times M$. This compares to the Thom space
of its tangent bundle as the normal bundle of $\Delta:M \to M \times M$ is isomorphic
to the tangent bundle of $M$ (see e.g. \cite[Lemma 11.5]{MilnorStasheff}). 
All of these  weak equivalences remain valid in the equivariant case, as well.
\end{remark}

We can to establish an equivariant generalization 
of \cite[Lemma 7.8]{Walker:Thomason}, with $ku$ replaced by $KU$.

\begin{lemma}\label{deltathom}
Let $X$ be a smooth projective complex $G$-variety of dimension
$d$. Let 
$j^*:KU^{2d}_G(X^{an} \times X^{an}, X^{an} \times X^{an} - \Delta)
\to KU^{2d}_G(X^{an} \times X^{an})$
be the map forgetting the support.
Then $\delta_{X^{an}}=(\beta_2)^d \cup j^*(t)$ for some
Thom class $t \in KU^{2d}_G(X^{an} \times X^{an}, X^{an} 
\times X^{an} - \Delta)$.
\end{lemma}
\begin{proof}
We roughly follow Walker's proof.  
We claim that the element $\delta_{X^{an}}$ lifts to an element $\tilde{\delta}_{X^{an}}\in KU^{0}_G(X^{an} \times X^{an}, X^{an} \times X^{an} - \Delta)$ such that 
the restriction of $\tilde{\delta}_{X^{an}}$   to
$KU^{0}_G(X^{an}\times G\{x\}, X^{an} - \Delta G\{x\}) =KU^{0}_{G,\Delta G\{x\}}(X\times G\{x\})$
is a generator for any orbit $i_{G\{x\}}: G\{x\} \to X$.
Using Bott periodicity, we therefore have that the element $\beta_{2}^{-d}\cup \tilde{\delta}_{X^{an}}\in KU_{G,\Delta}^{2d}(X^{an}\times X^{an})$ is a Thom class in the sense of
Definition \ref{Thomclass}, from which the result follows.

To begin we consider algebraic $K$-theory. Consider the following  diagram
$$
\xymatrix{
K_0^{G}(X) \ar[r]^-{\Delta_*} \ar[d]_{(i_x)^*} & K_0^{G}(X \times X) 
\ar[r] \ar[d]^{(id_X\times i_{G\{x\}})^*} & 
K_0^{G}(X \times X - \Delta) \ar[d]^{(id_X\times i_{G\{x\}})^*} \\
K_0^{G}(G \{x\}) \ar[r]^-{\Delta'_*}  & 
K_0^{G}(X\times G\{x\}) \ar[r]  & K^{G}_0(X\times G\{x\}-\Delta G\{x\}).
}
$$
The right-hand square is evidently commutative. The commutativity of the left-hand square follows from the commutativity of  
$$
\xymatrix{
K_0^{G_{x}}(X) \ar[r]^-{\Delta_*} \ar[d]_{(i_x)^*} & K_0^{G_{x}}(X \times X) 
 \ar[d]^{(id_X\times i_{x})^*}  \\
K_0^{G_{x}}(\{x\}) \ar[r]^-{(i_{x})_*}  & 
K_0^{G_{x}}(X ) ,
}
$$
which can be seen by using that $\pi:X\to \{x\}$ gives a $G_{x}$-equivariant section  of $i_{x}$.

The rows of the diagram are exact by
\cite[Corollary 5.8]{Thomason:algKgroup} (that is the equivariant
generalization of Quillen's resolution theorem) and by 
\cite[Theorem 2.7]{Thomason:algKgroup}.
The left hand side square induces a commutative square
$$
\xymatrix{
K_0^{G}(X) \ar[r]^-{\Delta_*}_-{\iso} \ar[d]_{(i_x)^*} 
& K_{0,\Delta}^{G}(X \times X) 
\ar[d]^{(id_X\times i_{G\{x\}})^*} \\
K_0^{G}(G \{x\}) \ar[r]^-{\Delta'_*}_-{\iso} & K_{0,\Delta G \{x\}}^{G}(X\times G\{x\}).
}
$$

Write $\tilde{\delta}_X=\Delta_*(1) \in 
K_{0,\Delta}^{G}(X \times X)$, which evidently maps to the element
$\delta_X=\Delta_*(1) \in K_{0}^{G}(X \times X)$. Moreover,
we have $(i_x)^*(1)=1 \in K_0^{G}(G \{x\})$
as $(i_x)^*$ is a ring homomorphism. We conclude that
$(id_{X}\times i_{G\{x\}})^*(\tilde{\delta}_{X})=\Delta'_*(1)$.
Now consider the  natural transformation $\epsilon:K_{0}^{G}(-)\to KU_{G}^{0}(-)$ obtained as the composite
$$
 K_{0}^{G}(X)\to K_{0}^{G,\,sst}(X)\to KU_{G}^{0}(X^{an})
$$ 
of the natural transformations from Theorem \ref{thm:algeq} and Corollary \ref{cor:qtopeq}.
The transformation $\epsilon$ is compatible with pullbacks, and it is a ring homomorphism. It is also compatible with pushforwards in both theories (where the push-forward in $K_{0}^{G}$ is the classical
push-forward in (equivariant) algebraic $K$-theory and  in $KU_{G}^{0}$ it is the one of Definition
\ref{transfer}) as a direct consequence of Definition \ref{transfer}. We therefore conclude  that we have $(id_{X^{an}}\times i_{G\{x\}})^*(\tilde{\delta}_{X^{an}})=\Delta'_*(1)$ in $KU^{0}_{G,\Delta G \{x\}}(X\times G\{x\})$.
Now since $KU^{0}_{G,\Delta G \{x\}}(X\times G\{x\})$ is a free $KU^{G}_{G}(G\{x\})$-module of rank 1 and $\Delta'_*$ is a $KU^{G}_{G}(G\{x\})$-module homomorphism we conclude that $(id_{X^{an}}\times i_{G\{x\}})^*(\tilde{\delta}_{X^{an}})$ generates $KU^{0}_{G,\Delta G \{x\}}(X\times G\{x\})$ as required.
\end{proof}

\textit{Proof of Theorem \ref{walker7.10equi}.}
The proof proceeds as in \cite[Theorem 7.10]{Walker:Thomason}. Namely the argument there shows that it suffices to show that $\delta_{X}/[X] = \beta^{d}\cup u$ and $\delta_{X^{an}}/[X^{an}] = \beta^{d}\cup v$ for units $u$ and $v$. 
The second equation follows from Lemma \ref{deltathom}. Using Lemma \ref{deltathom} and the results about equivariant pairings from
the previous section \cite[Proposition 7.9]{Walker:Thomason} 
generalizes to the equivariant setting, showing that $\pi^{*}(\delta_{Y}/a) =\delta_{X}/\pi^{*}(a)$ for a finite flat equivariant morphism $\pi:X\to Y$ of smooth projective $G$-varieties. It thus suffices to establish the result for $Y=\P^{d}_{\C}$, because we can always find a finite, surjective  (and hence flat) equivariant map $\pi:X\to \P^{d}_{\C}$. To see this, note that the quotient $X\to X/G$ is finite, surjective and equivariant and we may assume $X/G$ is connected so that Noether normalization yields a finite, surjective equivariant map $X/G\to \P^{d}_{\C}$.

It thus remains to see that $\delta_{\P^{d}}/[\P^{d}] = \beta^{d}\cup u'$, for a unit $u'\in K_{0}^{G,\,sst}(\P^{d}_{\C},\C)$. This follows from seeing that we have an isomorphism $K_{*}^{G,\,sst}(\P^{d}_{\C},\C) \iso \bu_{*
}^{G,\,\mf{c}}(\P^{d}_{\C},S^{0})$. To see this we can argue as in \cite[Proposition 2.7]{FW:real} to see that we have an equivariant equivalence $\mcal{K}_{G}(\Delta^{\bullet}_{top},\C)^{n}\wkeq \mcal{K}_{G}(\P^{1}_{\C}\times\Delta^{\bullet}_{top},\C)$ and a similar compatible equivalence for $\bu^{\mf{c}}_{G}$. The required isomorphism follows since  $\mcal{K}_{G}(\Delta^{\bullet}_{top},\C) \wkeq \bu^{\mf{c}}_{G}(S^{0},S^{0})$.    \qed

\section{Equivariant Thomason's theorem}\label{sec:algThom}

In this section, we explain how the work in the previous section gives an alternate proof of \cite[Theorem 5.9]{Thomason:famousequi}. This requires proving the expected comparison theorem
between algebraic and semi-topological equivariant
$K$-theory with finite coefficients. Working with the $\dtop$-construction rather than with the topological mapping spaces $Mor(-,-)$ makes this particularly straightforward.

Let $\mathcal{K}_G(-,-)$ be the bivariant
presheaf on quasi-projective complex $G$-varieties
with values in positive $\Omega$-$G$-spectra produced in Section \ref{sec:algK}.
By Proposition \ref{prop:fix}, we have that  
$$
\pi_{n}^{H}\mathcal{K}_G(X,Y)
\iso \pi_{n}\mathcal{K}(H;X\times\Delta^{\bullet}_{\C},Y)
$$ 
for any subgroup
$H < G$ and any $n$, where $\mathcal{K}(H;X,Y)$ is the 
$K$-theory spectrum of the category of coherent $H$-modules on $X\times Y$ which are finite and flat over $X$.
The following result, whose proof is as in \cite[Theorem 2.6]{FW:ratisos}, allows us to apply the work of the previous section to equivariant algebraic $K$-theory. Unlike Theorem \ref{mainthm},
the proof of the following theorem is rather formal
and applies to equivariant theories other than $\KK_{G}$,
provided they satisfy an appropriate equivariant rigidity theorem.

\begin{theorem}\label{KalgKsemi}
For a smooth quasi-projective complex $G$-variety $X$
and any integer $n>0$, the maps (\ref{eqn:algsst}) induce equivariant weak equivalences 
$$
\KK_G(X,\Z/n) \xrightarrow{\wkeq}\mcal{K}_{G}(X\times\Delta^{\bullet}_{top};\Z/n) \xleftarrow{\wkeq} \KK^{sst}_G(X,\Z/n)
$$
of $G$-spectra.
\end{theorem}
\begin{proof}
The right hand map is an equivariant weak equivalence by Lemma \ref{lem:delsst}. To show the left hand map is an equivariant weak equivalence, we must show that 
for all subgroups $H \subseteq G$, the map
$\pi_*^{H}\KK_{G}(X,\Z/n)\to \pi_{*}^{H}\mcal{K}_{G}(X\times\Delta^{\bullet}_{top},\Z/n)$
is an isomorphism. Using Proposition \ref{prop:fix}, 
the proof follows along the lines of the argument given in
\cite[Theorem 2.8]{FW:ratisos}. That is, we consider the map of presheaves
$\FF:=\KK(H;X,\Z/n) 
\to \GG:=\KK(H;X \times - ,\Z/n)$
where the first presheaf is globally constant.
By \cite{YO:equirigid} the map 
$$
\pi_{n}\KK(H;X,\Z/n) 
\xrightarrow{\iso} \pi_{n}\KK(H;X \times \mcal{O}_{T,t}^{h} ,\Z/n)
$$ 
is an isomorphism, 
where $T$ is any smooth variety, $t\in T(\C)$, and $\mcal{O}_{T,t}^{h}$ is the corresponding Henselian local ring. This rigidity isomorphism allows
us to conclude the result as in the proof of \cite[Theorem 2.8]{FW:ratisos}. That is, the map of  presheaves  $\pi_{n}\mcal{F}(-) \to \pi_{n}\mcal{G}(-)$ becomes an isomorphism, upon sheafification, of \'etale sheaves on $Sm/\C$. It therefore becomes an isomorphism upon further sheafification, of sheaves in the $uad$-topology on $Sch/\C$ because resolutions of singularities are $uad$-covers. Thus one may apply \cite[Theorem 2.6]{FW:ratisos} in order to conclude the result.
\end{proof}

Using the preceding theorem, we may lift the element
$\beta_2\in K_{2}^{G,\,sst}(\C,\Z/n)$ to an element $\beta$ in algebraic $K$-theory with finite 
coefficients. Proposition \ref{prop:cup} implies that the isomorphism $K_{*}^{G}(X,\Z/n) \iso K_{*}^{G,\,sst}(X,\Z/n)$ from the previous theorem is a graded ring isomorphism and therefore 
$$
K_*^G(X,\Z/n)[\beta^{-1}] \iso K_*^{G,sst}(X,\Z/n)[\beta_{2}^{-1}].
$$ keeping in mind the usual warning (see e.g. \cite[A.6]{Thomason:famous})
concerning very small values of $n$.
Also note that by \cite[p. 503]{Thomason:famous},
the algebraic Bott element $\beta$ Thomason considers really is a 
lift of the topological one.
Combining this theorem with the one
of the previous section, we obtain a new proof
of \cite[Theorem 5.9]{Thomason:famousequi} for finite groups:

\begin{theorem}\label{thm:lastapp}
For any smooth projective complex $G$-variety $X$ and any integer $n>0$,
we have a natural isomorphism of graded rings
$$
K_*^G(X,\Z/n)[\beta^{-1}] \stackrel{\iso}{\to}KU_G^{-*}(X^{an},\Z/n).
$$
\end{theorem}
\begin{proof}
This follows immediately from
Theorem \ref{mainthmbottinv} and Theorem \ref{KalgKsemi}.
\end{proof}

%\bibliographystyle{amsalpha} %stick in the bibliography.
%\bibliography{ET} 

\providecommand{\bysame}{\leavevmode\hbox to3em{\hrulefill}\thinspace}
\providecommand{\MR}{\relax\ifhmode\unskip\space\fi MR }
% \MRhref is called by the amsart/book/proc definition of \MR.
\providecommand{\MRhref}[2]{%
  \href{http://www.ams.org/mathscinet-getitem?mr=#1}{#2}
}
\providecommand{\href}[2]{#2}

\end{document}